\documentclass{amsart}
\allowdisplaybreaks[4]
\usepackage{color}

\newcommand{\Z}{\mathbf{Z}}
\newcommand{\R}{\mathbf{R}}

\newcommand{{\ba}}{\bf a}
\newcommand{\ve}{\varepsilon}
\newcommand{\la}{\lambda}
\newcommand{\La}{\Lambda}
\newcommand{\ga}{\gamma}
\newcommand{\pa}{\partial}
\newcommand{\ra}{\rightarrow}

\newcommand{\del}{\delta}

\newcommand{\cd}{\cdot}

\newcommand{\al}{\alpha}

\newcommand{\be}{\begin{equation}}
\newcommand{\ee}{\end{equation}}

\newcommand{\nn}{\nonumber}

\newtheorem{lem}{Lemma}{\bf}{\it}
\newtheorem{rem}{Remark}{\it}{\rm}

\newtheorem{theorem}{Theorem}

\numberwithin{theorem}{section}
\numberwithin{lem}{section}
\numberwithin{equation}{section}
\numberwithin{proposition}{section}
\numberwithin{corollary}{section}
\title[Diffusive LSW Equation]{On a diffusive version of the Lifschitz-Slyozov-Wagner Equation }

\author{Joseph G. Conlon}

\address{University of Michigan\\ Department of Mathematics\\ Ann Arbor,
  MI 48109-1109}
\email{conlon@umich.edu}


\begin{document}

\maketitle

\begin{abstract}
This paper is concerned with the Becker-D\"{o}ring (BD) system of equations and their relationship to the Lifschitz-Slyozov-Wagner (LSW) equations. A diffusive version of the LSW equations is derived from the BD equations. Existence and uniqueness theorems for this diffusive LSW system are proved. The major part of the paper is taken up with proving that solutions of the diffusive LSW system converge in the zero diffusion limit to solutions of the classical LSW system. In particular, it is shown that the rate of coarsening for the diffusive system converges to the rate of coarsening  for the classical system.

\end{abstract}

\section{Introduction.}
In this paper we shall be concerned with the Becker-D\"{o}ring (BD) system of equations \cite{bd} and their relationship to the Lifschitz-Slyozov-Wagner (LSW) equations \cite{ls,w}. The BD equations describe the time evolution for a mean field model of a population of particle clusters. Particle clusters are characterized by their volume $\ell$ which may take integer values $\ell=1,2...$. The monomers of volume $\ell=1$ play a distinctive role since they mediate interactions between particle clusters of different volumes. In particular, a cluster with volume $\ell$ can become a cluster of volume $\ell+1$ by addition of a monomer, or become a cluster of volume $\ell-1$ by evaporation of a monomer.  Finally, conservation of mass imposes a constraint which is non-local in $\ell$. 

The BD equations are determined by the rates at which a cluster of volume $\ell$ becomes a cluster of volume $\ell+1$ or $\ell-1$.  If $c_1(t)$ is the monomer density at time $t$, then a cluster of volume $\ell$ becomes a cluster of volume $\ell+1$ at rate $a_\ell c_1(t)$,  and a cluster of volume $\ell$ evaporates a monomer at rate $b_\ell$ to become a cluster of volume $\ell-1$. The net effect of these two processes yields a flux $J_\ell$  of $\ell$-clusters to $\ell+1$-clusters which depends on the density $c_\ell(t)$ of $\ell$-clusters and $c_{\ell+1}(t)$ of $\ell+1$-clusters, as well as the monomer density $c_1(t)$. In the BD model it is given by the formula $J_\ell=a_\ell c_1c_\ell-b_{\ell+1}c_{\ell+1}$. The BD equations for the time evolution of the system are therefore given by
\begin{eqnarray}
\frac{dc_\ell}{dt}&=&J_{\ell-1}-J_\ell, \quad \ell\ge2, \label{A1}  \\
\sum_{\ell=1}^\infty \ell c_\ell(t) &=&\rho. \label{B1}
\end{eqnarray}

Global existence and uniqueness theorems for the BD system (\ref{A1}), (\ref{B1}) were proven in the seminal paper of Ball et al \cite{bcp}  under fairly mild assumptions on the rates $a_\ell, \ b_\ell, \ \ell=1,2...$, the main one being that $a_\ell$ should grow at a sub-linear rate as $\ell\ra \infty$.

It is easy to see by solving $J_\ell=0, \ \ell\ge 1$, that (\ref{A1}), (\ref{B1}) has a family of equilibrium solutions parametrized by the monomer density $c_1$.  In order to classify the equilibrium solutions one needs detailed knowledge of the asymptotic behavior of  the rates $a_\ell, \ b_\ell$ as $\ell\ra \infty$. In this paper we shall assume that $a_\ell, \ b_\ell$ are given by the formulae
\be \label{O1}
a_\ell=a_1\ell^{1/3}, \ \ a_1>0; \qquad b_\ell=a_\ell(z_s+q/\ell^{1/3}),  \ \ z_s, \ q>0.
\ee
These values for $a_\ell, \ b_\ell$ may be obtained by deriving the BD equations as mean field equations for a $3$ dimensional microscopic system \cite{p}.

With $a_\ell, \ b_\ell$  given by (\ref{O1}), one sees that for any $c_1$ satisfying $0<c_1\le z_s$ there is an equilibrium solution $c_\ell=Q_\ell c_1^\ell, \ \ell=1,2...$.  For large $\ell$ the coefficient $Q_\ell$ has the asymptotic behavior
\be \label {C1}
Q_\ell \sim z_s^{-(\ell-1)}\frac{1}{\ell^{1/3}}\exp\left[-\frac{3q}{2z_s}\ell^{2/3}\right], \quad  \ell \ra \infty.
\ee
It follows from (\ref{C1}) that there is an equilibrium solution  corresponding to $c_1=z_s$.  We shall denote its density  (\ref{B1})   by $\rho_{\rm crit}$.
It was shown in \cite{bcp} that if $\rho\le \rho_{\rm crit}$, then the solution of (\ref{A1}), (\ref{B1}) converges strongly at large time to the corresponding equilibrium solution in the sense that
$$ \lim_{t\ra \infty} \sum_{\ell=1}^\infty \ell |c_\ell(t)-Q_\ell c_1^\ell| = 0. $$
If $\rho> \rho_{\rm crit}$ then $c_\ell(t)$ converges weakly to the equilibrium solution with maximum density, 
\be \label{D1}
 \lim_{t\ra \infty} c_\ell(t)=Q_\ell z_s^\ell, \ \ell\ge 1.
\ee
Equation (\ref{D1}) shows that for $\rho> \rho_{\rm crit}$  particles from the excess density $\rho- \rho_{\rm crit}$ over equilibrium  concentrate at large time in clusters of increasingly large volume. This is the phenomenon of {\it coarsening}, and it is an important problem to understand the rate at which the volume of clusters from the excess density increase with time.

A mechanism for determining this was proposed by Penrose in \cite{p}, where he argued that the large time behavior of clusters from the excess density is governed by the LSW equations,
\begin{eqnarray}
\frac{\pa c(x,t)}{\pa t}&=& \frac{\pa}{\pa x}\left[\left\{1-\left(\frac{x}{L(t)}\right)^{1/3}\right\}c(x,t)\right],  \quad x>0, \label{E1}  \\
\int_0^\infty x c(x,t) dx&=& \rho- \rho_{\rm crit}. \label{F1}
\end{eqnarray}
The parameter $L(t)$ in (\ref{E1}),  which turns out to be a measure of the average cluster {\it volume}, is uniquely determined by the conservation law (\ref{F1}). A mathematically rigorous relationship between solutions of the BD system (\ref{A1}), (\ref{B1}) and LSW system  (\ref{E1}), (\ref{F1}) was established in  \cite{laurmisc}  for the case $z_s=0$  when $\rho_{\rm crit}=0$, and in \cite{n} for the case $z_s>0$  when $\rho_{\rm crit}>0$.  Much still remains to be done however, in order to understand this in the precise way envisioned by Penrose \cite{p}. 

In this paper we take a rather different approach from  \cite{laurmisc, n}  in attempting to establish a relationship between solutions of the BD and LSW systems. Following  \cite{v}, we first observe that the flux $J_\ell$ can be written as
$$ J_\ell= [b_\ell c_\ell- b_{\ell+1}c_{\ell+1}] -a_1q[1-\{\ell/L(t)\}^{1/3}]c_\ell,$$
where $L(t)$ is given by the formula
\be \label{G1}
 c_1(t)=z_s+q/L(t)^{1/3}.
 \ee
Observe from (\ref{D1}) that the function $L(t)$ in (\ref{G1}) satisfies $ L(t)\ra\infty$ as $t\ra\infty$. 
Denoting by $D$ the forward difference operator and $D^*$ its adjoint, then  (\ref{A1}) is given by
\be  \label{H1}
\frac{\pa c(\ell,t)}{\pa t}=-D^*D[b_\ell c(\ell,t)] -a_1qD^*[(1-\{\ell/L(t)\}^{1/3})c(\ell,t)], \quad \ell\ge 2,
\ee  
where $c(\ell,t)=c_\ell(t)$ in (\ref{A1}). Evidently if we drop the first term on the RHS of (\ref{H1}) we obtain a discrete version of (\ref{E1}). 

We still need to account for the discrepancy in the conservation laws (\ref{B1}) for the BD system and (\ref{F1}) for the LSW system. Now from (\ref{D1}) the bulk of the equilibrium density $\rho_{\rm crit}$ is concentrated in clusters of volume $O(1)$ as $t\ra \infty$. Further, by choosing $\ell^*$ large enough  we may make $c(\ell^*,t)$ arbitrarily small for all large t,  and the density concentrated in clusters with volume larger than $\ell^*$ arbitrarily close to $\rho-\rho_{\rm crit}$.   It may seem reasonable therefore, that in order to understand the large time behavior of the BD system, one is justified in setting
$c(\ell^*,t)=0, \ t \ge T$, for some large but fixed $\ell^*$. This assumption is evidently unphysical since it cuts off interaction between clusters of volume less than $\ell^*$ and clusters of volume greater than $\ell^*$. One consequence of it is that the resulting conservation of mass for clusters of size greater than $\ell^*$ implies  $ c_1(t)>z_s$ for $t>T$. There is so far no rigorous proof that solutions of the BD equations with  $\rho> \rho_{\rm crit}$ satisfy $ c_1(t)>z_s, \ t>T$, for sufficiently large $T$. A proof of it would seem to be necessary in order to rigorously justify the relationship between the BD and LSW systems. In this connection one should also note that in \cite{n} it is necessary to make an assumption -equation (1.33) of the paper-in order to prove that a scaling limit of the BD system is a solution of the LSW system.  Equation (1.33) is however not a boundary condition, but essentially an upper bound assumption on the rate of coarsening for the BD model. Some interesting heuristic explanations \cite{fn} have recently been given to describe the transition, from the relaxation to equilibrium phase of the BD time evolution, to the coarsening phase governed by the LSW equations. These may be helpful towards constructing a fully rigorous theory.

We shall assume now that there is a time $T$ such that for $t\ge T$, the interaction between  clusters of large volume and clusters of $O(1)$ volume is negligible, and even make the significantly stronger assumption that there exists $\ell^*\ge 1$, such that $c(\ell^*,t)=0, \ t \ge T$. We may therefore study the large time evolution of the cluster density for particles from the excess density by solving the initial value problem for (\ref{H1}), subject to conservation of mass for particles with volume larger than $\ell^*$. Without loss of generality we may normalize $\ell^*=1, \ T=0$, whence the problem becomes   solving (\ref{H1}) subject to  the Dirichlet boundary condition $c(1,t)=0, \  t\ge 0$, given initial condition $c_\ell(0)=\gamma_\ell, \ \ell \ge 1$, and conservation law
\be \label{I1}
\sum_{\ell=1}^\infty \ell c(\ell,t) =\rho- \rho_{\rm crit}.  
\ee
In $\S2$ we prove global existence and uniqueness of a solution to this problem (Theorem 2.2), and obtain some properties of it.

Next we consider a continuous version of (\ref{H1}). 
 Setting $x\sim \ell$ and replacing finite differences in (\ref{H1}) by derivatives, we obtain the equation 

\be \label{J1}
\frac{\pa c(x,t)}{\pa t}=\frac{\pa^2}{\pa x^2} \left[\ve(1+x/\ve)^{1/3}c(x,t)\right]+ \frac{\pa}{\pa x}\left[\left\{1-\left(\frac{x}{L(t)}\right)^{1/3}\right\}c(x,t)\right],  \quad x>0,
\ee
with the parameter $\ve=1$ in (\ref{J1}). In (\ref{J1}) we have replaced the coefficient $b_\ell=a_1(q+z_sl^{1/3})$ of (\ref{H1}) by the diffusion coefficient  $(1+x)^{1/3}$. This is simply for convenience, since the only properties of the diffusion coefficient in (\ref{J1}) which we will use are that it behaves like $x^{1/3}$ for large $x$, and is bounded away from 0 for small $x$.  An equation similar to (\ref{J1}) was obtained in \cite{v}.
In $\S 3$ we prove global existence and uniqueness of solutions to (\ref{J1}) subject to the constraint (\ref{F1}) and the Dirichlet boundary condition $c(0,t)=0, \  t>0$, provided the nonnegative initial data $c(x,0), \ x>0$, satisfies the integrability property,
\be \label{K1}
\int_0^\infty(1+ x) \  c(x,0) dx<\infty.
\ee
This is the content of Theorem 3.1.  In \cite{v} an analogous  {\it local} existence and uniqueness theorem is proved (Theorem 6 of \cite{v}), for which one must also assume the point-wise estimate $\sup_{x>0} c(x,0)/(1+x)^{2/3}<\infty$ on the initial data.

There are several existence and uniqueness theorems in the literature for the LSW system (\ref{E1}), (\ref{F1}) and its generalizations. Collet and Gordon \cite{cg} have proved existence and uniqueness  for a system with globally Lipschitz coefficients, under the condition that the initial data $c(\cdot,0)$ is integrable and has finite first moment. The corresponding result for the LSW system-in which the coefficients are not globally Lipschitz- was proved by Lauren\c{c}ot \cite{laur}. Niethammer and Pego \cite{np2} proved existence and uniqueness  for the LSW system under the condition that the initial data $c(\cdot,0)$ is a Borel measure with compact support. In \cite{np3} they were able to extend their method to the case when the initial data $c(\cdot,0)$ is only assumed to be a Borel measure with finite first moment.

In $\S4, \  \S5$ we shall show that the solution of the diffusive LSW system (\ref{F1}), (\ref{J1}) with given initial data, converges as $\ve\ra 0$ on any finite time interval to the solution of the classical LSW system with the same initial data. One hopes  that these results will shed some light on the problem of understanding the mechanism of coarsening for the  BD and other models of Ostwald ripening. It has been argued in the physics literature \cite{m} that in models  of Ostwald ripening a diffusive LSW equation occurs. Although the effect of the diffusion term vanishes at large time it acts as a selection principle, so that the large time behavior of the model is governed by the unique infinitely differentiable self-similar solution of the LSW model. 

The fact that the effect of the diffusion term is expected to vanish at large time is closely related to the dilation invariance property of  the LSW system. 
Thus if $c(x,t), \ L(t), \  t>0$, is a solution to (\ref{E1}), (\ref{F1}), then for any $\lambda>0$  one has that $c_\lambda(x,t), \ L_\lambda(t), \  t>0$, defined by $c_\lambda(x,t)=\lambda c(\lambda x,\lambda t), \ L_\lambda(t)=L(\lambda t)/\lambda, \  t>0$, is also a solution. It follows that one may solve the LSW system for large time iteratively by solving (\ref{E1}), (\ref{F1}) on the unit time interval $0\le t\le 1$ with $L(0)=1$, and then rescaling $L(1)$ to $1$.  Hence, assuming that in the iteration one always has for some constants $a, b$ satisfying $1<a\le b$ the inequality  $a\le L(1)\le b$,  the function  $L(t)$ increases {\it linearly} in $t$ as $t\ra \infty$. This is a {\it quantitative} statement concerning the phenomenon of coarsening. 

Consider now  how this iteration scheme would apply to the diffusive LSW system (\ref{F1}), (\ref{J1}). Denoting the solution to (\ref{J1}) by $c_\ve(x,t), \ L_\ve(t)$, then dilation using the parameter $\lambda$  results in a function $\lambda c_\ve(\lambda x,\lambda t)$ which is a solution to (\ref{F1}), (\ref{J1}) with $\ve$ in (\ref{J1}) replaced by $\ve/\lambda$. Thus in the iteration scheme described above $\ve\sim1/L(t) \ra 0$ for large time $t$, corresponding roughly to an equivalence of $\ve\sim 1, \ t\ra \infty$ and $\ve\ra 0, \ t\sim 1$. 

 We summarize our main results on convergence of  solutions of the diffusive LSW system to solutions of the classical LSW system.
In $\S4$ we prove that if  $c_\ve(x,t), \ L_\ve(t), \ t>0$, is the solution to (\ref{F1}), (\ref{J1}) with initial data satisfying (\ref{K1}) then
\begin{eqnarray} \label{L1}
\lim_{\ve \ra 0} \int^\infty_x  c_\ve(x',t) dx' &=& \ \ \ \int^\infty_x \; c_0(x',t)dx',  \quad x \ge 0, \\
\lim_{\ve \ra 0}  L_\ve(t) &=& L_0(t), \quad t\ge 0, \nn
\end{eqnarray}
where $c_0(x,t), \ L_0(t), \ t>0$, is the corresponding solution to the classical LSW system (\ref{E1}), (\ref{F1}). This is the content of Theorem 4.1. In $\S5$ we show that the point-wise in time coarsening rate for the diffusive LSW system converges as $\ve\ra 0$ to the corresponding rate for the LSW system. Specifically, let us define $\Lambda_\ve(t), \ t\ge 0$, by
\be \label{M1}
\Lambda_\ve(t)=\int_0^\infty xc_\ve(x,t)dx \Big/\int_0^\infty c_\ve(x,t) dx, \quad t\ge 0,
\ee
so $\Lambda_\ve(t)$ is the mean cluster volume at time $t\ge 0$. Then if the initial data satisfies (\ref{K1}), the function $ \Lambda_\ve(t), \ t>0$, is differentiable for $\ve\ge 0$ and
\be \label{N1}
\lim_{\ve\ra 0}  \frac{d\Lambda_\ve(t)}{dt}=  \frac{d\Lambda_0(t)}{dt}, \quad t\ge 0.
\ee
This is the content of Theorem 5.1.

In \cite{dp} an upper bound on the time averaged coarsening rate for the LSW system is obtained by application of a rather general argument of Kohn and Otto \cite{ko}. At the end of $\S3$ we extend the method of \cite{dp} to prove an upper bound on the time averaged coarsening rate for the diffusive LSW system (\ref{F1}), (\ref{J1}) with $0<\ve\le 1$, which is uniform as $\ve\ra 0$.  This is the content of Theorem 3.2. 
\section{Existence and Uniqueness for the Discrete Case}
We begin by giving a proof  of the existence and uniqueness of solutions to the Becker-D\"{o}ring system (\ref{A1}), (\ref{B1}) which is somewhat different than that given in \cite{bcp}.  We assume that the initial data $c_\ell(0),  \ \ell \ge 1$, satisfies
\be \label{A2}
c_\ell(0) = \ga_\ell, \ \ell \ge 1, \ \sum^\infty_{\ell = 1} \ell \ga_\ell = \rho > 0,
\ee
where the  $\ga_\ell, \ \ell \ge 1$, are all non-negative.  For any $\del > 0$ we define a space $X_\del$ by
\[	X_\del = \Big\{ c_1 :[0,\del] \ra \R : c_1(0) = \ga_1, \; c_1 \ {\rm \ is \ continuous \ and}  \ 0 \le c_1(t) \le \rho, \ 0 \le t \le \del \Big\}.	\]
Now $X_\del$ is a complete metric space under the uniform norm, $\| \cd \|_\del$.  We define a mapping $T : X_\del \ra X_\del$ which has the property that the solutions to the Becker-D\"{o}ring system give rise to fixed points of the mapping.  To do this we consider functions $v(\ell, t),  \ell \ge 1, \ t > 0$, and define $J(\ell, v, t)$ for $\ell \ge 1, \ t>0$, by
\be \label{B2}
J(\ell, v, t) = a_\ell c_1(t)  v(\ell,t) - b_{\ell + 1} \; v(\ell + 1, t),
\ee
where $c_1(t) , t>0$, is a given non-negative function.  Let $c(\ell, t), \ \ell \ge 2, \ t > 0$, be the solution to the linear system of equations,
\be \label{C2}
\frac{\pa c(\ell,t)}{\pa t} = J(\ell - 1, c, t) - J(\ell , c, t) , \ \ell \ge 2,
\ee
with initial and boundary conditions given by
\be \label{D2}
c(1,t) = c_1(t),  \ t > 0, \quad c(\ell, 0) = \ga_\ell, \ \ell \ge 2.
\ee
By the maximum principle there is a unique positive solution to (\ref{C2}), (\ref{D2}).  The mapping $T : X_\del \ra X_\del$ is defined as follows:  For $c_1 \in X_\del$ let $c(\ell, t), \ 0 \le t \le \del$, be the corresponding solution of (\ref{C2}), (\ref{D2}).
Then $Tc_1$ is given by the formula
\be \label{M2}
Tc_1(t) = \max \left\{ \rho - \sum^\infty_{\ell = 2} \ell c(\ell,t), 0 \right\}, \ \ 0 \le t \le \del.
\ee
Evidently $Tc_1 \in X_\del$.  We show that for $\del$ sufficiently small $T$ is a contraction.

\begin{lem}  There exists $\del(\rho) > 0$ depending only on $\rho$, such that for $0 < \del < \del(\rho)$ the mapping $T$ on $X_\del$ is a contraction.
\end{lem}
\begin{proof} We have from (\ref{C2}) that
\be \label{K2}
\sum^\infty_{\ell = 2} \ell c(\ell,t) = \sum^\infty_{\ell = 2} \ell\ga_\ell + \int^t_0 ds \ J(1,c,s) + \sum^\infty_{\ell = 1} J(\ell,c,s).
\ee
It follows that there is a constant $K > 0$, independent of $\rho, \del$, such that
\[   K \sum^\infty_{\ell = 2} \ell c(\ell,t) \le \sum^\infty_{\ell = 2} \ell\ga_\ell + \del[1+\rho] \|  \sum^\infty_{\ell = 1} 
\ell^{1/3}\; c(\ell, \cdot)\|_\del ,  \]
for $0 \le t \le \del$.  It follows from this that for $\del$ sufficiently small, depending only on $\rho$, there is the inequality,
\be \label{I2}
\|  \sum^\infty_{\ell = 1} \ell \; c(\ell, \cdot)\|_\del \le K_1\rho,
\ee
for some constant $K_1$ independent of $\rho$ and $\del$.

Next we consider $c_1, c'_1 \in X_\del$ and for $0 \le \la \le 1$, let $c_{\la, 1} = \la c_1 + (1-\la)c'_1$.  Define $c_\la(\ell,t), \ \ell \ge 2, \ 0 < t \le \del$, to be the solution of (\ref{C2}), (\ref{D2}) with $c_{\la, 1}$ substituted for $c_1$.  Putting $u_\la(\ell,t) = \pa c_\la(\ell,t)/\pa \la$ it follows that $u_\la$ satisfies the equation,
\be \label{E2}
\frac {\pa u_\la(\ell,t)}{\pa t} = J(\ell - 1, u_\la, t) - J(\ell, u_\la, t) + \big[ c_1(t) - c'_1(t)\big] \big[a_{\ell -1} c_\la(\ell - 1, t) - a_\ell c_\la(\ell, t) \big], \  \ell \ge 2, \ t>0,
\ee
with the initial and boundary conditions,
\be \label{F2}
u_\la(1,t) = c_1(t) - c'_1(t), \ t > 0,\ u_\la(\ell, 0) = 0, \ \ell \ge 2.
\ee
Let $v_\la(\ell, t),  \ell \ge 2, \ 0 \le t \le \del$, be the solution to the equation,
\be \label{G2}
\frac {\pa v_\la(\ell,t)}{\pa t} = J(\ell - 1, v_\la, t) - J(\ell, v_\la, t) + \| c_1 - c'_1\|_\del \big[a_{\ell -1} c_\la(\ell - 1, t) + a_\ell c_\la(\ell, t) \big], \ell \ge 2, 0 \le t \le \del,
\ee
with the initial and boundary conditions,
\be \label{H2}
v_\la(1,t) = \| c_1 - c'_1\|_\del, \ t > 0, \ \ v_\la(\ell,0) = 0,\ \ \ell \ge 2.
\ee
By the maximum principle $v_\la(\ell, t)$ is positive for  $\ell \ge 2 , t > 0$, and there is the inequality
\be \label{J2}
|u_\la(\ell, t)| \le v_\la(\ell, t), \ \ \ell \ge 2, \ \ 0 \le t \le \del.
\ee
From (\ref{G2}), (\ref{H2}) we have that
\begin{eqnarray*}
\sum^\infty_{\ell = 2} \ell^{1/3} \; v_\la(\ell, t) &=& \int^t_0 ds \Big\{ 2^{1/3} J(1, v_\la, s) + \sum^\infty_{\ell =2} \Big[(\ell +1)^{1/3} - \ell^{1/3}\Big] J(\ell, v_\la, s)\Big\} \\
&+& \| c_1 - c'_1\|_\del \int^t_0 ds \Big\{ 2^{1/3} a_1c_\la(1,s) + \sum^\infty_{\ell =2} \Big[(\ell +1)^{1/3} + \ell^{1/3}\Big] a_\ell c_\la (\ell, s)\Big\}.
\end{eqnarray*}
 Using (\ref{I2}) and applying Gronwall's inequality in the previous equation, we see 
that there is a constant $K_2$ independent of $\rho$ and $\del$ such that 
\be \label{L2}
\| \sum^\infty_{\ell = 2} \ell^{1/3} v_\la(\ell, \cd)\|_\del \le K_2 \del \rho\| c_1 - c'_1\|_\del,
\ee
provided $\del$ is taken sufficiently small, depending only on $\rho$.

To see that $T$ is a contraction for small $\del$ observe that 
\[  |Tc_1(t) - Tc'_1(t)| \le \Big| \sum^\infty_{\ell = 2} \ell \Big[ c_1(\ell, t) - c_0(\ell, t)\Big] \Big|  \ . \] 
Hence on using (\ref{K2}) we have that
\[	|Tc_1(t) - Tc'_1(t)|  \le K_3[1+\rho] \int^t_0 ds \Big\{ |c_1(s) - c'_1(s)| +  \sum^\infty_{\ell = 2} \ell^{1/3} | \int^1_0 u_\la(\ell, s)d\la | \Big\}, 0 \le t \le \del,  \]
for a constant $K_3$ independent of $\rho, \del$.  The result follows from the last inequality, (\ref{J2}), (\ref{L2}). 
\end{proof}
\begin{theorem} The Becker-D\"{o}ring system (\ref{A1}), (\ref{B1}) with initial condition satisfying (\ref{A2}) has a unique global in time solution.
\end{theorem}
\begin{proof}  It is evident that a solution of the Becker-D\"{o}ring system  (\ref{A1}), (\ref{B1}) with initial condition (\ref{A2}) gives rise to a fixed point of the mapping $T$ of (\ref{M2}).  Hence by Lemma 2.1 the solution is unique.  Correspondingly, suppose $c_1$ is a fixed point of $T$.  Similarly to (\ref{K2}) we have that 
\[    \sum^\infty_{\ell = 2} \ell \Big[ c(\ell, t) - c(\ell, t + \eta)\Big] = - \int^{t+\eta}_t ds \Big\{ J(1,c,s)  \]
\[	+ \sum^\infty_{\ell = 2} J(l,c,s) \Big\} \ge - K\; \rho \int^{t+\eta}_t c_1(s) ds,  \]
for $0 < \eta < \del/2, \ 0 < t < \del/2$, where $K > 0$ is a constant independent of $\del, \rho$.  Since $Tc_1 = c_1$ it follows from (\ref{M2}) that $c_1$ satisfies the inequality,
\[	c_1(t+\eta) - c_1(t) \ge -K \rho \int^{t+\eta}_t c_1(s)ds.	\]
Hence if $c_1(0) >0$ then $c_1(t) > 0$ for $0 < t < \del/2$.  Conversely if $c_1(0) = 0$ there is a $t_0 \ge 0$ such that $c_1(t) = 0$ for $t \le t_0$ and $c_1(t) > 0$ for $t_0 < t < \del/2$.  Suppose $t_0 > 0$.  Then (\ref{K2}) implies that
\[	\sum^\infty_{\ell = 2} \ell c(\ell, t) < \sum^\infty_{\ell = 2} \ell \ga_\ell, \ \ 0 < t \le t_0.  \]

Since $Tc_1 = c_1$ we conclude from (\ref{M2}) that $c_1(t) > 0$ for  $0 < t < t_0$, a contradiction. Thus $t_0=0$ whence $c_1(t) > 0$ for $0 < t < \del/2$ in all situations. It follows then from (\ref{M2}) that the conservation law (\ref{B1}) is satisfied.  Thus $c_1(t), c(\ell, t), \ell \ge 2$, are a solution of the Becker-D\"{o}ring system (\ref{A1}), (\ref{B1}) . 
\end{proof}
Next we establish existence and uniqueness for (\ref{H1}), (\ref{I1}) subject to the Dirichlet boundary condition $c(1,t)=0, \ t\ge 0$.   For a given non negative function $c_1(t), \ t \ge 0$, let $c(\ell, t), \ \ell \ge 2, \  t > 0,$ be the solution to (\ref{C2}) subject to the initial and boundary conditions,
\be \label{N2}
c(1,t) =0, \ t \ge 0, \quad  \ c(\ell,0) = \ga_\ell, \ \ell \ge 2,
\ee
where the $\ga_\ell$ satisfy (\ref{A2}) with $\ga_0 = 0, \ \rho = 1$.  Normalizing $\rho-\rho_{\rm crit}=1$ in (\ref{I1}), the function $c_1(t), t \ge 0$, is then to be chosen to maintain the conservation law
\be \label{O2}
 \sum^\infty_{\ell = 2} \ell c(\ell, t) = 1 .
\ee
Assuming we have a solution to (\ref{C2}), (\ref{N2}), (\ref{O2}) it follows that
\begin{eqnarray*}
0 &=& \frac \pa{\pa t} \sum^\infty_{\ell = 2} \ell c(\ell, t) = J(1,c,t) + \sum^\infty_{\ell = 1} J(\ell,c,t) \\
&=& c_1(t) \sum^\infty_{\ell = 2} a_\ell c(\ell, t) - \sum^\infty_{\ell = 2} b_\ell c(\ell, t) - b_2 c(2,t).
\end{eqnarray*}
We conclude therefore that $c_1(t)$ is given by the formula,  
\be \label{P2}
c_1(t) = z_s + \left\{ a_1q  \sum^\infty_{\ell = 2}  c(\ell, t) + b_2\; c(2,t) \right\} \; \Big/ \;  \sum^\infty_{\ell = 2} a_\ell c(\ell, t). 
\ee
Setting $t=0$ in (\ref{P2}) yields then the identity,
\be \label{Q2}
c_1(0) = z_s + \left\{ a_1q  \sum^\infty_{\ell = 2}  \ga_\ell + b_2\;\ga_2  \right\} \; \Big/ \;  \sum^\infty_{\ell = 2} a_\ell \ga_\ell \ .
\ee
Observe from (\ref{P2}) that $c_1$ satisfies the inequality,
\[	z_s \le c_1(t) \le 2(z_s + q/2^{1/3}), \ \ t > 0.	\]
Hence for any $\del > 0$ we define a space $X_\del$ by
\[
X_\del = \Big\{ c_1 : [0,\del] \ra \R : c_1(0) \ \ {\rm is \ given \ by} \ (\ref{Q2}), \ \  c_1 \ {\rm is \ continuous\ and} \]
\[	z_s \le c_1(t) \le 2(z_s + q/2^{1/3}), \ 0 \le t \le \del \Big\}.	\]
As previously we equip $X_\del$ with the uniform norm $\|\cd\|_\del$ and define a mapping $T$ on $X_\del$ by
\be \label{R2}
Tc_1(t) = z_s + \left\{ a_1q \sum^\infty_{\ell = 2} c(\ell,t) + b_2c(2,t)\right\} \Big/ \sum^\infty_{\ell = 2} a_\ell c(\ell,t),  \ 0 \le t \le \del, 
\ee 
where $c(\ell,t),  \ t > 0$, is the solution to (\ref{C2}) subject to the initial and boundary conditions (\ref{N2}).
\begin{lem}  There exists $\del_0 > 0$ such that for $0 < \del < \del_0$ the mapping $T$ on $X_\del$ is a contraction.
\end{lem}
\begin{proof} We proceed as in Lemma 2.1.  We consider $c_1, \; c'_1 \in X_\del$ and for $0 \le \la \le 1$ we put $c_{\la,1} = \la c_1 + (1-\la)c'_1$.  Define $c_\la(\ell, t), \ \ell \ge 2, \ 0 \le t \le \del$, to be the solution of (\ref{C2}), (\ref{N2}) with $c_{\la,1}$ substituted for $c_1$.  Putting $u_\la(\ell,t) = \pa c_\la(\ell, t) \; / \; \pa \la$ we see from (\ref{R2}) that
\be \label{S2}
Tc_1(t) - Tc'_1(t) = \int^1_0 \ d\la \bigg[ \left\{ \sum^\infty_{\ell = 2} a_\ell c_\la(\ell, t) \right\} \ 
 \left\{ a_1q \sum^\infty_{\ell = 2} u_\la(\ell,t) + b_2 u_\la(2,t) \right\} 
\ee
\[- \left\{ a_1q \sum^\infty_{\ell = 2} c_\la(\ell,t) + b_2 c_\la (2,t) \right\} \ 
\left\{ \sum^\infty_{\ell = 2} a_\ell u_\la (\ell,t)\right\} \bigg] \ \Big/ \left[ \sum^\infty_{\ell = 2} a_\ell c_\la (\ell,t)\right]^2.\]
Evidently $u_\la(\ell,t), \ \ell \ge 2$, satisfies (\ref{E2}) with initial and boundary conditions given by
\be \label{T2}
u_\la(1,t) = 0, \ t \ge 0, \ u_\la(\ell,0) = 0, \ \ell \ge 2.
\ee
Letting $v_\la(\ell,t)$ be the solution to (\ref{G2}) with zero initial and boundary conditions corresponding to (\ref{T2}) it follows that (\ref{J2}) holds.

We wish to bound the RHS of (\ref{S2}) in terms of $\| c_1 - c'_1\|_\del$.  To do this we first observe that from (\ref{G2}) we have the identity,
\[  \sum^\infty_{\ell = 2}  v_\la (\ell,t) = -b_2 \int^t_0 ds \ v_\la(2,s) + 2 \| c_1 - c'_1\|_\del \int^t_0 \sum^\infty_{\ell = 2} a_\ell c_\la(\ell,s),
\]
whence we conclude that
\be \label{U2}
\| \sum^\infty_{\ell = 2} v_\la(\ell, \cd)\|_\del \le 2\del \| c_1 - c'_1\|_\del \| \sum^\infty_{\ell = 2} a_\ell c_\la(\ell,\cd)\|_\del \  .
\ee
 
Now from (\ref{C2}) we have that
$$
\sum^\infty_{\ell = 2} a_\ell c_\la (\ell,t) = \sum^\infty_{\ell = 2} a_\ell \ga_\ell + \int^t_0 ds \Big\{ a_2J(1,c,s) 
+ \sum^\infty_{\ell = 2} \Big[ a_{\ell + 1} - a_\ell \Big] J(\ell, c, s) \Big\} .
$$
Defining now $g(t), \ t \ge 0$, by
\[	g(t) = \sum^\infty_{\ell = 2} a_\ell c_\la (\ell,t),  \]
it follows that there is a constant $K > 0$ such that
\[    |g(t) - g(0)| \le K \ \int^t_0 \ g(s)ds, \ \ t \ge 0.	\]
Hence $h(t) = |g(t) - g(0)|$ satisfies the inequality,
\[   h(t) \le Kt \ g(0) + K\ \int^t_0 \ h(s)ds. \]
We conclude that $h$ satisfies the inequality,
\[	h(t) \le \left[ e^{Kt}-1\right] g(0), \ \ t \ge 0. 	\]
We may therefore choose $\del_0$ small enough so that $g(0)/2 \le g(t) \le 3g(0)/2$, $0 \le t \le \del_0$.  Now from (\ref{J2}), (\ref{U2}) we see that 
\be \label{V2}
\bigg| \left\{ a_1q \sum^\infty_{\ell = 2} u_\la (\ell,t) + b_2 u_\la(2,t) \right\} \Big/ \sum^\infty_{\ell = 2} a_\ell c_\la (\ell,t)\bigg| \le K_1\del \| c_1 - c'_1\|_\del, \ 0 \le t \le \del, 
\ee
for some constant $K_1$ independent of $\del_0$, provided $0 < \del < \del_0$. 

Next we observe from (\ref{E2}) that
\begin{eqnarray*}
\sum^\infty_{\ell = 2} a_\ell u_\la (\ell,t) &=& \int^t_0 ds \Big\{ a_2 J(1, u_\la, s) + \sum^\infty_{\ell = 2} \Big[ a_{\ell + 1} - a_\ell \Big] J(\ell, u_\la, s) \Big\}  \\
&+& \int^t_0 ds \Big[c_1(s) - c'_1(s)\Big] \sum^\infty_{\ell = 2} \Big(a_{\ell+1} - a_\ell\Big) a_\ell c_\la (\ell,s) .
\end{eqnarray*}
It follows then from (\ref{J2}) that there are constants $K_2, K_3 > 0$ independent of $\del_0$ such that
$$
\Big| \sum^\infty_{\ell = 2} a_\ell u_\la (\ell,t)\Big| \le K_1\del \Big\| \sum^\infty_{\ell = 2}  v_\la (\ell,\cd)\Big\|_\del 
+ K_2\del \| c_1 - c'_1\|_\del \; \Big\|\sum^\infty_{\ell = 2} a_\ell c_\la (\ell,\cdot) \Big\|_\del, \ 0 \le t \le \del.
$$
Hence, on using (\ref{U2}), we see that there is a constant $K_4$ such that
\be \label{W2}
\Big| \sum^\infty_{\ell = 2} a_\ell u_\la (\ell,t) \Big| \ \Big/ \ \sum^\infty_{\ell = 2} a_\ell c_\la (\ell,t) \le 
K_4\del \| c_1 - c'_1\|_\del,  \ 0 \le t \le \del.
\ee
It is clear now from (\ref{S2}), (\ref{V2}), (\ref{W2}) that on taking $\del_0$ sufficiently small the mapping $T$ is a contraction for $0 < \del \le \del_0$. 
\end{proof}
\begin{theorem} The system (\ref{C2}), (\ref{N2}), (\ref{O2}) has a unique solution global in time.
\end{theorem}
\begin{proof} From Lemma 2.2 the solution exists for time $0 \le t < \del$.  Since the only constraint on the initial condition is (\ref{O2}) at $t=0$ and on $c_1(t)$ equation (\ref{Q2}) at $t=0$ we may propagate the solution beyond time $\del$ since these conditions also hold at $t=\del$. 
\end{proof}
It is evident from (\ref{P2}) that for a solution to (\ref{C2}), (\ref{N2}), (\ref{O2}) the corresponding function $c_1(t)$ satisfies $c_1(t) > z_s, \ 0 \le t < \infty$.  We shall show that the  $\liminf$ of $c_1(t)$ as $t \ra \infty$ is $z_s$.  

\begin{lem}  Let $c(\ell,t), \ \ell \ge 2$, be a solution to (\ref{C2}), (\ref{N2}), (\ref{O2}), and define $g(t)$ by
\[	g(t) = \sum^\infty_{\ell = 2} c(\ell,t).  \]
Then $g$ is monotonic decreasing and $\lim_{t \ra \infty} g(t) = 0$.
\end{lem}
\begin{proof} From (\ref{C2}) we see that $g'(t) = -b_2 c(2,t) < 0$ whence $g$ is strictly monotone decreasing.  From (\ref{P2}) there is the inequality.
\[	c_1(t) \ge z_s + q \; \Big/ \; \left< X^{1/3} \right>_t \ ,	\]
where $\left< \cd \right>_t$ is the probability measure on the integers $\ell \ge 2$ with weights $c(\ell, t)$.  Using the  Jensen inequality we have that 
\[	\left< X^{1/3} \right>_t  \le \left< X\right>^{1/3}_t  = 1/g(t)^{1/3},  \]
whence we conclude that
\[c_1(t) \ge z_s + qg(t)^{1/3}.   \]
Suppose now that $g(t)$ does not converge to 0 as $t \ra \infty$.  Then there exists $\ga, T > 0$ such that $g(t) \ge \ga>0$ for $t \ge T$.  We shall show that this implies
\be \label{X2}
\lim_{t \ra \infty} \sum^\infty_{\ell = 2} c(\ell,t) = \infty,
\ee
contradicting (\ref{O2}).

To prove (\ref{X2}) we first observe that for any $\ell \ge 2$ one has that
\be \label{Y2}
\int^\infty_0 \ c(\ell, s) ds \ < \ \infty, \ \ \ \ell \ge 2.
\ee
For $\ell = 2$ the inequality (\ref{Y2}) evidently follows from the relation $g'(t) = -b_2\;c(2,t)$ and the fact that $g$ is always positive.  For $\ell = 3$ it follows from (\ref{Y2}) for $\ell = 2$ and the fact that $g(t) - c(2,t)$ is bounded for all $t$, and we can proceed by induction to establish it for all $\ell \ge 2$.   Now using (\ref{K2}) we have that for $t > T$,
\begin{eqnarray*}
\sum^\infty_{\ell = 2} \ell c(\ell,t) &=& 1 + \int^t_T ds \Big\{ - b_2 c(2,s) + \sum^\infty_{\ell = 2} \big[a_\ell c_1(s) - b_\ell\big] c(\ell,s) \Big\} \\
&\ge& K(t) + \frac 1 2 \ \int^t_T \ ds \ q\ga^{1/3}\ g(s),
\end{eqnarray*}
where by (\ref{Y2}) $K(t)$ is bounded below uniformly in $t$ as $t \ra \infty$.  Hence (\ref{X2}) holds. 
\end{proof}
\begin{lem} There is the limit, $\displaystyle{\liminf_{t\ra\infty}} \; c_1(t) = z_s$.
\end{lem}
\begin{proof}  We assume that $\displaystyle{\liminf_{t\ra\infty}} \; c_1(t) > z_s$ and obtain a contradiction.  Thus we are assuming that there exists $T, \; \ga > 0$ such that $c_1(t) > z_s + \ga$ for $t \ge T$.  Let $f$ be a positive function on the set $\{ \ell \in \Z : \ell \ge 2\}$.  Then from (\ref{C2}) we have that
\begin{eqnarray*}
\frac d{dt} \ \sum^\infty_{\ell = 2} f(\ell)  c(\ell,t) &=& f(2)\; J(1,c,t) + \sum^\infty_{\ell = 2} \Big[f(\ell+1) -f(\ell)\Big] \; J(\ell,c,t)\\
&=& \left\{ \sum^\infty_{\ell = 2} \Big[f(\ell +1) - f(\ell)\Big]  a_\ell c_1(t) - \Big[f(\ell) - f(\ell - 1)\Big]  b_\ell \right\}c(\ell,t).
\end{eqnarray*}
where we put $f(1) = 0$.  We define now $f(\ell)$ by the recurrence,
\[
\Big[f(\ell +1) - f(\ell)\Big]  a_\ell (z_s + \ga) - \Big[f(\ell) - f(\ell - 1)\Big]  b_\ell = 0, \ f(1) = 0,  \]
for $\ell = 2,3,...,L(\ga)$, where $L(\ga)$ is an integer depending only on $\ga$.  For $\ell > L(\ga)$ we define $f(\ell)$ by the formula,
\[	f(\ell) = f\Big( L(\ga) \Big) \Big[\ell\; / \; L(\ga) \Big]^{1/3} .	\]
It is evident that $L(\ga)$ can be chosen in such a way that there is the inequality,
\[
\Big[f(\ell +1) - f(\ell)\Big]  a_\ell (z_s + \ga) - \Big[ f(\ell) - f(\ell - 1)\Big]  b_\ell \ge 0, \ \ell \ge 2.  \]
We conclude therefore that
\[	\sum^\infty_{\ell = 2} f(\ell) c(\ell,t) \ \ge \ \sum^\infty_{\ell = 2} f(\ell) c(\ell,T) , \ \ t \ge T,   \]
whence it follows that
\[	\liminf_{t\ra\infty} \ \sum^\infty_{\ell = 2} \ell^{1/3} \  c(\ell,t) > 0.	\]
Now from (\ref{P2}) and Lemma 2.3 we have that $\displaystyle{\lim_{t\ra\infty}} \; c_1(t) = z_s$, a contradiction. 
\end{proof}
\begin{rem}  The analogue of Lemma 2.4 for the LSW equation is proved in \cite{np1}, Proposition 3.2.
\end{rem}

\section{Existence and Uniqueness for the Continuous Case}
In this section we prove a global existence theorem for solutions to (\ref{J1}) subject to the constraint (\ref{F1}), with non-negative measurable initial data satisfying (\ref{K1}) and Dirichlet boundary condition $c(0,t)=0, \ t>0$.  We shall also assume without loss of generality that in (\ref{F1}) one has $\rho-\rho_{\rm crit}=1$, and in (\ref{J1}) that $\ve=1$.  Observe that the parameter $L(t)$ in (\ref{J1}) ought to be given in terms of the solution $c(x,t)$ by the formula,
\be \label{A3}
L(t)^{1/3}=\int_0^\infty x^{1/3} c(x,t) dx \Bigg/ \int_0^\infty c(x,t) dx
\ee
in order that the conservation law (\ref{F1}) holds. Thus $L(t)$ is proportional to the average cluster radius at time $t$.

Let ${\mathcal L}_t$ denote the differential operator on the RHS of (\ref{J1}) and  ${\mathcal L}_t^*$ its formal adjoint. Thus ${\mathcal L}_t^*$ is given by the formula
\be  \label{B3}
{\mathcal L}_t^* = (1+x)^{1/3} \ \frac{\pa^2}{\pa x^2} - \left[ 1 - \left\{ \frac x{L(t)} \right\}^{1/3} \right] \frac \pa{\pa x} .
\ee
One can estimate the solution of (\ref{J1}) by solving the equation $\pa w/\pa t = - {\mathcal L}^*w$ backwards in time.  Hence if for some $T > 0, \ w(x,t), \ x > 0, \ t <T$, is the solution with Dirichlet boundary condition $w(0,t) = 0$, we have that
\be \label{C3}
\int^\infty_0 \; w(x,T)c(x,T)dx = \int^\infty_0 \; w(x,0)c(x,0)dx .
\ee
It is well known  \cite{ks}  that solutions to parabolic equations can be written as expectation values.  In particular let $X(s), \ s \ge t$, be the solution to the stochastic differential equation,
\be \label{D3}
dX(s) = - \; \left[ 1 - \left\{ \frac{X(s)}{L(s)} \right\}^{1/3} \right]ds + \sqrt{2}\; \left( 1+X(s)\right)^{1/6}\; dW(s),
\ee
with the initial condition $X(t) = x$.  If $\tau_{x,t}$ is the first hitting time for the process at 0, then
\be \label{E3}
w(x,t) = E\left[ w_0(X(T)); \ \tau_{x,t} > T \right], \ \ x >0, \ t < T,
\ee
is the solution to the equation $\pa w / \pa t = - {\mathcal L}^*w, \ t < T$, with terminal data $w(x,T) = w_0(x)$ and Dirichlet boundary condition $w(0,t) = 0$.  We shall show that we can control the solution to this terminal-boundary value problem by perturbation theory uniformly for $x>0$ if $L(t)$ has a positive lower bound and for $T-t$ sufficiently small.

\begin{lem}  Suppose that the parameter $L(t), \ t \ge 0$, is continuous and satisfies $L(t) \ge L_0 > 0, \; t \ge 0$, where $L_0 < 1$.  Let $w(x,t)$ be the function (\ref{E3}) with $w_0 \equiv 1$.  Then there are positive universal constants $C_1, C_2$ such that
\[
C_1 x / \sqrt{T-t} \le w(x,t) \le C_2x / \sqrt{T-t}, \; 0 \le x \le \sqrt{T-t},   \  C_1 \le w(x,t) \le 1, \ x \ge \sqrt{T-t},
\]
provided $t$ lies in the region $0 < T-t < \del\; L^{1/2}_0$, where $\del > 0$ is universal.
\end{lem}
\begin{proof} We first show the lower bound on $w$.  Since the terminal data $w_0$ is positive and monotonic increasing it follows that $w(x,t) \ge u(x,t)$ where $u(x,t)$ satisfies the equation,
\be \label{F3}
\frac {\pa u}{\pa t} + (1+x)^{1/3} \ \frac {\pa^2 u}{\pa x^2} - \frac {\pa u}{\pa x} = 0, \ \ t < T,
\ee
with Dirichlet boundary condition $u(0,t) = 0$ and terminal data $u(x,T) = w_0(x), \ x \ge 0$.  We can estimate the solution of (\ref{F3}) using perturbation theory provided $T-t < < 1$.  Let $G(x,t)$ be the fundamental solution of the heat equation,
\be \label{CL3}
G(x,t) = \frac 1{\sqrt{4\pi t}} \ \exp \left[ \frac{-x^2}{4t} \right], \ \ t > 0, \ -\infty < x < \infty.
\ee
For $x,y \ge 0, \ t < T$, we define the kernel $K_T(x,y,t)$ by
\be \label{M3}
K_T(x,y,t) = G\left( x-y, \; (1+y)^{1/3}(T-t) \right) - G\left( x+y, \; (1+y)^{1/3}(T-t) \right).
\ee
We have that 
\[	\frac{\pa K_T}{\pa t} + (1+x)^{1/3} \ \frac{\pa^2 K_T}{\pa x^2}- \frac{\pa K_T}{\pa x}= \]
\[   \left\{ a \left( x-y, \; (1+y)^{1/3}(T-t) \right) + \left[ (1+x)^{1/3} - (1+y)^{1/3} \right] b\left( x-y, \; (1+y)^{1/3}(T-t) \right) \right\} \]
\[G \left( x-y, \; (1+y)^{1/3}(T-t) \right) 
-  \bigg\{ a \left( x+y, \; (1+y)^{1/3}(T-t) \right) + \left[ (1+x)^{1/3} - (1+y)^{1/3} \right] \]
\[b\left( x+y, \; (1+y)^{1/3}(T-t) \right) \bigg\} 
 G \left( x+y, \; (1+y)^{1/3}(T-t) \right), \]
where the functions $a,b$ are given by
\[	a(x,t) = x/2t, \ b(x,t) = x^2/4t^2 - 1/2t.	\]
Consider now the function $v(x,t), \ x \ge 0, \ t < T$, defined by
\be \label{G3}
v(x,t) = \int^\infty_0 \ K_T(x, y, t) w_0(y) dy.
\ee
Then one sees from the previous equation that
\be \label{H3}
\frac {\pa v}{\pa t} + (1+x)^{1/3} \ \frac{\pa^2v}{\pa x^2} - \frac{\pa v}{\pa x} = g(x,t), \ x \ge 0, \ t< T,
\ee
where the function $g$ satisfies the inequality,
\be \label{I3}
|g(x,t)| \le C/\sqrt{T-t}, \ \ 0 < T-t < 1,
\ee
for some universal constant $C > 0$.  We also have that $v(0,t) = 0$ and $v(x,T) = w_0(x)$.  Hence $v(x,t)$ is given by the formula,
\[	v(x,t) = u(x,t) - E \left[ \int^T_t \; g\big(X(s),s\big) ds \; ; \; \tau_{x,t} > T \right].	\]
We conclude therefore that $w(x,t) \ge v(x,t) - 2C\sqrt{T-t}$.  
 From (\ref{G3}) we can obtain a lower bound on $v(x,t)$.  To see this first observe from (\ref{M3}) that
\be \label{EA3}
K_T(x,y,t) \ge  \eta_1 G\left( x-y, \; (1+y)^{1/3}(T-t) \right), \quad  |x-y|\le x/2,  \ x\ge\sqrt{T-t},  \ T-t\le 1,
\ee
for some universal $\eta_1>0$. Since $(1+x)^{1/3}(T-t)\le 2x^2$ for $x\ge\sqrt{T-t}, \ T-t\le1$, we also have that
\be \label{EB3}
\int_{|x-y|<x/2} G\left( x-y, \; (1+y)^{1/3}(T-t) \right) dy \ge \eta_2, \quad  x\ge\sqrt{T-t},  \ T-t\le 1,
\ee
for some universal $\eta_2>0$. Hence from (\ref{EA3}), (\ref{EB3}) there exists a universal $\eta_3>0$ such that $v(x,t)\ge \eta_3, \ x\ge\sqrt{T-t},  \ T-t\le 1$.
We conclude that $w(x,t) \ge C_1 > 0$ provided $T-t < \del < 1$ for some suitable universal constant $\del > 0$ if $x$ satisfies $x \ge \sqrt{T-t}$.

To conclude the proof of the lower bound we are left to consider the region $0 \le x \le \sqrt{T-t}$.  We write the solution of (\ref{H3}) with zero Dirichlet and terminal conditions $v(0,t) = 0, \ v(x,T) =  0$, as a series
\begin{eqnarray} \label{J3}
v(x,t) &=& \sum^\infty_{n=0} \ v_n(x,t), \\
v_n(x,t) &=& - \; \int^T_t ds \ \int^\infty_0 \; dy \; K_s(x,y,t)g_n(y,s), \nn \\
g_0 = g, \ g_{n+1} &=& g_n - \left\{ \frac \pa{\pa t} + (1+x)^{1/3} \; \frac {\pa^2}{\pa x^2}  - \frac \pa{\pa x} \right\}v_n, \ n\ge 0. \nn
\end{eqnarray}
 The recurrence relation for $g_n$ can more simply be written as
\[
g_{n+1}(x,t) =  \int^T_t ds  \int^\infty_0 \; dy \left\{ \frac \pa{\pa t} + (1+x)^{1/3} \; \frac {\pa^2}{\pa x^2}-\frac{\pa}{\pa x} \right\} K_s(x,y,t) g_n(y,s).
\]
It follows now in exactly the same way that we derived (\ref{I3}) that there is a universal constant $C > 0$ such that
\[	g_n(x,t) \le C^n(T-t)^{n/2 \;-\; 1/2}, \ \ n \ge 0.	\]
Hence the series (\ref{J3}) converges for $T-t$ sufficiently small.  Observe also that the sum of the $g_n, \ n=0,...,\infty$, is bounded by $C/\sqrt{T-t}$ for some universal constant $C$.  Using the fact that
\[	\int^\infty_0 K_T(x,y,t)dy \le C \min \left[ 1, \ x \big/ \sqrt{T-t} \right],	\]
we can bound the function $v$ in (\ref{J3}).  Suppose $x=\sqrt{s_0 - t}$ where $t < s_0 < t + (T-t)/2$.  Then one has that
\begin{eqnarray*}
 \int^T_t ds  \int^\infty_0 \; dy \ K_s(x,y,t) \; \big/ \; \sqrt{T-s} &\le& C \int^{s_0}_t \; ds \; \Big/ \; \sqrt{T-s} \\
+ \ Cx \; \int^T_{s_0} \; ds \; \Big/ \; \sqrt{(s-t)(T-s)} &\le& 2C \left[ \sqrt{T-t} - \sqrt{T-s_0} \right] \\
+\  \frac{Cx\sqrt{2}}{\sqrt{T-t}} \int^{(T+t)/2}_{s_0} \ \frac{ds}{\sqrt{s-t}} &+& \frac{Cx\sqrt{2}}{\sqrt{T-t}} \int^T_{(T+t)/2} \ \frac{ds}{\sqrt{T-s}} \le C'x,
\end{eqnarray*}
for some constant $C'$.  The lower bound follows easily now by obtaining a lower bound on the function (\ref{G3}) and using the previous inequality.

We consider next the upper bound.  Since it is clear that $w(x,t) \le 1$ we need only restrict ourselves to $x$ in the region $0 \le x \le \sqrt{T-t}$.  Let $u(x,t)$ satisfy the equation,
\be \label{K3}
\frac {\pa u}{\pa t} + (1+x)^{1/3} \; \frac {\pa^2 u}{\pa x^2}  + \left[ \left( \frac x{L_0}\right)^{1/3} - 1\right] \frac {\pa u}{\pa x} = 0,    \  \  t<T,
\ee
with Dirichlet boundary condition $u(0,t) = 0$ and terminal data $u(x,T) = w_0(x), \ x \ge 0$.  Since the data $w_0$ is positive and monotonic increasing we have that $w(x,t) \le u(x,t)$.  To obtain an upper bound on $u(x,t)$ we first note that $u(x,t) = P(\tau_{x,t} > T)$ where $\tau_{x,t}$ is the first hitting time at 0 for the process (\ref{D3}) started at $x$ at time $t$ with $L(s) \equiv L_0$.  For $0 < x < L^{1/4}_0$ let $\tau^*_{x,t}$ be the first exit time from the interval $[0, L^{1/4}_0]$.  Evidently one has
\[	P(\tau_{x,t} > T) \le P(\tau^*_{x,t} > T) + u(x),	\]
where $u(x)$ is the probability that the process started at $x$ exits the interval $[0, L^{1/4}_0]$ through the boundary $L^{1/4}_0$.  Now $u(x)$ satisfies the boundary value problem,
\[   (1+x)^{1/3} \; \frac {\pa^2 u}{\pa x^2}  + \left[ \left( \frac x{L_0}\right)^{1/3} - 1\right] \frac {\pa u}{\pa x} = 0, \ 0 < x < L^{1/4}_0, \ u(0)=0, \ u(L^{1/4}_0) = 1.  \]
The problem is explicitly solvable with solution,
\[	u(x) = \int^x_0 \ \exp\left[ - \int^z_0 h(z')dz' \right] dz \ \Big/ \ \int^{L^{1/4}_0}_0 \exp \left[ - \int^z_0 h(z')dz' \right] dz,  \ 0 < x < L^{1/4}_0, \]
where $h(z) = \left[ (z/L_0)^{1/3} - 1 \right] \big/ [1+z]^{1/3}$.  Since $L_0 < 1$ we conclude that there is a universal constant $C > 0$ such that $u(x) \le Cx/L^{1/4}_0$.  Hence to obtain the upper bound on $w(x,t)$ for $0 < x < \sqrt{T-t}$ it is sufficient to obtain an upper bound on $P(\tau^*_{x,t} > T)$.

We can do this by perturbation theory just as we did for the lower bound.  Thus setting $u(x,t) = P(\tau^*_{x,t} > T)$ it is clear that $u(x,t)$ satisfies (\ref{K3}) for $0 < x < L^{1/4}_0, \ t < T$, with terminal condition $u(x,T) = 1$ and boundary conditions $u(0,t) = u(L^{1/4}_0, \; t) = 0$ .  Let $G_D(x,y,t)$ be the Green's function for the heat equation on the interval $\big[ 0, L^{1/4}_0\big]$ with Dirichlet boundary conditions.  Thus $G_D$ is given by the method of images as an infinite series,
\be \label{P3}
G_D(x,y,t) = \sum^\infty_{m=0} (-1)^{p(m)} \; G(x-y_m, \; t),
\ee
where $y_0 = y$ and $y_m, \; m\ge 1$, are the multiple reflections of $y$ in the boundaries $0, L^{1/4}_0$, with $p(m)$ being the parity of the reflection, $p(0) = 0$.  For $0 \le x,y \le L^{1/4}_0$ and $t < T$ we define the kernel $K_T(x,y,t)$ by 
\be \label{R3}
K_T(x,y,t)  = G_D \left( x, y, (1+y)^{1/3}(T-t) \right).
\ee
We have then that
\begin{multline}         \label{L3}
\frac{\pa K_T}{\pa t} + (1+x)^{1/3} \frac{\pa^2 K_T}{\pa x^2} + \Big[ \Big( \frac x{L_0}\Big)^{1/3} - 1 \Big] \frac{\pa K_T}{\pa x}  \\
= \sum^\infty_{m=0} (-1)^{p(m)} G\left( (x-y_m), (1+y)^{1/3}(T-t) \right) \Big\{ a\big(x-y_m, (1+y)^{1/3}(T-t) \big)\Big[1 - \Big( \frac x{L_0}\Big)^{1/3}\Big]   \\
+ \left[ (1+x)^{1/3} - (1+y)^{1/3} \right] b(x-y_m, (1+y)^{1/3}  (T-t)) \Big\}, 
\end{multline}
just as in (\ref{M3}).  Consider now the function $v(x,t), \ 0 \le x \le L^{1/4}_0$ defined by
\be \label{N3}
v(x,t) = \int^{L^{1/4}_0}_0 K_T(x,y,t)dy, \ \ \ t < T.
\ee
Then one has that
\[   \frac{\pa v}{\pa t}+(1+x)^{1/3} \; \frac {\pa^2 v}{\pa x^2}  + \left[ \left( \frac x{L_0}\right)^{1/3} - 1\right] \frac{\pa v}{\pa x} = g(x,t), \]
where the function $g$ satisfies the inequality,
\[	|g(x,t)| \le C\; \big/ \; L^{1/4}_0 \; \sqrt{T-t}.   \]
Arguing as previously we thus obtain the upper bound on $w(x,t)$ for $0 \le x \le \sqrt{T-t}$ provided $T-t < \del L^{1/2}_0$ for suitably small $\del$ independent of $L_0$. 
\end{proof}
\begin{lem}  Suppose that the parameter $L(t), \ t \ge 0$, is continuous and satisfies $L(t) \ge L_0 >0, \ t \ge 0$,  where $L_0 < 1$.  Let $w(x,t)$ be the function (\ref{E3}) with $w_0(x) = x^{1/3}, \ x > 0$.  Then there are positive universal constants $C_1, C_2$ such that
\begin{eqnarray*}
C_1x \; \Big/ \; (T-t)^{1/3} &\le& w(x,t) \le C_2 x \; \Big/ \; (T-t)^{1/3}, \ 0 \le x \le \sqrt{T-t}, \\
C_1x^{1/3} &\le& w(x,t) \le C_2 x^{1/3}, \  x \ge \sqrt{T-t}, 
\end{eqnarray*}
provided $t$ lies in the region $0 < T-t < \del L^{1/2}_0$, where $\del >0$ is universal.
\end{lem}
\begin{proof} The proof of the lower bound is similar to the proof of the lower bound in Lemma 3.1.  For the proof of the upper bound however we need to make a modification since $w_0(x) = x^{1/3}$ is an unbounded function.  Our starting point is as before that we wish to find an upper bound on the solution to (\ref{K3}).  Then if $X(s)$ is the stochastic process started at $x$ at time $t$ which is associated with the PDE (\ref{K3}), we have that
\begin{eqnarray} \label{O3}
u(x,t) &=& E\left[ w_0 (X(T)); \tau_{x,t} > T \right] \\
&\le& E\left[ w_0 (X(T)); \tau^*_{x,t} > T \right] + u(x) \sup_{t< t' < T} u\Big(L^{1/4}_0, \; t'\Big), \nn
\end{eqnarray}
using the notation of Lemma 3.1.  The first term on the right can be bounded above using perturbation theory as before.  We shall therefore be finished if we can show that $u\Big(L^{1/4}_0, \; t'\Big) \le CL^{1/12}_0$ for $t< t' < T$, provided $T-t< L^{1/2}_0$.  We can use Ito calculus to prove this.  Thus we have that
\[	dX(s) \ = \  \left[ \left\{ \frac{X(s)}{L_0} \right\}^{1/3}- 1 \right]ds + \sqrt{2} \; \Big(1 + X(s)\Big)^{1/6} \; dW(s).	\]
It follows that if $X(t) = x$, then for $t' > t$,
\begin{eqnarray*}
X(t')^2 &=& x^2 + 2 \int^{t'}_t  X(s) \ \left[ \left\{ \frac{X(s)}{L_0} \right\}^{1/3} - 1 \right]ds \\
&+& 2 \int^{t'}_t  \Big( 1 + X(s)\Big)^{1/3}\; ds + 2\sqrt{2}  \int^{t'}_t X(s) \Big(1 + X(s)\Big)^{1/6} \; dW(s).
\end{eqnarray*}
Setting $\tau = \tau_{x, t}$ we conclude that
\[
E\big[ X(t' \wedge \tau)^2\big] = x^2 + E\left[ \int^{t'\wedge \tau}_t  X(s) \ \left[ \left\{ \frac{X(s)}{L_0} \right\}^{1/3}- 1 \right]ds \right] + 2 \left[ \int^{t'\wedge \tau}_t  \Big( 1 + X(s)\Big)^{1/3}\; ds \right].
\]
We take now $x = L^{1/4}_0$ and restrict $t'$ by $t' - t < L^{1/2}_0$.  It follows that there is the inequality,
\[   E\big[ X(t' \wedge \tau)^2\big] \le C_1L^{1/2}_0 + C_2L^{-1/2}_0 \  E\left[ \int^{t'\wedge \tau}_t  X(s)^2 \; ds \right] , \]
for some universal constants $C_1, C_2$.  Gronwall's inequality therefore yields,
\[    E\big[ X(t' \wedge \tau)^2\big] \le C_1L^{1/2}_0  \exp \left[ C_2 (t' - t) \big/ L^{1/2}_0 \right].   \]
We can now apply the Chebyshev inequality to conclude that $u(L_0^{1/4},t')\le CL_0^{1/12}$. 
\end{proof}
\begin{lem}  Let $w(x,t)$ be the function (\ref{E3}) with $w_0 \equiv 1$.  Then $0 \le \pa w(x,t)/\pa x \le C/\sqrt{T-t}$, provided $0 < (T-t) < \del L^{1/2}_0$.
\end{lem}
\begin{proof} We have already observed that $0 \le \pa w(x,t)/\pa x $ so we need an upper bound.  We first prove the upper bound for $0 \le x \le L^{1/4}_0\big/ 2$.  To do this we observe that $w(x,t), \; 0 \le x \le L^{1/4}_0, \ t < T$, satisfies the diffusion equation with terminal conditions $w_0(x)=1,  \ 0 \le x \le L^{1/4}_0$, and boundary conditions $w(0,t) = 0,  \ w \big(  L^{1/4}_0, t\big) \le 1, \ t < T$.  We write $w(x,t) = w_1(x,t) + w_2(x,t)$ where both $w_1$ and $w_2$ satisfy the diffusion equation.  The function $w_1(x,t)$ has terminal data $w_1(x,T) = w_0(x), \  0 \le x \le L^{1/4}_0$, and boundary data $w_1(0,t)=w_1 \big(  L^{1/4}_0, t\big) = 0$.  The function $w_2(x,t)$ has terminal data $w_2(x,T) = 0, \ 0\le x \le L_0^{1/4}$, and boundary data  $w_2(0,t)=0, \ w_2 \big(  L^{1/4}_0, t\big) = w\big(L^{1/4}_0, t\big)$.

We first consider the function $w_1(x,t)$ which we construct by perturbation expansion.  The first term in the series is the function $v(x,t)$ of (\ref{N3}).  Setting $g(x,t) = \pa v / \pa t + {\mathcal L}_t^*v$, then
\begin{eqnarray*}
w_1(x,t) &=& v(x,t) - \sum^\infty_{n=0} \ v_n(x,t), \\
v_n(x,t) &=&- \int^T_t ds \int^{L^{1/4}_0}_0 dy \ K_s(x,y,t) g_n(y,s), \\
g_0=g, & \ &  \ g_{n+1} = g_n - \left\{ \frac \pa {\pa t} + {\mathcal L}_t^* \right\} v_n, \ n \ge 0.
\end{eqnarray*}
Just as before we see that the functions $g_n$ satisfy the inequality,
\be \label{S3}
|g_n(x,t)| \le C^n(T-t)^{n/2 \;-\; 1/2} \ \Big/ \ L^{(n+1)/4}_0, \ \ n \ge 0.
\ee
It easily follows that $|\pa w_1/\pa x| \le C_1 \big/ \sqrt{T-t}$ for some constant $C_1$ provided $0 < T-t < \del L^{1/2}_0$.

Next we consider the function $w_2(x,t)$.  This can be represented in terms of a Green's function $G(x,y,t,T)$, $0 \le x,y \le L^{1/4}_0, \; t<T$, which is the Dirichlet Green's function for the operator ${\mathcal L}_t^*$ on the interval $\big[0, L^{1/4}_0\big]$.  Thus if
\[	u(x,t) = \int^{L^{1/4}_0}_0 \; G(x,y,t,T)  w_0(y)dy, \quad t < T,   \]
then
\begin{multline*} 
\frac {\pa u} {\pa t} + {\mathcal L}_t^*u = 0, \ 0 < x < L^{1/4}_0, \quad t<T, \\ 
 u(0,t) = u\big( L^{1/4}_0, t \big)=0,  \quad u(x,T) = w_0(x),  \ 0 < x < L^{1/4}_0.  
 \end{multline*}
 The function $w_2(x,t),  \ t < T$, has the representation,
\be \label{Q3}
w_2(x,t) = - \int^T_t ds \; w\left( L^{1/4}_0, s\right) \frac{\pa G}{\pa y} \left(x, L^{1/4}_0, t, s\right) \left[1 + L^{1/4}_0\right]^{1/3}.
\ee
To estimate $w_2(x,t)$ we compute the Green's function by perturbation series expansion as we did before.  The first term in the expansion for $G(x,y,t,T)$ is the function $K_T(x,y,t)$ of (\ref{R3}).  If we replace $G$ by this in (\ref{Q3}) it is easy to see that $|\pa w_2 / \pa x| \le C(T-t) \; \big/ \; L^{3/4}_0$, $0 < x < L^{1/4}_0 \big/ 2$.  The complete expansion for $G(x,y,t,T)$ is given by
\begin{eqnarray} \label{T3}
G(x,y,t,T) &=& K_T(x,y,t) - \sum^\infty_{n=0} v_{n,T}(x,y,t), \\
v_{n,T}(x,y,t) &=& - \int^T_t ds \int^{L^{1/4}_0}_0 dy' K_s(x,y',t) g_{n,T}(y',y,s), \nn \\
g_{0,T}(x,y,t) &=& g_T(x,y,t) = \left[ \frac \pa{\pa t} + {\mathcal L}_{t,x}^* \right] K_T(x,y,t), \nn \\
g_{n+1,T} &=& g_{n,T} - \left\{ \frac \pa{\pa t} + {\mathcal L}_t^* \right\} v_{n,T}, \ n \ge 0, \nn \\
g_{n+1,T}(x,y,t) &=&  \int^T_t ds \int^{L^{1/4}_0}_0 dy' \left\{ \frac \pa{\pa t} + {\mathcal L}_{t,x}^* \right\} K_s(x,y',t) g_{n,T}(y',y,s). \nn
\end{eqnarray}
In(\ref{T3}) we have written ${\mathcal L}_t^*={\mathcal L}_{t,x}^*$ to denote that the operator ${\mathcal L}_t^*$ acts on the $x$ variable.
Analogously to (\ref{S3}) we have the estimate
\be \label{U3}
|g_{n,T}(x,y,t)| \le \frac{C^n(T-t)^{n/2 \;-\; 1/2}}{L^{(n+1)/4}_0} \ G\left( x-y, \Big( 1 + L^{1/4}_0\Big)^{1/3} (T-t) \right),
\ee
for some universal constant $C$.  Hence the series (\ref{T3}) converges for $0 < T-t < \del L^{1/2}_0$.  We need now to differentiate the series term by term with respect to the $y$ variable.  Consider first the function $v_{0,T}$  whose derivative is formally given by the expression
\be \label{V3}
\frac{\pa v_{0,T}} {\pa y} (x,y,t) = - \int^T_t ds \int^{L^{1/4}_0}_0 dy' \ K_s(x,y',t) \frac{\pa g_T}{\pa y}(y',y,s).
\ee
Analogously to (\ref{U3}) one has the estimate,
\[	\Big|\frac{\pa g_T(x,y,t)}{\pa y}\Big| \le \frac C{(T-t)L^{1/4}_0}  \ G\left( x-y, \left(1 + L^{1/4}_0 \right)^{1/3} (T-t) \right).  \]
This estimate gives a nonintegrable singularity in (\ref{V3}), so we need to integrate by parts in (\ref{V3}) in the $s,y'$ variables.  First we write
\be \label{W3}
\frac{\pa v_{0,T}}{\pa y} (x,y,t) = - \int^{(T+t)/2}_t ds \int^{L^{1/4}_0}_0 dy' \ K_s(x,y',t) \frac{\pa g_T}{\pa y}(y',y,s)
\ee
\[	- \int_{(T+t)/2}^T ds \int^{L^{1/4}_0}_0   K_s(x,y',t) \left\{ \frac{\pa}{\pa s} + {\mathcal L}^*_{s,y'} \right\} \frac{\pa K_T}{\pa y}(y',y,s).
\]
We have now from (\ref{R3}), (\ref{L3}) that
\be \label{X3}
\left\{ \frac{\pa}{\pa s} + {\mathcal L}^*_{s,y'} \right\} \frac{\pa K_T}{\pa y}(y',y,s) = \left[ \left\{ \frac{y'}{L(s)}\right\}^{1/3} - 1 \right] \frac{\pa^2 K_T}{\pa y'\pa y}
\ee
\[	+ \frac{(1+y')^{1/3}} {3(1+y)^{4/3}} \ \frac{\pa K_T}{\pa s} + \left[ 1 - \left( \frac{1+y'}{1+y} \right)^{1/3}  \right] 
\frac{\pa^2K_T}{\pa s\pa y}.	\]
We substitute the RHS of (\ref{X3}) into the second integral on the RHS of (\ref{W3}).  We then integrate by parts w.r. to $y'$ for the first term on the RHS of (\ref{X3}), and w.r. to $s$ for the second two terms.  Note that since $K_s(x,0,t) = K_s(x, L^{1/4}_0, t) = 0$ there are no boundary terms in the integration by parts w.r. to $y'$.  Once this is accomplished we can estimate $\pa v_{0,T} / \pa y$ since we have removed all non integrable singularities.  Thus we obtain the bound,
\be \label{Y3}
\left| \frac{\pa v_{0,T}}{\pa y} (x,y,t) \right| \le \frac C{L^{1/4}_0} \ G\Big(x - y, \Big( 1 + L^{1/4}_0 \Big)^{1/3} (T-t)\Big).
\ee
Similarly we also have that
\be \label{Z3}
\left| \frac{\pa g_{1,T}}{\pa y} \; (x,y,t)\right| \le  \frac C{ L^{1/2}_0\sqrt{T-t} } \ G\Big(x - y, \Big( 1 + L^{1/4}_0 \Big)^{1/3} (T-t)\Big).
\ee
Once we have the estimate of (\ref{Z3}) it follows by the same method as was used to derive (\ref{U3}) that
\be \label{WW3}
\left| \frac{\pa g_{n,T}}{\pa y} \; (x,y,t)\right| \le  \frac {C^n(T-t)^{n/2-1}} { L_0^{(n+1)/4} } \ G\Big(x - y, \Big( 1 + L^{1/4}_0 \Big)^{1/3} (T-t)\Big), \ n \ge 1.
\ee
Hence from (\ref{T3}) we also have that
\be \label{UU3}
\left| \frac{\pa^2 v_{n,T}}{\pa x\pa y} \; (x,y,t)\right| \le  \frac {C^n(T-t)^{n/2-1/2}} { L_0^{(n+1)/4} } \ G\Big(x - y, \Big( 1 + L^{1/4}_0 \Big)^{1/3} (T-t)\Big), \ n \ge 1.
\ee
Now, just as we derived (\ref{Y3}) we can see that (\ref{UU3}) also holds for $n=0$.  We conclude therefore that there is a universal constant $C$ such that
\[
\left| \frac{\pa^2G}{\pa x \pa y} (x, L^{1/4}_0, t, s) \right| \le \frac C{L^{3/4}_0},  \quad 0 < x < L^{1/4}_0\big/ 2, \ 0 < s-t<\del L^{1/4}_0.
\]
It follows then from (\ref{Q3}) that one has $|\pa w_2(x,t) / \pa x| \le C\big/ L^{1/4}_0$,  \ $0 < x < L^{1/4}_0\big/ 2, \ T-t < \del L^{1/4}_0$. 

We consider next the case $x \ge L^{1/4}_0\big/ 2$.  First note that we have shown that $\pa w(x,t)/\pa x \le C\big/L^{1/4}_0$, $x = L^{1/4}_0\big/ 2, \ T-t < \del L^{1/4}_0$.  Thus if $v(x,t) = \pa w(x,t)/\pa x, \  x > L^{1/4}_0\big/ 2, \; t<T$, then it is clear that $v$ is the solution to the terminal- boundary value problem, 
\begin{multline} \label{AA3}
\frac {\pa v}{\pa t} + (1+x)^{1/3}  \frac {\pa^2 v}{\pa x^2} + \left[ \frac 1{3(1+x)^{2/3}} + \left\{\frac x{L(t)}\right\}^{1/3} -1\right] \frac {\pa v}{\pa x} \\
 + \frac 1{3\{x^2 L(t)\}^{1/3}}  v = 0,   \quad  t < T,  \  x > L^{1/4}_0\big/ 2, 
\end{multline}
$$v\big( L^{1/4}_0\big/ 2, t\big) = \pa w/\pa x\big(L^{1/4}_0\big/ 2, \ t), \ \  t < T;  
\quad v(x,T) = 0, \ \ x \ge L^{1/4}_0\big/ 2. $$	

Let $X(s)$ denote the diffusion corresponding to the equation (\ref{AA3}) and for $X(t) = x > L^{1/4}_0\big/ 2$, let $\tau_{x,t}$ be the first hitting time at $L^{1/4}_0\big/ 2$.  Then there is the inequality, 
\[
v(x,t) \le \frac C{L^{1/4}_0} \ E \left[ \exp \left\{ \int^{T\wedge \tau_{x,t}}_t \frac{ds}{3[X(s)^2L(s)]^{1/3}}\right\}\bigg| X(t)=x\right].
\]
In view of the fact that $T-t < \del L^{1/2}_0$ and $X(s) \ge L^{1/4}_0\big/ 2$, $t < s < T \wedge \tau_{x,t}$, we conclude that $v(x,t) \le C_1\big/ L^{1/4}_0$ for some universal constant $C_1$ if $x \ge L^{1/4}_0\big/ 2$. 
\end{proof}
\begin{lem}  Let $w(x,t)$ be the function (\ref{E3}) with $w_0(x) = x^{1/3}, x > 0$.  Then $0 \le \pa w(x,t) /\pa x\le C\big/(T-t)^{1/3}$, provided $0 < (T-t) < \del L^{1/2}_0$.
\end{lem}
\begin{proof} We proceed exactly as in Lemma 3.3 but this time using the results of Lemma 3.2.  Thus we first show that $\pa w(x,t)/\pa x \le C\big/(T-t)^{1/3}$ for $x \le L^{1/4}_0\big/ 2$ and $\pa w(x,t)/\pa x \le C\big/ L^{1/6}_0$ for $x=L^{1/4}_0\big/ 2$.  For $x > L^{1/4}_0\big/ 2$ we consider the terminal-boundary value problem (\ref{AA3}) but now with the terminal data given by $v(x,T) = 1/3x^{2/3}$.  Since $v(x,T) \le C_1\big/ L^{1/6}_0, x \ge L^{1/4}_0\big/ 2$, one concludes just as in Lemma 3.3 that $ \pa w(x,t)/\pa x \le C\big/ L^{1/6}_0, \ x \ge L^{1/4}_0\big/ 2$. 
\end{proof}
\begin{theorem} Suppose $c(x,0), \ x >0$, is a non negative function which satisfies,
\[	0 < \int^\infty_0 (1 + x) c(x,0)dx < \infty.	\]
Then there is a unique solution to (\ref{J1}) subject to the constraint,
\[	\int^\infty_0 \ xc(x,t)dx = \int^\infty_0 xc(x,0) dx.	\]
\end{theorem}
\begin{proof} Just as in $\S 2$ we define a mapping on functions $c_1(T) = 1/L(T)^{1/3}$ where $L(T)$ is defined by (\ref{A3}).  In particular, $c_1(0)$ is determined by the initial data, where $c_1(0) = 1\big/L^{1/3}_1$ for some $L_1 > 0$.  Suppose now we are given $c_1(T), \ 0 \le T \le T_0$, with $c_1(0) = 1/L^{1/3}_1$.  Then we solve (\ref{J1}) with the corresponding function $L(T), \; 0 \le T \le T_0$,  and define a new function $Ac_1(T), 0 \le T \le T_0$, by the RHS of (\ref{A3}).  We define a space $X$ of functions $c_1 : [0, T_0] \ra (0,\infty)$ which are continuous and satisfy $\|c_1\|_\infty \le 1\big /L^{1/3}_0$ for some $L_0 > 0$.  We shall show that the mapping $A$ leaves $X$ invariant provided $L_0 < L_1$ is sufficiently small and $T_0 = \del L^{1/2}_0$ for some universal $\del, \ 0 < \del < 1$.  To see this we write
\begin{eqnarray} \label{AD3}
Ac_1(T) &=& \int^\infty_0  c(x,T)dx \  \Big/ \  \int^\infty_0 x^{1/3} c(x,T) dx \\
&=& \int^\infty_0 w_1(x,0) c(x,0)dx \ \Big/ \  \int^\infty_0 w_2(x,0) c(x,0) dx, \nn
\end{eqnarray}
where $w_1(x,t), \ t < T$, is given by (\ref{E3}) with $w_0 \equiv 1$, and $w_2(x,t), \ t < T$, is given by (\ref{E3}) with $w_0(x) = x^{1/3}, \ x > 0$.  In view of Lemma 3.1 and 3.2 we have then that
\be \label{AB3}
Ac_1(T) \le \left[ \frac{C_2}{\sqrt{T}} \int^{\sqrt{T}}_0 x\; c(x,0)dx + \int_{\sqrt{T}}^\infty  c(x,0)dx \right] \Big/
\ee
\[	\left[ \frac{C_1}{T^{1/3}} \int^{\sqrt{T}}_0 x\; c(x,0)dx + C_1 \int_{\sqrt{T}}^\infty x^{1/3} c(x,0)dx \right] , \]
provided $T < \del L^{1/2}_0$, for some universal $\del > 0$.  Observe now that
\begin{eqnarray*}
\int^\infty_0 x^{1/3} c(x,0)dx &\le& \frac{L^{1/3}_1}{2} \int^{L_1/8}_0 c(x,0)dx + \int_{L_1/8}^\infty x^{1/3}c(x,0)dx \\
&\le& \frac 1 2 \; \int^\infty_0 x^{1/3} c(x,0)dx + \int_{L_1/8}^\infty x^{1/3}c(x,0)dx,
\end{eqnarray*}
whence we conclude that
\be \label{AH3}
\int_{L_1/8}^\infty x^{1/3}c(x,0)dx \ge \frac 1 2 \; \int^\infty_0 x^{1/3} c(x,0)dx .
\ee
It follows then from (\ref{AB3}) that
\[	Ac_1(T) \le 2 \max [C_2, 1] \big/ C_1 L^{1/3}_1, \ T < L^2_1 \big/ 64. \]
One also has from (\ref{AB3}) that
\be \label{AE3}
Ac_1(T) \le \big [C_2/C_1 + 1/C_1] \big/ T^{1/6}, \ 0 < T < \del L^{1/2}_0.	
\ee
It is clear now that there is a universal constant $\eta > 0$ such that if we choose $L_0 = \eta \min[L_1, 1]$ then $A$ leaves $X$ invariant.  Next we show that for $T_0$ sufficiently small the mapping $A$ is a contraction.  To see this we proceed as in $\S 2$.  Thus if $c_1,c'_1 \in X$ then
\begin{eqnarray*}
Ac_1(T) &-& Ac'_1(T) = \int^1_0 d\la \frac \pa{\pa \la} \; A\big[ \la c_1+(1-\la)c'_1\big](T) \\
&=& \int^1_0 d\la \Big\{ \int^\infty_0 x^{1/3} c_\la(x,T)dx \int^\infty_0 \frac {\pa c_\la}{\pa \la} (x,T) dx \\
&-& \int^\infty_0 c_\la(x,T)dx \int^\infty_0 x^{1/3} \frac {\pa c_\la}{\pa \la} (x,T) dx \Big\}\bigg/ \Big[ \int^\infty_0 x^{1/3} c_\la(x,T)dx\Big]^2,
\end{eqnarray*}
where $c_\la$ is the solution to (\ref{J1}) with $1/L(t)^{1/3} = \la c_1(t) + (1-\la)c'_1(t)$.  Hence $u_\la = \pa c_\la/\pa \la$ is the solution to the initial-boundary value problem,
\[	\frac{\pa u_\la}{\pa t} = {\mathcal L}_t u_\la - [c_1(t) - c'_1(t) ] \frac \pa{\pa x} \big\{ x^{1/3} c_\la\big\}, \ \ t > 0,  \]
\[	\ \ \ \ \ u_\la(x,0) = 0, \ x > 0, \ \ \ \ \ \ u_\la(0,t) = 0, \ \ t > 0.	\]
It follows that 
\begin{eqnarray} \label{AC3}
\int^\infty_0 u_\la(x,T)dx &=& \int^T_0 [c_1(t) - c'_1(t) ]dt \int^\infty_0 dx \frac{\pa w_1}{\pa x} (x,t)x^{1/3}  c_\la(x,t), \\
\int^\infty_0 x^{1/3} u_\la(x,T)dx &=& \int^T_0 [c_1(t) - c'_1(t) ]dt \int^\infty_0 dx \frac{\pa w_2}{\pa x} (x,t)x^{1/3} c_\la(x,t), \nn
\end{eqnarray}
where the functions $w_1, w_2$ are as in (\ref{AD3}).  Hence from Lemma 3.3 there is a universal constant $C$ such that
\be \label{AF3}
\bigg| \int^\infty_0 u_\la(x,T)dx \; \bigg/ \; \int^\infty_0 x^{1/3} c_\la(x,T)dx\bigg| \le C\ga(T)\sqrt{T}  \|c_1-c'_1\|_\infty, \ 0< T < \del L^{1/2}_0,
\ee
where $\ga(T)$ is defined by 
\[	\ga(T) = \sup_{0< t < T} \int^\infty_0 x^{1/3} c_\la(x,t)dx \Big/ \int^\infty_0 x^{1/3} c_\la(x,T)dx .  \]
If we use the second equation in (\ref{AC3}) and (\ref{AE3}) we also conclude there is a universal constant $C$ such that
\begin{multline}  \label{AG3}
\bigg| \int^\infty_0 c_\la(x,T)dx  \int^\infty_0 x^{1/3} u_\la(x,T)dx \Big/ \left[  \int^\infty_0 x^{1/3} c_\la(x,T)dx \right]^2 \bigg| \\
 \le C \ga(T) \sqrt{T}  \|c_1-c'_1\|_\infty, \quad  0< T < \del L^{1/2}_0.
\end{multline}
Using (\ref{AH3}) and Lemma 3.2 we see that there is a universal constant $C$ such that
\[   \ga(T) \le C \max \left[ \Big( T/L^2_1\Big)^{1/3}, \ 1 \right], \ \ 0 < T < \del L^{1/2}_0.	\]
It follows now from (\ref{AF3}), (\ref{AG3}) that if we take $T_0 = \nu \min \big[L^{1/2}_0, \; L^{4/5}_1\big] $ for some universal constant $\nu > 0$ then the mapping $A$ is a contraction.  Hence we obtain a unique solution to the diffusive LSW problem up to time $T_0$.  Observe also that for sufficiently small $L_1$ one has that $Ac_1(T_0) \sim 1/L^{2/15}_1 < 1/L^{1/3}_1$, whence one obtains global existence in time. \end{proof}
We show here that the Kohn-Otto argument \cite{ko} may be applied to prove time averaged coarsening for the diffusive LSW model (\ref{F1}), (\ref{J1}).   Letting $c_\ve(x,t)$ be the solution to the diffusive LSW system   (\ref{F1}), (\ref{J1}), we denote by $E_\ve(t)$ the energy
\be \label {AZ3} 
E_\ve(t)=\int_0^\infty x^{2/3} c_\ve(x,t) dx \ .
\ee
From the conservation law (\ref{F1}), the quantity $1/E_\ve(T)^3$ is a measure of the average cluster volume at time $T$. 
\begin{theorem} Let $\rho-\rho_{\rm crit}$ in (\ref{F1}) be normalized to $1$, and $c_0(x), \ x\ge 0$, be a non-negative function which satisfies
\be \label{AY3}
\int_0^\infty (1+x^{4/3} )c_0(x) dx =M_0<\infty.
\ee
Suppose further that $c_\ve(x,t)$ is the solution of the diffusive LSW system (\ref{F1}), (\ref{J1}) with initial data $c_\ve(x,0)=c_0(x), \ x\ge 0$. Then for $0<\ve\le 1$, there are universal constants $K_1, \ K_2$ such that 
\be \label{AX3}
\left[\frac{1}{T}\int_0^TE_\ve(t)^2 dt\right]^{-3/2}\le K_1T,
\ee
provided $T\ge K_2  M_0^3$.
\end{theorem}
\begin{proof} We first show that $E_\ve(t)$ is decreasing. In fact we have from (\ref{J1}) and (\ref{A3}) that
\begin{multline} \label{AW3}
\frac{3}{2}\frac{dE}{dt} = - \frac{\ve}{3} \int^\infty_0 x^{-4/3} (1+x/\ve)^{1/3} c_\ve(x,t) dx - \int^\infty_0 x^{-1/3} c_\ve(x,t)dx \\
+ \left[ \int^\infty_0  c_\ve(x,t)dx \right]^2 \ \Big/ \ \int^\infty_0 x^{1/3} c_\ve(x,t)dx < 0. 
\end{multline}
Next define a length scale $M_\ve(t)$ by
\be \label{AV3}
M_\ve(t)=\int_0^\infty x^{4/3} c_\ve(x,t) dx,
\ee
so that $M_\ve(0)\le M_0$. It follows again from (\ref{J1}) that
\be \label{AU3}
\frac{3}{4}\frac{dM_\ve}{dt} =  \frac{\ve}{3} \int^\infty_0 x^{-2/3} (1+x/\ve)^{1/3} c_\ve(x,t) dx + \int^\infty_0 x^{1/3} \left[ \left\{ \frac x{L_\ve(t)}\right\}^{1/3} - 1 \right] c_\ve(x,t)dx,
\ee
where $L_\ve(t)$ is given by (\ref{A3}) with $c_\ve(\cdot,t)$ in place of $c(\cdot,t)$.
Now just as in \cite{dp} we have from the Schwarz inequality,
\begin{multline} \label{AJ3}
 \left[ \int^\infty_0 x^{1/3} \left[ \left\{ \frac x{L_\ve(t)}\right\}^{1/3} - 1 \right]c_\ve(x,t)dx\right]^2 
\le \int^\infty_0 xc_\ve(x,t)dx \int^\infty_0 x^{-1/3} \left[ \left\{ \frac x{L_\ve(t)}\right\}^{1/3} - 1 \right]^2 c_\ve(x,t)dx \\
 =    \int^\infty_0 x^{-1/3} c_\ve(x,t)dx-\left[ \int^\infty_0  c_\ve(x,t)dx \right]^2 \ \Big/ \ \int^\infty_0 x^{1/3} c_\ve(x,t)dx,
\end{multline}
where we have used (\ref{F1}), (\ref{A3}).  

We show that for $0<\ve\le1$ there is a constant $C>0$ such that
\be \label{AI3}
\left[ \ve \int^\infty_0 x^{-2/3}(1+x/\ve)^{1/3} c_\ve(x,t)dx \right]^2 \le C  \ve\int^\infty_0 x^{-4/3}(1+x/\ve)^{1/3} c_\ve(x,t)dx.  
\ee
To see this observe that the LHS is bounded above by
\begin{eqnarray*}
&\ & 2 \left[ \ve\int^1_0 x^{-2/3}(1+x/\ve)^{1/3} c_\ve(x,t)dx \right]^2 + 2 \left[\ve \int^\infty_1 x^{-2/3}(1+x/\ve)^{1/3} c_\ve(x,t)dx \right]^2 \\
&\le& 2 \left[\ve^2 \int^1_0 x^{-4/3}(1+x/\ve)^{2/3} c_\ve(x,t)dx \right] \ \left[ \int^1_0 c_\ve(x,t)dx \right] \\
&+& 2 \left[\ve^2 \int^\infty_1 x^{-7/3}(1+x/\ve)^{2/3} c_\ve(x,t)dx \right] \ \left[ \int^\infty_1 x\,c_\ve(x,t)dx \right] \\
&\le& 2^{4/3} \ve^{2/3}\left[ \ve\int^\infty_0 x^{-4/3}(1+x/\ve)^{1/3} c_\ve(x,t)dx \right] \ \left[ \int^\infty_0 (1+x)c_\ve(x,t)dx \right] \\
&\le& 2^{4/3} \left[\ve \int^\infty_0 x^{-4/3}(1+x/\ve)^{1/3} c_\ve(x,t)dx \right] \ \left[ \int^\infty_0 (1+x)c_\ve(x,0)dx \right].
\end{eqnarray*}
We conclude  then from (\ref{AW3})-(\ref{AI3}) that
\be \label{AT3}
|dM_\ve/dt|^2 \le K_3 \ |dE_\ve/dt| \ ,
\ee
for some universal constant $K_3$, provided $0<\ve\le 1$. It is easy to see from the Schwarz inequality that
\be \label{AS3}
E_\ve(t)  M_\ve(t) \ge  1, \ \ \ \ \ t > 0.
\ee
The result follows from (\ref{AT3}), (\ref{AS3}) and Lemma 2 of \cite{pego}.
\end{proof}

\section{The zero diffusion limit}
Our goal in this section will be to establish the limits (\ref{L1}) under the assumption (\ref{K1}) on the initial data for the diffusive LSW system (\ref{F1}), (\ref{J1}). First let us recall how the method of characteristics can be applied to give a representation for the solution of (\ref{E1}). For any $x,t\ge 0$ let $x(s), \ 0\le s\le t$, be the solution of the terminal value problem,
\be \label{A4}
\frac{dx}{ds}=-\left[1-\left(\frac{x(s)}{L(s)}\right)^{1/3}\right],\quad x(t)=x.
\ee
We set $F(x,t)=x(0)$, whence $F(\cdot,t)$ is a mapping on the non-negative reals with derivative
\be \label{R4}
\frac{\pa F(x,t)}{\pa x}= \exp\left[-\frac{1}{3}\int_0^t \frac{ds}{\{x(s)^2L(s)\}^{1/3}}\right] \ .
\ee
 The solution $c_0(x,t), \ x,t\ge 0$, of (\ref{E1}) is given in terms of the initial data $c_0(x), \ x\ge 0$,  by the formula,
\be \label{B4}
\int^\infty_x \; c_0(x', t)dx' = \int^\infty_{F(x,t)} \; c_0(x')dx', \quad x,t\ge 0.
\ee
\begin{lem}  Let $L(\cd)$ be a continuous function on the interval $[0, T]$ satisfying $\inf L(\cd) > 0$ and $F(x,t),  \ x \ge 0, \  0 \le t \le T$, be the corresponding mapping derived from (\ref{A4}).  Suppose $c_0(\cd)$ is a nonnegative function on $(0, \infty)$ satisfying $\int^\infty_0 \; c_0(x)dx < \infty$, and $c_\ve(x,t), \ L_\ve(t)$ is the solution of (\ref{J1}) with initial data $c_0(\cd)$ and Dirichlet boundary condition $c_\ve(0,t) = 0$.  Then if the functions $L_\ve(\cd)$  converge uniformly to the function $L(\cd)$ on the interval $[0,T]$ as $\ve \ra 0$, one has
\[
\lim_{\ve \ra 0} \int^\infty_x \; c_\ve(x', T)dx' = \int^\infty_{F(x,T)} \; c_0(x')dx', \ \ x \ge 0.
\]
\end{lem}
\begin {proof} As in $\S3$ we use the solution to the adjoint problem.  Thus let ${\mathcal L}^*_{t,\ve}$ be the operator,
\be \label{C4}
{\mathcal L}^*_{t,\ve} = \ve\left(1 + x/\ve\right)^{1/3} \frac{\pa^2}{\pa x^2}-\left[ 1 - \left\{ \frac x{L_\ve(t)}\right\}^{1/3} \right] 
\frac{\pa}{\pa x},
\ee
and $w_\ve$ satisfy $\pa w_\ve\big/\pa t = -{\mathcal L}^*_{t,\ve} w_\ve, \ t < T$, $w_\ve(x,T) = w_0(x), \ w_\ve(0,t)=0$.  Then
\be \label{U4}
\int^\infty_0 w_0(x) c_\ve(x,T)dx = \int^\infty_0 w_\ve(x,0) c_0(x)dx.	
\ee
For any $x_0 \ge 0$ we take $w_0(x) = 1, \  x \ge x_0$, $w_0(x) = 0, \ x < x_0$.  Evidently $w_\ve(x,0) \le 1, \ x\ge 0$.  Hence to prove the result we need to show that
\begin{eqnarray} \label{D4}
\lim_{\ve \ra 0} w_\ve(x,0) &=& 1, \ \ x > F(x_0, T), \\
&=& 0, \ \ x < F(x_0, T). \nn
\end{eqnarray}
Recall now that $w_\ve(x,0) = P\big( X_\ve(T) > x_0;  \ \tau_{x,0} > T | X_\ve(0) = x\big) $, where $X_\ve(s)$ is the solution to the stochastic equation,
\be \label{E4}
dX_\ve(s) = - \left[ 1 - \left\{ \frac {X_\ve(s)} {L_\ve(s)} \right\}^{1/3} \right]ds + \sqrt{2\ve} \Big(1 + X_\ve(s)/\ve \Big)^{1/6} dW(s),
\ee
and $\tau_{x,t}$ is the first hitting time at the boundary for the process started at $x$ at time $t$.  We first show that
\be \label{F4}
\lim_{\ve \ra 0} P\left( \tau_{x,0} > T | X_\ve(0) = x \right) = 0, \ x < F(0,T).
\ee
By our assumptions there exists $\ve_0 > 0$ and $L_0 > 0$ such that for $\ve \le \ve_0,  \ \inf L_\ve(\cd) \ge L_0$.  Hence if $Y_\ve(s)$ is the solution to the equation,
\be \label{G4}
dY_\ve(s) = - \left[ 1 - \left\{ \frac {Y_\ve(s)}{L_0}\right\}^{1/3} \right]ds + \sqrt{2\ve} \Big(1 + Y_\ve(s)/\ve\Big)^{1/6} dW(s),
\ee
then there is the inequality,
\be \label{H4}
P\Big(\tau_{x,t} > T | X_\ve(t) = x \Big) \le P\Big(\tau_{x,t} > T | Y_\ve(t) = x \Big), \ t < T.
\ee
Now let $u_\ve(x)$ be the probability that the process $Y_\ve$ started at $x < \del$ exits the interval $[0,\del]$ through the boundary 0.  Then one has
\[   u_\ve(x) = \int^\del_x \exp \left[ - \int^z_0 h_\ve(z')dz' \right]dz \ \Big/\int_0^\del \exp \left[ - \int^z_0 h_\ve(z') dz'\right]dz,  \]
where
\be \label{I4}
h_\ve(z) = \left[ \Big(z/L_0\Big)^{1/3} - 1 \right] \ \big/ \ \ve \left[ 1 + z/\ve \right]^{1/3}.
\ee
Choosing $\del = L_0/2$, it is easy to see that on any interval $[0,\del ']$ with $\del' < \del$ the functions $u_\ve$ converge uniformly to 1 as $\ve \ra 0$.  Next for $0 < x < \del$ let $v_\ve(x)$ be defined by $v_\ve(x) = E\big[\tau_x | Y_\ve(0) = x\big]$, where $\tau_x$ is the first exit time from the interval $[0,\del]$.  Then $v_\ve$ satisfies the equation,
\be \label{J4}
-\ve\Big(1 + x/\ve\Big)^{1/3} v''_\ve(x) + \Big[1 - (x/L_0)^{1/3}\Big] v'_\ve(x) = 1, \  \ 0 < x < \del,
\ee
with Dirichlet  boundary condition $v_\ve(0) = v_\ve(\del) = 0$.  It follows from the maximum principle that
\[   v_\ve(x) \le x \; \big/ \; \Big[1 - (\del/L_0)^{1/3}\Big] , \ 0 < x < \del .  \]
We conclude therefore that 
\be \label{K4}
P \big( Y_\ve(s) \ {\rm hits \ 0 \ before \ time} \  t \  \big| Y_\ve(0) = x \Big)
\ee
\[	\ge \ \ u_\ve(x) - P\Big( \tau_x > t \  \big| Y_\ve(0) = x \Big)  \]
\[ \ge \ \ u_\ve(x) - x/t \Big[1 - (\del/L_0)^{1/3}\Big].  \]
Note that if $x < < t$ then the RHS of (\ref{K4}) is positive as $\ve \ra 0$.  Now on setting $\ve = 0$ in (\ref{J4}) we see that the time for the classical system to exit the interval $[0,\del]$ is $T_x$, where
\[	T_x = \int^x_0 dz \; \big/ \; \Big[1 - (z/L_0)^{1/3}\Big].	\]
Hence we should have that 
\be \label{L4}
\lim_{\ve \ra 0} P\left( \tau_x > T_x + \eta | Y_\ve(0) = x \right) = 0,  
\ee
for any $\eta > 0$.  In order to prove (\ref{L4}) we first show that $\lim_{\ve \ra 0} v_\ve(x) = T_x$.  From (\ref{J4}) we have that
\[   v'_\ve(x) = A_\ve \exp \left[-\int^x_0 h_\ve(z)dz \right] - \int^x_0 \frac{\exp \left[-\int^x_z h_\ve(z')dz' \right]dz}{\ve(1+z/\ve)^{1/3}} \ ,
\]
where $A_\ve$ is given by the formula
\[
A_\ve = \int_{0< z< x' < \del} \frac{\exp \left[-\int^{x'}_z h_\ve(z')dz' \right]}{\ve(1+z/\ve)^{1/3}} dz dx' \bigg/ \int^\del_0 \exp \left[-\int^{x'}_0 h_\ve(z)dz \right]dx'.  \]
Letting $g(z) = 1 - (z/L_0)^{1/3}$ we have that
\[
- \int^x_0 \frac{\exp \left[-\int^x_z h_\ve(z')dz' \right]}{\ve(1+z/\ve)^{1/3}}dz =  \int^x_0 \frac{h_\ve(z)}{g(z)}\exp \left[-\int^x_z h_\ve(z')dz' \right]dz  \]
\[   = \frac 1{g(x)} - \frac 1{g(0)} \exp \left[-\int^x_0 h_\ve(z')dz'\right] + \int^x_0 \frac{g'(z)}{g(z)^2} 
\exp \left[-\int^x_z h_\ve(z')dz' \right]dz.  \]
Similarly we have that 
\[
A_\ve = \frac 1{g(0)} - \int^\del_0 \frac{dx'}{g(x')} \ \Big/ \ \int^\del_0 \exp \left[-\int^{x'}_0 h_\ve(z')dz' \right]dx'  \]
\[
-\int_{0< z< x' < \del} \frac{g'(z)}{g(z)^2} \exp \left[-\int^{x'}_z h_\ve(z')dz' \right] dz dx' \bigg/ \int^\del_0 \exp \left[-\int^{x'}_0 h_\ve(z')dz' \right]dx'.  \]
We conclude that
\begin{eqnarray*}
v'_\ve(x) &=& \frac 1{g(x)} -  \exp \left[-\int^x_0 h_\ve(z)dz\right] \int^\del_0 \frac{dx'}{g(x')} \ \Big/ \int_0^\del \exp \left[-\int^{x'}_0 h_\ve(z)dz\right]dx' \\
&+& \int^\del_0 \ \frac{g'(z)} {g(z)^2} \exp \left[-\int^x_z h_\ve(z')dz' \right] \Big\{ H(x-z) \\
&-& \int^\del_z \exp \left[-\int^{x'}_0 h_\ve(z')dz' \right] dx' \bigg/ \int^\del_0 \exp \left[-\int^{x'}_0 h_\ve(z')dz' \right]dx'\Big\}dz,
\end{eqnarray*}
where $H(y)$ is the Heaviside function, $H(y) = 1, \ y > 0$, $H(y) = 0, \ y < 0$.  It follows easily from the previous expression that $\lim_{\ve \ra 0} v'_\ve(x) = 1/g(x)$ uniformly in any interval $[0,\del ']$ with $\del' < \del$.  Hence $\lim_{\ve \ra 0} v_\ve(x) = T_x$ for $0 \le x < \del$.

We can in a similar way estimate $w_\ve(x) = E\big[\tau^2_x | Y_\ve(0) = x\big]$.  In fact $w_\ve$ satisfies the equation,
\be \label{M4}
-\ve(1+x/\ve)^{1/3} w''_\ve(x) + \left[ 1 - (x/L_0)^{1/3} \right]w'_\ve(x) = 2v_\ve(x), \ \ 0 < x< \del,
\ee
with Dirichlet boundary condition $w_\ve(0) = w_\ve(\del) = 0$.
Proceeding as in the previous paragraph we see that $\lim_{\ve \ra 0}w'_\ve(x) = 2T_x/g(x)$ uniformly in any interval $[0, \del']$ with $\del' < \del$.  It follows that $\lim_{\ve \ra 0}w_\ve(x) = T^2_x$ for $0 < x < \del$.  Finally we conclude from the Chebyshev inequality that (\ref{L4}) holds.

From (\ref{H4}) and (\ref{L4}) we see that for $x$ satisfying $0 < x < L_0/2$ then (\ref{F4}) holds provided $T > T_x$.  We show now that it holds for all $x < F(0,T)$.  To do this let $\bar \tau_{x,t}$ be the first hitting time for the process $X_\ve(s)$ with $X_\ve(t) = x > L_0/4$ to hit $L_0/4$.  Putting $\bar\tau= \tau_{x,0}$ we see as in Lemma 3.2 that
\begin{eqnarray*}
E\left[X_\ve\big(t' \wedge \bar \tau\big)^2\right] &=& x^2 + 2 E\left[ \int^{t'\wedge\bar\tau}_0 X_\ve(s)\left[ \left\{ \frac{X_\ve(s)}{L_\ve(s)} \right\}^{1/3} - 1\right] ds \right] \\
&+& 2 \ve E \left[ \int^{t' \wedge\bar\tau}_0 \big( 1 + X_\ve(s) / \ve\big)^{1/3} ds \right].  
\end{eqnarray*}
It follows that there is a constant $C(L_0)$ depending only on $L_0$ such that
\[   E \left[X_\ve\big(t' \wedge \bar \tau\big)^2\right] \le x^2 + C(L_0) \ E \left[ \int^{t' \wedge\bar\tau}_0 X_\ve(s\wedge\bar\tau)^2ds\right], \quad 0 \le t' \le T.   \]
Gronwall's inequality therefore yields,
\be \label{N4}
 E \left[X_\ve\big(t' \wedge \bar \tau\big)^2\right] \le x^2 \exp\left[ C(L_0) t' \right], \ \ 0 \le t' \le T.
\ee
Next, applying Ito's lemma to $\left[X_\ve(t') - X_0(t')\right]^2$ we see from (\ref{E4}) that
\begin{eqnarray*}
\left[X_\ve(t') - X_0(t')\right]^2 &=& 2 \int^{t'}_0 \left[X_\ve(s) - X_0(s)\right] \left[ \left\{ \frac{X_\ve(s)}{L_\ve(s)}\right\}^{1/3} - \left\{ \frac{X_0(s)}{L_0(s)}\right\}^{1/3} \right]ds \\
&+& 2\sqrt{2\ve} \int^{t'}_0 \left[X_\ve(s) - X_0(s)\right] \left[ 1 + X_\ve(s)/\ve\right]^{1/6} dW(s) + 2\ve  \int^{t'}_0  \left[ 1 + X_\ve(s)/\ve\right]^{1/3} ds.
\end{eqnarray*}
With $\bar \tau$ as in (\ref{N4}) we then have that
\begin{multline}  \label{O4}
E \left\{ \left[ X_\ve(t' \wedge \bar \tau) - X_0(t' \wedge \bar \tau)\right]^2\right\} \le C(L_0) \bigg\{ \left[ \sup_{0\le s\le T} |L_\ve(s) - L_0(s)| + \ve^{2/3}\right] 
x^2 \exp \big[ C(L_0) t' \big]  \\
+ \int^{t'\wedge\bar\tau}_0 E\left\{ \left[X_\ve (s \wedge \bar \tau) - X_0 ( s\wedge \bar \tau) \right]^2 ds \right\},  \quad 
0 \le t' \le T \wedge T_x,		
\end{multline}
for a constant $C(L_0)$ depending only on $L_0$, where $T_x$ is defined by $X_0(T_x)=0$.  Again from the Ito lemma we have that 
\be \label{P4}
X_\ve(t') - X_0(t') =  \int^{t'}_0 \left\{ \frac{X_\ve(s)}{L_\ve(s)}\right\}^{1/3} - \left\{ \frac{X_0(s)}{L_0(s)}\right\}^{1/3} \;ds
+ \sqrt{2\ve} \int^{t'}_0  \left[ 1 + X_\ve(s)/\ve\right]^{1/6} dW(s).
\ee
Now by the Kolmogorov inequality we have that for $\eta > 0$,
\[
P\left( \sup_{0\le t'\le T} \bigg| \sqrt{2\ve} \int^{t' \wedge \bar \tau}_0 \left[ 1 + X_\ve(s)/\ve\right]^{1/6} dW(s) \bigg| > \eta/2\right)   
\le \ \frac{8\ve}{\eta^2} \ E \left[ \int^{T \wedge\bar\tau}_0 \left[ 1 + X_\ve(s)/\ve\right]^{1/3} \; ds \right].  
 \]
We also have that 
\begin{multline*}
P\left( \sup_{0\le t'\le T \wedge T_x} \Big| \int^{t'\wedge \bar \tau}_0 \left\{ \frac{X_\ve(s)}{L_\ve(s)}\right\}^{1/3} - \left\{ \frac{X_0(s)}{L_\ve(s)}\right\}^{1/3} \;ds\Big| > \eta/4 \right)   \\
\le \frac{C(L_0)}{\eta} \sqrt{T} \ E \left[ \int^{T \wedge T_x \wedge \bar\tau}_0 \Big\{ X_\ve(s) - X_0(s)\Big\}^2 ds \right]^{1/2}, 
\end{multline*}
for a constant $C(L_0)$ depending only on $L_0$.  It follows then from (\ref{N4}), (\ref{O4}), (\ref{P4}) that for any $\eta > 0$,
\be \label{Q4}
\lim_{\ve \ra 0} \ P\left( \sup_{0\le t'\le T \wedge T_x \wedge \bar \tau} \Big| X_\ve(t') - X_0(t') \Big| > \eta \right) = 0.
\ee

We finish the proof of (\ref{F4}).  For $x < F(0,T)$ note that $T_x < T$.  We take $\eta$ in (\ref{Q4}) so that $\eta < L_0/4$ and consider the event $\big\{ | X_\ve(T_x \wedge \bar \tau) - X_0(T_x \wedge \bar \tau)|<\eta \big\}$.  We evidently must have $\bar \tau < T_x$ for this event.  Since then $|L_0/4 - X_0(\bar \tau)| < \eta$ the stopping time $\bar \tau$ must be close to the time required for the classical trajectory to go from $x$ to $L_0/4$.  Now by starting the diffusion $X_\ve$ at time $t = \bar \tau$ and using (\ref{H4}), (\ref{L4}) we conclude that (\ref{F4}) holds.

To complete the proof of the lemma we need to show (\ref{D4}).  For $x > F(0,T)$ one has $T_x > T$ whence the event $\big\{ | X_\ve(T_x \wedge \bar \tau) - X_0(T_x \wedge \bar \tau) | < \eta \big\} $ has high probability.  If $\bar \tau > T$ then $\tau_{x,0} > T$ so we consider the situation $\bar\tau<T$. Since $|L_0/4 - X_0(\bar \tau)| < \eta$ in this case, the stopping time $\bar \tau$ must again be close to the time required for the classical trajectory to go from $x$ to $L_0/4$.  Using our previous argument we can estimate the time taken for the diffusion $X_\ve$ started at time $t = \bar \tau$ at $L_0/4$ to hit 0.  For $\ve$ small this is close to the time for the classical trajectory, whence $\tau_{x,0} > T$ with high probability.  We conclude that
\[	\lim_{\ve \ra 0} P\big( \tau_{x,0} > T | X_\ve(0) = x \big) = 1, \ \ x > F(0,T).		\]
This completes the proof of (\ref{D4}) for $x_0 = 0$.  The argument for $x_0 > 0$ is similar. 
\end{proof}
\begin{lem}  Let $c_\ve(x,t), \ L_\ve(t), \ t\ge 0$, be a solution of the system (\ref{F1}), (\ref{J1}) with initial data $c_0(x)$ satisfying $c_0(x) \ge 0, \ x \ge 0$, $\int^\infty_0 (1+x)c_0(x)dx < \infty$, $\int^\infty_0 x c_0(x)dx = 1$.  Then for any $\ve_0, \ T>0$ the set of functions $L_\ve(\cd), 0 < \ve \le \ve_0$, on the 
interval $[0,T]$ is an equicontinuous family.
\end{lem}
\begin{proof} Let us first consider the classical LSW system (\ref{E1}), (\ref{F1}) and let $\Lambda(t), \ t\ge 0$, be the mean particle volume as in (\ref{M1}). Then one has by Jensen's inequality that $L(t) \le \La(t)$ and it is also easy to see that $\La(t)$ is an increasing function.  For $x\ge 0$ let $w_0(x)$ be defined by
$$w_0(x)=\int_x^\infty c_0(x')dx'. $$
From (\ref{B4}) we have that $\La(t) = 1/w_0(F(0,t))$ and one can further see that $F(0,t) < t$.  Noting that $\La(0)$ is given by
\[	\La(0) = \int^\infty_0 x c_0(x)dx \bigg/ \int^\infty_0 c_0(x)dx,	\]
we see that
\[	\int^\infty_{\La(0)/2} x c_0(x)dx \ge  \frac 1 2 \; \int^\infty_0 x c_0(x)dx   = \frac 1 2 \ ,  \]
whence $w_0\big(\La(0)/2\big) > 0$.  Hence $\La(\cd)$ is bounded above in the interval $[0, \La(0)/4]$, and a-fortiori $L(\cd)$ is bounded above in $[0, \La(0)/4]$.  Next we obtain a lower bound on $L(\cd)$ in this interval.  To do this we use the formula (\ref{A3}) for $L(t)$ which, if we denote by $w(x,t)$ the LHS of (\ref{B4}), is the same as
$$ L(t)^{1/3}=\frac{1}{3}\int_0^\infty x^{-2/3} w(x,t) dx \Big/ w(0,t). $$
 Since $F(x,t) < x + t$ we have that
\begin{eqnarray*}
\frac 1 3 \int^\infty_0 x^{-2/3} w(x,t)dx &\ge& \frac 1 3 \int^\infty_0 x^{-2/3} w_0(x + t)dx \\
&\ge& a^{1/3} \ w_0(a+t), \ \ a > 0, \ \ t > 0.
\end{eqnarray*}
We have then that
\[	L(t)^{1/3} \ge \La(t)a^{1/3} w_0(a+t), \ \ t > 0, \ \ a > 0.	\]
If in this last inequality we set $a = \La(0)/4$ we see that $L(\cd)$ is bounded below on the interval $[0, \La(0)/4]$ by a positive constant.  This argument can be easily extended, using the fact that $\La(\cd)$ is an increasing function, to conclude that $\La(\cd)$ is bounded above on any interval $[0, T]$ and $L(\cd)$ is bounded below by a positive constant on the interval.

The boundedness below of $L(\cd)$ on $[0,T]$ implies the continuity of the functions $\La(\cd)$, $L(\cd)$.  To see this note that for $\del > 0$,
\[  0 < \La(t)^{-1} - \La(t + \del)^{-1} < w_0 \left( F(0,t) \right) - w_0 \left( F(\del, t) \right).	\]
In view of (\ref{R4}) we have that $0 < F(\del, t) - F(0,t) < \del$.  The continuity of $\La(\cd)$ follows now from the continuity of $w_0(\cd)$.  To show continuity of $L(\cd)$ we note that since
\[     \int^\infty_0 w_0 \left( F(x, t) \right)dx = 1 , \]
there exists for any $\eta > 0$ an $x_\eta > 0$ such that
\be\label{S4}
\int^\infty_{x_\eta} \ x^{-2/3} w_0 \left( F(x, t) \right)dx < \eta.	
\ee
We consider the integral
\begin{multline} \label{T4}
\int^{2x_\eta}_0 \ x^{-2/3} \left[ w_0 \left( F(x, t) \right) - w_0 \left( F(x, t + \del) \right)\right] dx = \\
  \int^{2x_\eta}_0 \ x^{-2/3} \left[ w_0 \left( F(x, t) \right) - w_0 \left( F(g(x, t,\del),t )\right)\right] dx , 
 \end{multline}
where
\[	|g(x,t,\del) - x| \le \del \left[ 1 + (x/L_0)^{1/3} \right].	\]
Using the continuity of $w_0(\cd)$ and (\ref{R4}) we conclude that $\del$ can be chosen sufficiently small that the integral in (\ref{T4}) is bounded in absolute value by $\eta$.  Hence from (\ref{S4}) we have that
\[  \Big| \int^\infty_0 \ x^{-2/3} \left[ w_0 \left( F(x, t) \right) - w_0 \left( F(x, t + \del) \right)\right] dx \Big| < 3\eta. \]
This last inequality and the continuity of $\La(\cd)$ implies the continuity of $L(\cd)$.

Next we wish to apply the previous argument to the system (\ref{F1}), (\ref{J1}). Let $\La_\ve(t)$ be as in (\ref{M1}), whence it follows that $\La_\ve(t)$ is an increasing function of $t$ and $\La_\ve(0) = \La(0)$.  From (\ref{U4}) we have that
\be \label{AB4}
	\La_\ve(T)^{-1} = \int^\infty_0 w_\ve(x,0)c_0(x) dx, 
\ee
where $w_\ve(x,T) = 1$ and $w_\ve$ is a solution to the adjoint equation as in Lemma 4.1.  It is easy to see from (\ref{E4}) that
\[	w_\ve(x,0) \ge P\left( \tau_{x,0} > T| X_\ve(0) = x \right),	\]
where $X_\ve(s)$ is the solution to the stochastic equation,
\be \label{V4}
dX_\ve(s) = -ds + \sqrt{2\ve} \left(1 + X_\ve(s)/\ve\right)^{1/6} dW(s).
\ee
We shall show that there is a positive constant $\ga(T, \ve_0)$ depending only on $T$ and $\ve_0$ such that 
\be \label{W4}
P\left( \tau_{x,0} > T| X_\ve(0) = x \right) \ge \ga(T, \ve_0),  \ \ x > 2T,
\ee
provided $0 < \ve \le \ve_0$.  To see this we argue as in Lemma 4.1.  Thus from Ito's lemma applied to (\ref{V4}) we have that
$$
E\left[ X_\ve(t \wedge \tau_{x,0})^2\right] = x^2 - 2E \left[ \int_0^{t \wedge \tau_{x,0}} X_\ve(s)ds \right] 
+ 2\ve E \left[ \int_0^{t \wedge \tau_{x,0}} \Big[ 1 + X_\ve(s)/\ve \Big]^{1/3} ds \right] .
$$
On using the inequality,
\[	(1+z)^{1/3} \le 1+z^{1/3} \le 1+\ve^{-1/3} + \ve^{5/3} \; z^2,	\]
we see then that 
\[	E\left[ X_\ve(t \wedge \tau_{x,0})^2\right] \le x^2 + 2t(\ve + \ve^{2/3}) + 2\ve^{2/3}E \left[ \int_0^{t \wedge \tau_{x,0}} X_\ve(s)^2ds \right]  \ . \] 
It follows from the Gronwall inequality that
\be \label{X4}
E\left[ X_\ve(t \wedge \tau_{x,0})^2\right] \le \Big[ x^2 + 2t(\ve + \ve^{2/3})\Big] \exp \left[ 2\ve^{2/3} \; t\right].
\ee
We also have from (\ref{V4}) that
\[  X_\ve(t \wedge \tau_{x,0}) - x + t = \sqrt{2\ve} \int^{t\wedge \tau_{x,0}}_0 \left[1 + X_\ve(s)/\ve\right]^{1/6} dW(s),
\]
whence one has that
\[ \eta^2 \; P\left( \sup_{0\le t\le T} |X_\ve(t \wedge \tau_{x,0}) - x + t |> \eta\right) \le 2\ve E\left[ \int^{T\wedge \tau_{x,0}}_0 \Big [1 + X_\ve(s)/\ve \Big]^{1/3} ds \right]. 
\]
Using now (\ref{X4}) and the H\"{o}lder inequality we conclude that
\be \label{Y4}
\eta^2 \; P\left( \sup_{0\le t\le T} | X_\ve(t \wedge \tau_{x,0}) - x + t| > \eta\right) \le 
4T \ve^{2/3} \left\{ \ve^2 + \Big[x^2 + 2T (\ve + \ve^{2/3})\big] \exp \left[ 2\ve^{2/3} \; T \right] \right\}^{1/6}. 
\ee
Choosing $\eta = x/2$ in (\ref{Y4}) we see that there is an $x_T \ge 2T$ such that the RHS of (\ref{Y4}) divided by $\eta^2$ does not exceed 1/2 provided $x \ge x_T, \ 0 < \ve \le \ve_0$.  Furthermore there exists $\ve_T > 0$ such that we may take $x_T = 2T$ if $\ve$ is in the region $0 < \ve \le \ve_T$.  We conclude that 
\begin{eqnarray} \label{Z4}
P \left( \tau_{x,0} > T | X_\ve(0) = x \right) &\ge& 1/2, \ x > 2T, \ 0 < \ve \le \ve_T, \\
&\ge& 1/2, \ x > x_T, \ 0 < \ve \le \ve_0. \nn
\end{eqnarray}
To complete the proof of (\ref{W4}) then we need to show that
\be \label{AA4}
P\left( \tau_{x,0} > T | X_\ve(0) = x \right) \ge \ga(T,\ve_0), \ 2T < x < x_T, \ \ve_T < \ve < \ve_0.
\ee
This can be demonstrated using perturbation theory.  Choose an integer $N$ and points $x_j =2 jT/N, j=0,1,...,N+1$.  Now define $p_j, \ j=1,..., N$ by 
\begin{multline*}
p_j = \inf_{\ve_T < \ve < \ve_0} \Big\{ {\rm Probability} \ \ X_\ve(t) \ \ {\rm with} \ \ X_\ve(0) = x_j \\
{\rm exits \ the \ interval} \ \ \big[x_{j-1}, x_{j+1}\big] \ \ {\rm in \ time \ greater \ than} \ \ T/N \big\}  \ . 
\end{multline*}
Then the LHS of (\ref{AA4}) is bounded below by $p_1 \cdots p_N$.  To estimate the $p_j$ we choose $N$ large enough so that the Dirichlet Green's function for the interval $[x_{j-1}, x_{j+1}]$ can be expanded in a converging perturbation series as in Lemma 3.1.  The integer $N$ can be chosen dependent only on $T$ because $\ve_T < \ve < \ve_0$.  One can thus show that the solution of the diffusion equation corresponding to (\ref{V4}) with initial data 1 and Dirichlet boundary conditions is bounded below by a constant depending only on $T$ at $x=x_j, \ t = T/N$.  Hence we obtain a lower bound on $p_j$ and thus have established (\ref{AA4}).

We have now from (\ref{AB4}), (\ref{W4}) that
\be \label{AC4}
\La_\ve(T) \le 1 \bigg/ \ga(T,\ve_0) \int^\infty_{2T} c_0(x)dx,
\ee
whence $\La_\ve(T)$ is uniformly bounded above for $0 < \ve \le \ve_0$ provided $T \le \La(0)/4$.  Arguing as before we can obtain a uniform lower bound on $L_\ve(t)$, \ $0 \le t \le \La(0)/4$.  To do this we write
\[	\int^\infty_0 x^{1/3} \; c_\ve(x,t)dx = \frac 1 3 \int^\infty_0 x^{-2/3} \; \int^\infty_x c_\ve(x',t)dx' ,  \]
whence for any $a > 0$,
\[	L_\ve(t)^{1/3} \ge \La_\ve(t) a^{1/3}  \int^\infty_a c_\ve(x',t)dx' .  \]
Now in the same way as we obtained (\ref{AC4}) we see from this that there is a positive constant $K\big( \La(0) \big)$ depending only on $\La(0)$ such that
\be \label{AD4}
L_\ve(t)^{1/3} \ge K\big( \La(0) \big) \int^\infty_{\La(0)/2} c_0(x)dx, \ 0 \le t \le \La(0)/4, \ 0 < \ve \le \ve_0.
\ee

We have shown that the functions $\La_\ve(\cd)$, $L_\ve(\cd), 0 < \ve \le \ve_0$, are uniformly bounded above and away from zero in the interval $0 \le t \le \La(0)/4$.  Next we show that they are equicontinuous.  To do this we use a formula analogous to (\ref{AB4}).  Thus for $\del > 0$ we write
\be \label{AE4}
\La_\ve(T) - \La_\ve(T+\del)^{-1} = \int^\infty_0 w_\ve(x,0) c_0(x)dx,	
\ee
where $w_\ve(x,t)$ is a solution to the adjoint equation as in Lemma 4.1 with terminal data specified at $t=T$.  The terminal data is given by
\[	w_\ve(x,T) = 1 - P\left( X_\ve(s) > 0, \ T < s < T+\del | X_\ve(T) = x \right),  \]
where $X_\ve(s)$ satisfies the stochastic equation (\ref{E4}).  It is easy to see that
\be  \label{AF4}
	w_\ve(x,T) \le 1 - P\left( X_\ve(s) > 0, \ 0 < s < \del | X_\ve(0) = x \right),  
\ee
where $X_\ve(s)$ satisfies the stochastic equation (\ref{V4}).  We can use (\ref{Y4}) then to estimate the RHS of (\ref{AF4}).  Thus on setting $T=\del$ and $\eta =\sqrt{\del} + x/2$ in (\ref{Y4}) we see from (\ref{AF4}) that $w_\ve(x,T)$ satisfies the inequalities,
\begin{eqnarray} \label{AG4}
w_\ve(x,T) &\le& 1, \ \ \ 0 < x < 4\sqrt{\del}, \\
w_\ve(x,T) &\le& C(\ve_0) \del\big[1 + x^{1/3}\big] \big/ x^2,  \  x >4\sqrt{\del}, \nn
\end{eqnarray}
where $C(\ve_0)$ is a constant depending only on $\ve_0$.

We shall show that if $w_\ve(x,T)$ satisfies (\ref{AG4}) then the RHS of (\ref{AE4}) is small for small $\del$, uniformly in $\ve, \ 0 < \ve \le \ve_0$.  To see this we first consider the case where $0 < \ve < \del^{1/4}$.  Let $X_\ve(s)$ satisfy (\ref{E4}) and $Y_\ve(s)$ the corresponding deterministic equation,
\be \label{AH4}
dY_\ve(s) = - \left[ 1 - \left\{ \frac{Y_\ve(s)}{L_\ve(s)}\right\}^{1/3} \right] ds, \ \ 0 < s < T.
\ee
We obtain an estimate on the size of the difference $X_\ve(T) - Y_\ve(T)$ when $X_\ve(0) = Y_\ve(0) = x$, which is similar to (\ref{Y4}).  Letting $\tau_x$ be the first exit time from the interval $(0, \infty)$ for the process $X_\ve(s)$ started at $x$ at $s=0$ we have as in the derivation of (\ref{X4}) that
$$
E\left[ X_\ve(t \wedge \tau_{x,0})^2\right] \le x^2 + 2t(\ve + \ve^{2/3}) + 2t/L^{1/3}_0 
+ \big(2\ve^{2/3} + 2/L^{1/3}_0\big) E \left[ \int^{t\wedge \tau_x}_0 X_\ve(s)^2 ds \right],
$$
where $L_0$ is a lower bound on $L_\ve(\cd)$.  We conclude that
\be \label{AI4}
E\left[ X_\ve(t \wedge \tau_{x,0})^2\right] \le \big[ x^2 + 2t(\ve + \ve^{2/3} + 1/L^{1/3}_0 )\big] 
\exp \big[ 2( \ve^{2/3} + 1/L^{1/3}_0 ) t \big] .
\ee
From (\ref{E4}), (\ref{AH4}) we have that
\begin{multline}  \label{AJ4}
E \left\{ \left[ X_\ve(t \wedge \tau_x) - Y_\ve(t \wedge \tau_x)\right]^2 \right\} = 2E \left\{ \int^{t\wedge \tau_x}_0 \big[X_\ve(s) - Y_\ve(s)\big] \big[ X_\ve(s)^{1/3} - Y_\ve(s)^{1/3} \big] \bigg/ L_\ve(s)^{1/3} ds\right\}   \\
 + 2\ve E \left\{ \int^{t\wedge \tau_x}_0 \left[1 + X_\ve(s)/\ve\right]^{1/3} ds \right\},  
 \end{multline}
provided  $Y_\ve(s) > 0, \ 0 \le s \le t$.  Setting $y(t)$ to be the LHS of (\ref{AJ4}) we see that
\begin{eqnarray*}
|dy/dt| &\le& 2y(t)/L^{1/3}_0 Y_\ve(t)^{2/3} + 2\ve E \left\{ \big[ 1+X_\ve(t \wedge \tau_x)/\ve\big]^{1/3} \right\} \\
&\le& 2y(t)/L^{1/3}_0 Y_\ve(t)^{2/3} + 2(\ve + \ve^{2/3}) + 2\ve^{2/3} E \left[ X_\ve(t\wedge \tau_x)^2 \right].
\end{eqnarray*}
Integrating this last inequality we conclude that
\begin{multline} \label{AK4}
E \left\{ \left[ X_\ve(t \wedge \tau_x) - Y_\ve(t \wedge \tau_x)\right]^2 \right\} \le \\
 \int^t_0 dt'\Big\{ 2(\ve + \ve^{2/3})
+ 2\ve^{2/3} E  \left[ X_\ve(t' \wedge \tau_x)^2 \right] \Big\}  \exp \left[ \int^t_{t'} 2\big/L^{1/3}_0 Y_\ve(s)^{2/3} ds \right] .  
\end{multline}
Using (\ref{AI4}) we see from (\ref{AK4}) that $X_\ve(t \wedge \tau_x)$ differs from $Y_\ve(t \wedge \tau_x)$ by $O(\ve^{1/3}).$

Recall now that the function $w_\ve(x,0)$ in (\ref{AE4}) is given by the formula,
\be \label{AL4}
w_\ve(x,0) = E\left[ w_\ve(X_\ve(T), T); \tau_x > T | X_\ve(0) = x  \right],
\ee
where $w_\ve(x,T)$ satisfies (\ref{AG4}).  It is easy to see that the RHS of (\ref{AE4}) is small for small $\del$ uniformly in $\ve, \ 0 < \ve < \ve_0$, if we replace $X_\ve(T)$ in (\ref{AL4}) by $Y_\ve(T)$.  In fact let $a_\ve$ have the property that if $Y_\ve(0) = a_\ve$ then $Y_\ve(T) =0$, whence $w_\ve(x,0) = 0, \ x < a_\ve$.  Now from (\ref{R4}) it follows that if $Y_\ve(0) > a_\ve + \del^{1/24}$ then $Y_\ve(T) > \del^{1/24}$.  We conclude therefore that $w_\ve(x,0)$ satisfies the inequalities,
\begin{eqnarray} \label{AM4}
w_\ve(x,0) &=& 0,  \ x < a_\ve \ ; \quad  w_\ve(x,0) \le 1,  \  a_\ve < x < a_\ve + \del^{1/24}, \nn \\
 \  \\
w_\ve(x,0) &\le& K\del^{11/12}, \ \ x > a_\ve + \del^{1/24}, \nn
\end{eqnarray} 
for a constant $K$ depending only on $\ve_0$.  Since $c_0(\cd)$ is an $L^1$ function (\ref{AM4}) implies that the RHS of (\ref{AE4}) is small for small $\del$, uniformly in $\ve, \ 0 < \ve < \ve_0$.

We can extend the argument in the previous paragraph to estimate the actual function $w_\ve(x,0)$ of (\ref{AL4}) by using (\ref{AK4}).  In fact one sees using the Chebyshev inequality and $w_\ve(\cd, T) \le 1$ that
\be \label{AN4}
w_\ve(x,0) \le K\del^{1/12}, \ \ x > a_\ve + \del^{1/24}.
\ee
Here we are using the assumption $\ve < \del^{1/4}$.  Next we show that $w_\ve(x,0)$ is small if $x < a_\ve - \del^{1/24}$.  To see this first note that for such an $x$ then $Y_\ve(t)$ with $Y_\ve(0) = x$ satisfies $Y_\ve(t) = 0$ for some $t=t_\ve$ satisfying $t_\ve < T - \del^{1/24}$.  Now from (\ref{AK4}) one has then that $X_\ve(t_\ve \wedge \tau_x) < O(\del^{1/12})$ with high probability.  We need to show that this implies that $\tau_x < T$ with high probability.  We consider the probability that the diffusion $X_\ve(t)$ with $X_\ve(t_\ve) = x$ hits 0 before time $t_\ve + \del^{1/24}$.  This is greater than the probability of the diffusion $Y_\ve(s)$ defined by (\ref{G4}) with $Y_\ve(0) = x$ hitting 0 before time $ \del^{1/24}$.  Let $I_\del$ be the interval $[0, \del^{3/48}]$, and $u_\ve(x)$ be the probability that $Y_\ve(s)$ with $Y_\ve(0)=x$ exits $I_\del$ through 0.  As previously $u_\ve(x)$ is given by the formula,
\[
u_\ve(x) = \int^{\del^{3/48}}_x \exp \Big[ - \int^z_0 h_\ve(z')dz'\Big]dz \; \Big/ \; 
\int^{\del^{3/48}}_0 \exp \Big[ - \int^z_0 h_\ve(z')dz'\Big]dz ,
\]
where $h_\ve(z)$ is defined by (\ref{I4}).  Next let $v_\ve(x)$ be the expected time for the diffusion $Y_\ve(s)$ with $Y_\ve(0) = x$ to exit $I_\del$.  Then there is the inequality,
\be \label{AO4}
1 - P\Big( Y_\ve(s) \ {\rm hits \ 0 \ before \ time} \ \del^{1/24} | Y_\ve(0) = x \Big) \le 1 - u_\ve(x)
 + \del^{-1/24} v_\ve(x), \ x \in I_\del.
\ee
Comparing the RHS of (\ref{AO4}) to the formulas obtained after (\ref{L4}) we conclude that if $0 < x < \del^{7/96}$ then the RHS of (\ref{AO4}) is bounded by $O(\del^{1/32})$.  Combining this with the Chebyshev inequality applied to (\ref{AK4}) we conclude 
\be \label{AP4}
w_\ve(x,0) \le K\del^{1/48}, \ \ 0 < x < a_\ve - \del^{1/24}.
\ee
The inequalities (\ref{AN4}), (\ref{AP4}) now imply that the RHS of (\ref{AE4}) is small for small $\del$, uniformly in $\ve, \ 0 < \ve < \del^{1/4}$.

To complete the proof of the equicontinuity of the functions $\La_\ve(T)$ we need to consider the case $\del^{1/4} < \ve < \ve_0$.  Our goal will be to show that
\be \label{AQ4}
w_\ve(x,T - \del^{1/3}) \le K\del^{1/24}, \ \ x > 0,
\ee
for some constant $K$ depending only on $\ve_0, \; L_0$.  It follows that $w_\ve(x,0)$ is also bounded by the RHS of (\ref{AQ4}), whence the RHS of (\ref{AE4}) is small for small $\del$, uniformly in $\ve, \  \del^{1/4} < \ve < \ve_0$.  To establish (\ref{AQ4}) we shall use perturbation theory to solve the terminal-boundary value problem for $w_\ve(x,t)$ on the domain $0 < x < L_0$, $T-\del^{1/4} < t < T$.  The terminal data satisfies (\ref{AG4}).  The boundary data at $x=0$ is zero and we may estimate the boundary data at $x=L_0$ from (\ref{AG4}), (\ref{AK4}).  Thus since $t > T-\del^{1/4}$ we have as before from the Chebyshev inequality that
\be \label{AR4}
w_\ve(L_0, t) \le K \del^{1/4}, \ \ T-\del^{1/4} < t < T,
\ee
again for a constant $K$ depending only on $\ve_0, L_0$.

We proceed in a similar way to Lemma 3.3.  Thus we write $w_\ve(x,t) = w_{1,\ve}(x,t) + w_{2,\ve}(x,t)$ where $w_{1,\ve}$ has zero  boundary data and $w_{2,\ve}$ has zero terminal data.  To estimate $w_{1,\ve}$, let $G_D(x,y,t)$ be the Dirichlet Green's function for the interval $[0, L_0]$, where $G_D$ is given by the formula (\ref{P3}).  Similarly to (\ref{R3}) we define $K_{\ve,T}(x,y,t)$ by
\be \label{AS4}
K_{\ve,T}(x,y,t)  = G_D \left( x, y, \ve(1+y/\ve)^{1/3}(T-t)\right), \ 0 < x,y < L_0, \ t < T.
\ee
Consider now the function
\be \label{AT4}
v_\ve(x,t) = \int^{L_0}_0 K_{\ve,T}(x,y,t) w_\ve(y, T)dy, \ t < T.
\ee
Then one has that 
\[  \frac{\pa v_\ve}{\pa t} + \ve\big(1 + x/\ve\big)^{1/3} \; \frac{\pa^2 v_\ve}{\pa x^2} + \Big[ \Big\{ \frac x{L_\ve(t)} \Big\}^{1/3} - 1 \Big] \frac{\pa v_\ve}{\pa x} = g_\ve(x,t),   \]
where $g_\ve$ may be estimated as in Lemma 3.1.  Observing that
\[	| \ve\big(1 + x/\ve \big)^{1/3} \ - \ve\big(1 + y/\ve\big)^{1/3} | \le |x-y|,  \] 
we conclude that there is a constant $C$ depending only on $\ve, L_0$ such that
$$
|g_\ve(x,t)| \le C \Big/ \sqrt{\ve(T-t)}, \ \ \ t < T.
$$
We conclude that
\be \label{AU4}
|w_{1,\ve}(x,t)| \le v_\ve(x,t) + C\sqrt{(T-t)/\ve}.
\ee
To estimate $w_{2,\ve}$ we note that $w_{2,\ve}$ has the representation,
\[   w_{2,\ve}(x,t) = E\Big[ w_\ve(L_0, \tau_{x,t}) ;  \tau_{x,t} < T, \; X_\ve(\tau_{x,t}) =L_0 | X_\ve(t) = x \Big],   \]
where $\tau_{x,t}$ is the first exit time from $[0, L_0]$ for the diffusion $X_\ve(s), s > t$, with $X_\ve(t) = x$.  Hence from (\ref{AR4}) we conclude that
\be \label{AV4}
w_{2,\ve}(x,t) \le K\, \del^{1/4}, \ T - \del^{1/4} < t < T.
\ee
The inequality (\ref{AQ4}) follows now from (\ref{AT4}), (\ref{AV4}) since $\ve > \del^{1/4}$.  This completes the proof of the equicontinuity of the functions $\La_\ve(T), \ 0 < \ve < \ve_0$, in the interval $0 < T < \La(0)/4$.

Next we show equicontinuity of the functions $L_\ve(\cd), \ 0 < \ve < \ve_0$, in the same interval.  Analogously to 
 (\ref{S4}) we show that for any $\eta > 0$ there exists $x_\eta > 0$ such that
\be \label{AW4}
\int^\infty_{x_\eta} dx \ x^{-2/3} \int^\infty_x c_\ve(x', T)dx' \ < \ \eta,
\ee
for all $\ve,  \ 0 < \ve < \ve_0$.  Observe that for $A > 0$,
\be \label{AX4}
\int^\infty_A dx \int^\infty_x c_\ve(x', T)dx' = \int^\infty_0 dz \; c_0(z) \int^\infty_A dx \; P\left(X_\ve (T) > x; \tau_{z,0} > T | X_\ve(0) = z \right).
\ee
From (\ref{AI4}) there is the inequality,
\[	 P\left(X_\ve (T) > x; \tau_{z,0} > T | X_\ve(0) = z \right) \le \min\left[ 1, \ K(z+1)^2 \big/ x^2 \right], \]
for some constant $K$ depending only on $L_0, \ve_0, T$.  Hence
\be \label{AY4}
\int^\infty_A dx  \ P\left(X_\ve (T) > x; \tau_{z,0} > T | X_\ve(0) = z \right) \le K_1 \; \min\left[ z+1, (z+1)^2 \big/ A \right]
\ee
for a constant $K_1$ depending only on $L_0, \ve_0, T$.  It follows easily from (\ref{AY4}) that the RHS of (\ref{AX4}) can be made arbitrarily small by choosing $A$ sufficiently large.  Now the existence of $x_\eta$ satisfying (\ref{AW4}) follows.

To complete the proof of equicontinuity of $L_\ve(\cd)$ we need to show that 
\be \label{AZ4}
\int^{2x_\eta}_0 dx \; x^{-2/3}  \int^\infty_x \left[ c_\ve(x',T) - c_\ve(x', T+\del) \right]dx' < \eta
\ee
for sufficiently small $\del > 0$,  provided $0 < \ve < \ve_0$.  As in (\ref{AE4}) we write
\be \label{BA4}
\int^\infty_a \left[ c_\ve(x,T) - c_\ve(x, T+\del) \right]dx = \int^\infty_0 \; w_{\ve,a}(x,0) c_0(x)dx.
\ee
Thus to prove (\ref{AZ4}) we need to get estimates on the functions $w_{\ve,a}$ which are similar to the estimates we obtained for $a=0$, but which are uniform for $a$ satisfying $0 \le a \le 2x_\eta$.  This is a straightforward  extension of the method we used for the case $a=0$.

We have proved equicontinuity of $\La_\ve(T), L_\ve(T)$, $0 < \ve < \ve_0$, in the interval $0 < T < \La(0)/4$.  We wish to extend this now to arbitrary values of $T$.  To see this suppose $L_\ve(\cd)$ is bounded below by $L_0 > 0$ in the interval $[0,T]$.  Then there is a constant $A$ depending only on $L_0, \; T, \ve_0$ such that
\be \label{BB4}
\int^\infty_A dx \int^\infty_x  c_\ve(x',T) dx' < 1/8, \ 0 < \ve < \ve_0 .
\ee
This follows from (\ref{AX4}) and (\ref{AY4}).  To show that $L_\ve(\cd), \La_\ve(\cd)$ remain bounded beyond the interval $[0,T]$ we note that
\be \label{BC4}
\int^\infty_{\La_\ve(T)/2}  dx \int^\infty_x  c_\ve(x',T) dx' + \frac{\La_\ve(T)}{2} \int^\infty_{\La_\ve(T)/2}  c_\ve(x',T) dx'
\ge 1/2.
\ee
From (\ref{BB4}), (\ref{BC4}) we see that there is a constant $K > 0$ depending only on $L_0, T, \ve_0$ and $\La_\ve(T)$ such that
\[   \int^\infty_{\La_\ve(T)/2} c_\ve(x',T) dx' \ge K, \ \ 0 < \ve < \ve_0.  \]
We proceed now as previously using the same methodology as we used to prove (\ref{AC4}). 
\end{proof}
\begin{theorem}  Let $c_\ve(x,t),  \ L_\ve(t), \ 0 < x,t < \infty$, be the solution  to the diffusive LSW problem (\ref{F1}), (\ref{J1}) with initial data $c_0(x)$ satisfying
\[	\int^\infty_0 (1 + x)c_0(x)dx < \infty, \ \ \int^\infty_0 x\; c_0(x)dx = 1.   \]
Denote by $c_0(x,t), \ L(t)$, the solution of the LSW problem (\ref{E1}), (\ref{F1}) with the same initial data. Then there are for all $x,t \ge 0$ the limits,
\begin{eqnarray} \label{BD4}
\lim_{\ve \ra 0} \int^\infty_x  c_\ve(x',t) dx' &=& \ \ \ \int^\infty_x \; c_0(x',t)dx', \\
\lim_{\ve \ra 0}  L_\ve(t) &=& L(t). \nn
\end{eqnarray}
\end{theorem}
\begin{proof} From Lemma 4.2 we have that the functions $L_\ve(t), \ 0 < \ve < \ve_0$, are equicontinuous on any finite interval $[0,T]$.  Hence there is a subsequence $\ve_j, \ j = 1,2,...,$ with $\lim_{j \ra \infty} \ \ve_j = 0$ such that $L_\ve(t)$ converges uniformly on the interval $[0,T]$ as $\ve \ra 0$ through the sequence $\{ \ve_j \}$ to a continuous function $L(t)$.  By Lemma 4.1 it follows that the first identity in (\ref{BD4}) holds, where $c_0(x,t)$ is the solution to (\ref{E1}).  Since (\ref{F1}) also holds for $c_\ve(\cdot,t)$ , it follows from (\ref{BD4}) that (\ref{F1}) must hold for $c_0(\cdot,t)$, whence $c_0(x,t)$ is the solution to the LSW problem (\ref{E1}), (\ref{F1}).  By uniqueness for the solution to the LSW problem, we can then conclude that (\ref{BD4}) holds as $\ve \ra 0$ through the reals. 
\end{proof}

\section{Convergence of Coarsening Rate}
Finally we wish to show that the rate of coarsening for the diffusive LSW problem (\ref{F1}), (\ref{J1}) converges as $\ve \ra 0$ to the rate of coarsening for the LSW model (\ref{E1}), (\ref{F1}).  To do this we will prove that
\be \label{A5}
\lim_{\ve \ra 0} \frac d{dT} \int^\infty_0  c_\ve(x,T) dx = \frac d{dT} \int^\infty_0  c_0(x,T) dx,
\ee
where $c_\ve$ and $c_0$ are as in Theorem 4.1. Evidently (\ref{N1}) follows from Theorem 4.1 and (\ref{A5}).
Observe that from (\ref{R4}), (\ref{B4}) the derivative on the RHS of (\ref{A5}) exists provided 
the function $c_0(\cdot)$ is continuous at $x=F(0,T)$.
We show next that the derivative on the LHS of (\ref{A5}) exists for all $\ve > 0$.
\begin{lem} Suppose the initial data $c_0(x)$ for the diffusive LSW problem (\ref{F1}), (\ref{J1}) satisfies the conditions of Theorem 4.1.  Then for $\ve, T > 0$ the function $\int^\infty_0 c_\ve(x,T)dx$ is differentiable w.r. to $T$.
\end{lem}
\begin{proof} Let $w_\ve(x,t,T), \ t < T, x > 0$, be the solution of $\pa w_\ve/\pa t = - {\mathcal L}^*_{t,\ve}  w_\ve$, $w_\ve(x,T,T) = 1, \ w_\ve(0,t,T) = 0, \ t < T$, where ${\mathcal L}^*_{t,\ve}$ is given by (\ref{C4}).  Then we have
\[	\int^\infty_0 c_\ve(x,T)dx = \int^\infty_0 w_\ve(x,0,T)c_0(x) dx .  \]
We shall show by perturbation theory that for $t < T$ such that $T-t$ is sufficiently small, the function $w_\ve(x,t,T)$ is differentiable w.r. to $T$ and the derivative is a bounded function.  This will prove the result since the function $v_\ve(x,t,T) = -\pa w_\ve(x,t,T)/\pa T$ satisfies $\pa v_\ve /\pa t =- {\mathcal L}^*_{t,\ve}v_\ve, \; t < T$, with boundary condition $v_\ve(0,t,T) = 0$.

Observe that $1 - w_\ve$ satisfies the diffusion equation with zero terminal data and boundary data 1 at $x=0$.  We shall show how to construct this function using perturbation theory.  We first restrict ourselves to some finite interval $0 < x < \ve$.  Let $w_{1,\ve}(x,t,T), \; t < T, 0 < x < \ve$, be the solution of $\pa w_{1,\ve}/\pa t = -{\mathcal L}^*_{t,\ve} w_{1,\ve},  \ w_{1,\ve} (x,T,T) = 0,  \ w_{1,\ve} (0,t,T)  = 1, \; w_{1,\ve}  (\ve,t,T) = 0, \  t < T, 0 < x < \ve$.  Just as in (\ref{Q3}) $w_{1,\ve}$ can be represented in terms of the Dirichlet Green's function $G_\ve(x,y,t,s)$ for the interval.  Thus
\be \label{J5}
w_{1,\ve}  (x,t,T) = \ve \; \int^T_t \ ds \ \frac{\pa G_\ve}{\pa y} \; (x, 0, t, s).	
\ee
As in (\ref{T3}) we can represent $G_\ve$ in a series expansion,
\begin{eqnarray} \label{B5}
G_\ve(x,y,t,T) &=& K_{\ve,T}(x,y,t) - \sum^\infty_{n=0} v_{n,\ve,T}(x,y,t), \\
v_{n,\ve,T}(x,y,t) &=& - \int^T_t \; ds \; \int^{\ve}_0 dy' \, K_{\ve,s}(x,y',t)g_{n,\ve,T}(y',y,s), \nn \\
g_{0,\ve,T}(x,y,t) &=& g_{\ve,T}(x,y,t) = \left[ \frac \pa{\pa t} + {\mathcal L}^*_{t,\ve} \right] K_{\ve,T}(x,y,t), \nn \\
g_{n+1,\ve,T} &=& g_{n,\ve,T} - \left\{ \frac \pa{\pa t} + {\mathcal L}^*_{t,\ve} \right\} v_{n,\ve,T}, \ n \ge 0, \nn
\end{eqnarray}
where $K_{\ve,T}$ is given by (\ref{AS4}).  The function $g_{n,\ve,T}$ is given by the recursion formula,
\be \label{C5}
g_{n+1,\ve,T}(x,y,t) =\int^T_t ds \int^{\ve}_0 dy' \left\{ \frac \pa{\pa t} + {\mathcal L}^*_{t,\ve} \right\} K_{\ve,s}(x,y',t)g_{n,\ve,T}(y'y,s).
\ee
Now as in (\ref{U3}) there is a universal constant $C > 0$ such that
\be\label{D5}
|g_{n,\ve,T}(x,y,t)| \le \frac{C^n(T-t)^{n/2 \;-\; 1/2}} {\ve^{(n+1)/2}} \ G(x-y, 2\ve(T-t)), \ n \ge 0.
\ee
Hence the series (\ref{B5}) converges for $T-t< \ve/C$.  If we formally differentiate the RHS of (\ref{D5}) w.r. to $T$ we are led to expect the inequality,
\be \label{E5}
|\pa g_{n,\ve,T}(x,y,t)/\pa y| \le \frac{C^n(T-t)^{n/2 \;-\; 1}} {\ve^{(n/2 \;+\; 1)}} \ G(x-y, 2\ve(T-t)), \ n \ge 0.
\ee
It is easy to see that (\ref{E5}) holds for $n=0$.  For $n=1$ we use (\ref{C5}).  Writing
\be \label{F5}
g_{1,\ve,T} = \int^{(T+t)/2}_t \; ds + \int^T_{(T+t)/2} \; ds,
\ee
one can easily see that the derivative of the first integral on the RHS of (\ref{F5}) w.r. to $y$ is bounded by the RHS of (\ref{E5}) for $n=1$.  To estimate the derivative of the second integral we need to integrate by parts.  We argue as in (\ref{X3}).  Thus
\be \label{G5}
\left\{ \frac {\pa}{\pa s} + {\mathcal L}^*_{\ve,y'} \right\} \frac{\pa K_{\ve,T}} {\pa y} (y',y,s) = \left[ \left\{ \frac{y'}{L_\ve(s)}\right\}^{1/3} - 1 \right] \frac{\pa^2 K_{\ve,T}} {\pa y' \; \pa y} 
\ee
\[ +	\frac{(1+y'/\ve)^{1/3}} {3\ve(1+y/\ve)^{4/3}} \ \frac{\pa K_{\ve,T}} {\pa s} + \left[ 1 - \frac{(1+y'/\ve)^{1/3}} {(1+y/\ve)^{1/3}} \right] \frac{\pa^2 K_{\ve,T}} {\pa s \; \pa y} \  .    \]
We substitute the RHS of (\ref{G5}) into the second integral in (\ref{F5}) as the expression for $\pa g_{0,\ve,T}(y',y,s)/\pa y$.  We then integrate by parts w.r. to $y'$ for the first term, and $s$ for the last two terms in (\ref{G5}).  It is easy to see from this that (\ref{E5}) holds for $n=1$.  Hence (\ref{E5}) holds by induction for all $n \ge 0$.  It follows then from (\ref{B5}) that
\be \label{H5}
|\pa v_{n,\ve,T}(x,y,t)/\pa y| \le \frac{C^n(T-t)^{n/2}}{\ve^{n/2\;+\;1}} \ G(x-y, 2\ve(T-t)), \ n \ge 0.
\ee
Note that to obtain (\ref{H5}) for $n=0$ we need to use the representation (\ref{G5}) and integrate by parts.  Hence if $(T-t) < \ve/C$ the series in (\ref{B5})  for the derivative $\pa G_\ve(x,y,t,T)/\pa y$ converges, and we conclude that
\be \label{I5}
|\pa G_\ve (x,y,t,T)/\pa y| \le \frac C{\sqrt{\ve(T-t)}} \ G(x-y, 2\ve(T-t)), 
\ee
for some constant $C$.  Now from (\ref{J5}), (\ref{I5}) we see that $w_{1,\ve}(x,t,T)$ is differentiable w.r. to $T$ and is given by the formula
\be \label{K5}
\frac{\pa w_{1,\ve}}{\pa T} \; (x,t,T) = \ve \; \frac{\pa G_\ve}{\pa y}(x, 0, t, T).
\ee

We return to consideration of the function $w_\ve(x,t,T)$.  If $X_\ve(s)$ is the diffusion process associated with ${\mathcal L}^*_{s,\ve}$ then
\[	1 - w_\ve(x,t,T) = P\big( \tau_{x,t} < T| X_\ve(t) = x \big),	\]
where $\tau_{x,t}$ is the first hitting time at 0 for $X_\ve(s), s \ge t$.  Suppose that $0 < x < \ve$ and let $\tau_{1,x,t} $ be the first exit time from the interval $[0,\ve]$ for $X_\ve(s), s \ge t$, with $X_\ve(t) = x$.  If $X_\ve(\tau_{1,x,t}) = \ve$ denote by $\tau_{2,x,t} > \tau_{1,x,t}$ the first hitting time at $\ve/2$.  The density $d\mu_{x,t}(s), \ s>t$, associated with $\tau_{2,x,t}$ is defined by
\[	P\left( \tau_{2,x,t} < T\; ; \; X_\ve(\tau_{1,x,t}) = \ve \right) = \int^T_t d\mu_{x,t}(s).	\]
We can use this density to write $1-w_\ve$ in terms of the function $w_{1,\ve}$.  Thus
\begin{multline} \label{L5}
1 - w_\ve(x,t,T) = w_{1,\ve}(x,t,T) + \int^T_t \ d\mu_{x,t}(s) w_{1,\ve}(\ve/2, s, T) \\
+ \displaystyle{\int_{t<s_1 < s_2 < T}} d\mu_{x,t}(s_1) d\mu_{\ve/2,\;s_1}(s_2) w_{1,\ve}(\ve/2,\;s_2,\;T) + \cdots \ . 
\end{multline}
It is evident that the series converges and is term by term differentiable w.r. to $T$, provided $(T-t) << \ve$. 
\end{proof}
Let us write $\rho_\ve(x,t,T) = -\pa w_\ve(x,t,T)/\pa T$ where the function $w_\ve(x,t,T)$ is given in Lemma 5.1.  Then $\rho_\ve(x,t,\cd)$ is the density for the exit time $\tau_{x,t}$ at 0 of the diffusion $X_\ve(s)$ given by (\ref{E4}) with $X_\ve(t) = x$.
\begin{lem}  There are positive constants $c_1,C_1 > 0$ such that if $T-t = c_1\ve$ then $\rho_\ve$ satisfies the inequality,
\[	\rho_\ve(x,t,T) \le \frac{C_1}{\ve} \exp\left[ -(x/4\ve)^{5/3} \right], \ \ x \ge 0.	\]
\end{lem}
\begin{proof} The result follows from Lemma 5.1 provided $0 < x < \ve$.  For $x > \ve$ let $\tau_{x,t}$ be the first hitting time at $\ve/2$ for the diffusion $X_\ve(s)$ with $X_\ve(t) = x$.  Then from the formula (\ref{L5}) it will be sufficient to show that
\be \label{M5}
P(\tau_{x,t} < T) \le C_1 \exp \left[ -(x/4\ve)^{5/3} \right].
\ee
For any non-negative integer $N$ let $I_N$ be the interval $[2^N\ve, 2^{N+1} \ve]$, and $Z_\ve(s)$ the diffusion process started at $Z_\ve(t) = x \in I_N$ which satisfies the stochastic equation,
\be \label{N5}
dZ_\ve(s) = -ds + \sqrt{2\ve} (1 + Z_\ve(s)/\ve)^{1/6} dW(s).
\ee
We denote by $\tau_{1,x,t}$ the exit time from $I_N$ for the process $Z_\ve(s)$ with $Z_\ve(t) = x$.  Evidently one has $P(\tau_{x,t} < T) \le P(\tau_{1,x,t} < T)$.  We can generate the density for $\tau_{1,x,t}$ by perturbation theory as in Lemma 5.1.  The series converges provided $(T-t)/2^{N/3}\ve << 1$  If we take $x \sim 2^{N+1/2}\ve$ then one sees that provided $T-t = c_1\ve$ for some sufficiently  small universal constant $c_1$ there is the inequality,
\[    P(\tau_{1,x,t} < t + c_1 \ve) < C_1 \exp \left[ -(2^N/10)^{5/3} \right], \]
where $C_1$ is also a universal constant.  The inequality (\ref{M5}) follows. 
\end{proof}
From Lemma 5.2 we see that the integral of the function $\rho_\ve(\cd, t, T)$ is bounded independent of $\ve$ as $\ve \ra 0$ provided $T-t \sim \ve$.  We shall show that the integral is in fact close to 1 for $T-t = O(1)$ as $\ve \ra 0$.
\begin{lem}  There exists $\del_0 > 0$ such that if $\del$ satisfies $0 < \del < \del_0$ then one can find $\ve(\del) > 0$ for which the following inequality holds:
\be \label{O5}
\left| 1 - \int^\infty_0 \; \rho_\ve(x,t,T)dx \right| < \del^{1/10}, \ T-t = \del, \ 0 < \ve < \ve(\del).
\ee
\end{lem}
\begin{proof}  We first show that (\ref{O5}) holds in an averaged sense.  Thus let $0 < \eta < \del/2$.  We shall see that there exists $\ve(\del,\eta) > 0$ such that for $0 < \ve < \ve(\del,\eta)$ there is the inequality,
\be \label{P5}
\left| 1 - \frac 1{2\eta} \int^{t+\del + \eta}_{t+\del - \eta} dT \int^\infty_0 \rho(x,t,T)dx \right| < \del^{1/4}.
\ee
Let $\tau_{x,t}$ be the first hitting time at 0 for the diffusion process $X_\ve(s)$ associated with ${\mathcal L}^*_{s,\ve}$, where $X_\ve(t) = x$.  Then (\ref{P5}) is the same as
\be \label{Q5}
\left| 1 - \frac 1{2\eta} \int^\infty_0 P(t + \del -\eta < \tau_{x,t} < t + \del + \eta)dx \right| < \del^{1/4}.
\ee
Now (\ref{Q5}) will follow if we can estimate $P(\tau_{x,t} > T)$ sufficiently accurately for $T-t \sim \del$.  In particular we show that $P(\tau_{x,t} > T) \sim 1$ if $x$ is slightly larger than $T-t$ (depending on $\del$ and $\ve$) and $P(\tau_{x,t} > T) \sim 0$ if $x$ is slightly less than $T-t$.  To do this we proceed as in Lemma 4.1.  We consider the situation when $x$ is large.  For $Z_\ve(s)$ satisfying (\ref{N5}) let $\tau_x$ be the first exit time for the diffusion $Z_\ve(s)$ with $Z_\ve(t) = x$ from the interval $[x/2, 2x]$.  Then $P(\tau_{x,t} < T) \le P(\tau_x < T)$.  Now one has
\[	Z_\ve(s \wedge \tau_x) = x - [s \wedge \tau_x - t] + \int^{s \wedge \tau_x}_t \; \sqrt{2\ve} \ [1 + Z_\ve(s')/\ve]^{1/6} dW(s').   \]
It follows that provided $(T-t) < x/4$ there is the inequality,
\be \label{DA5} 	
P(\tau_{x} < T) \le \left(\frac 4 x \right)^2 \ 2\ve[1+ 2x/\ve]^{1/3} \ (T-t).	
\ee
Hence one can choose $\ve(\del)$ so that if $\ve < \ve(\del)$ the integral in (\ref{Q5}) over $x > 8\del$ is negligible.  For $x < 8\del$ we see just as in the derivation of (\ref{L4}) that
\be \label{DB5} 	
\lim_{\ve \ra 0} P\left(\tau_{x,t} > t + \eta' + x [1 + K\del^{1/3}] \right) = 0,  
\quad    \lim_{\ve \ra 0} P(\tau_{x,t} < t - \eta' + x) = 0,  
\ee
for any $\eta' > 0$, where $K$ is a constant independent of $\ve$ and $\del$.  The inequality (\ref{Q5}) follows from  (\ref{DA5}) and (\ref{DB5}).

To obtain the pointwise estimate (\ref{O5}) we use the fact that $\rho_\ve(x,t,T)$ satisfies the equation $\pa \rho_\ve/\pa t + {\mathcal L}^*_{t,\ve} \rho_\ve = 0$, $\rho_\ve(0,t,T) = 0$.  Let $c_\ve(x,s), s > t$, satisfy the adjoint equation $\pa c_\ve/\pa s = {\mathcal L}_{s,\ve} c_\ve$, $c_\ve(0,s) = 0, s > t,  \ c_\ve(x,t) = 1, x > 0$.  Then one has that
\be \label{R5}
\int^\infty_0 \rho_\ve(x,t,T)dx = \int^\infty_0 \rho_\ve(x,s,T)c_\ve(x,s)dx,  \  \ s > t.
\ee
From (\ref{J1}) we see that $c_\ve(x,t)$ is given by the expectation value,
\be \label{S5}
c_\ve(x,s) = E\left[\exp\left\{ - \int^s_t \frac 2{9\ve} \Big( 1 + X_\ve(s')/\ve \Big)^{-5/3} + \frac 1{3\big( X_\ve(s')^2\; L_\ve(s')\big)^{1/3}} ds' \right\}\ ; \ \tau_{x,s} < t \right],
\ee
where $X_\ve(s')$ is the diffusion process satisfying
\begin{multline}  \label{AF5}
dX_\ve(s') = - \left\{ 1 - \left[ \frac{X_\ve(s')}{L_\ve(s')} \right]^{1/3} + \frac 2 3 \left(1 + \frac{X_\ve(s')}{\ve}\right)^{-2/3} \right\} ds'  \\
+ \sqrt{2\ve} \ \left( 1 + X_\ve(s')/\ve \right)^{1/6} \ dW(s'), \ s' < s, 
\end{multline}
with $X_\ve(s) = x$.  Note that the process $X_\ve(s')$ is running backwards in time.  The stopping time $\tau_{x,s}$ is the first hitting time for the process on the boundary $x = 0$.  From (\ref{S5}) we see that $c_\ve(x,s) < 1, \; s > t$.  Hence by taking $T - s \sim \ve$ we can conclude from Lemma 5.2 that the LHS of (\ref{R5}) is uniformly bounded as $\ve \ra 0$ for any fixed $t < T$.
Let $\rho'_\ve (x,t,T)$ be the density for the diffusion $Z_\ve$ which corresponds to the density $\rho_\ve(x,t,T)$ for $X_\ve$.  Using the time translation invariance of $Z_\ve$ we have from the inequality (\ref{P5}) for $\rho'_\ve$ that
\be \label{T5}
\left| 1 - \frac 1{2\eta} \int^{T-\del + 2\eta}_{T-\del} ds \int^\infty_0 \rho'_\ve(x,s,T)dx \right| < \del^{1/4} ,
\ee
for $0 < \ve < \ve(\del,\eta)$.  Now from equation (\ref{R5}) for $\rho'_\ve$ one has that
\be \label{U5}
\int^\infty_0 \rho'_\ve(x,T-\del,T)dx  = \frac 1{2\eta} \int^{T-\del + 2\eta}_{T-\del} ds \int^\infty_0 \rho'_\ve(x,s,T) c'_\ve(x,s)dx,
\ee
where $c'_\ve$ is given by an expectation similar to (\ref{S5}).  Using the fact that the integral on the RHS of (\ref{U5}) is concentrated at $x \sim \del$ and (\ref{S5}) for $c'_\ve$, we conclude from (\ref{T5}) that one can choose $\ve(\del)$ such that for $0 < \ve < \ve (\del)$ one has the pointwise in time estimate
\[	| 1 - \int^\infty_0 \rho'_\ve(x, T-\del, T)dx | < \del^{1/4}.   \]
The inequality (\ref{O5}) follows if we can show that $\rho'_\ve(\cd, T-\del, T)$ and $\rho_\ve(\cd, T-\del, T)$ are close.

We first compare $\rho_\ve(x, t, T)$ and $\rho'_\ve(x,t,T)$ when $T-t = c_1\ve$ as in Lemma 5.2.  To do this we write $\rho_\ve(x, t, T)$ as a sum over walks on the numbers $0, 1/2, 1, 2, ..., 2^N$, where $2^N \ve \sim 1$.  For a walk $X(n), n = 0,1,2,...$ we take $X(0) > 0$.  If $0 < X(n) < 2^N$ then $X(n+1)$ can be either of the neighbors of $X(n)$.  If $X(n) = 2^N$ then $X(n+1) = 0$.  Finally let $\tau$ be defined by $X(\tau) = 0$ and $X(n) > 0,  \ n < \tau$.  Then for $x = 2^n\ve$ for some integer $n$ with $-1 \le n \le N$ we have that
\begin{multline} \label{V5}   
\rho_\ve(x, t, T) = \begin{array}[t]{c}{\displaystyle\sum}\\ {\scriptstyle\{ walks\ X(\cdot):X(0) = x/\ve\}} \end{array}
\begin{array}[t]{c}{\displaystyle\int}\\ {\scriptstyle s_0< s_1 < s_2 .. <s_\tau} \end{array} \del(s_0 - t)\del(s_\tau - T)
\\
  ds_0 ds_1...ds_\tau \ \prod^\tau_{n=1} \ \rho_\ve \left(\ve X(n-1), \ve X(n), s_{n-1}, s_n \right).  
\end{multline}
Here $\rho_\ve(x, x', s, s')$ is the density at $x'$ at time $s'$ for the diffusion $X_\ve(\cd)$ of (\ref{E4}) with $X_\ve(s) = x$ exiting the interval $I_x$ where $I_x = [x/2, 2x]$ if $x =2^n\ve$ with $0 \le n < N, \ I_x = [0,\ve]$ if $x = \ve/2$.  For $x = 2^N\ve$ then $x' = 0$ and $\rho_\ve(x, x', s, s') = \rho_\ve(x,  s, s')$.  Just as we derived (\ref{I5}) we see that there are positive universal constants $C_1,\; C_2$ such that
\be \label{W5}
\rho_\ve(x, x', s, s') \le \frac{C_1}{s'-s} \exp \left[ -C_2 x^{5/3} \big/ \ve^{2/3}(s' - s) \right].
\ee
It follows that we may assign transition probabilities on the walks $X(\cd)$ so that
\be \label{X5}
\rho_\ve(x, t, T) \le \frac C \ve \ E \left\{ \exp \left[ -\eta \sum^\tau_{n=0} X(n)^{5/3} \right] \big| X(0) = x/\ve \right\},
\ee
for some positive constants $C, \eta$.  This last inequality gives another proof of Lemma 5.2.  To compare $\rho_\ve$ and $\rho '_\ve$ we use the representation (\ref{V5}) and an interpolation formula.  Thus we have
\be \label{Y5}
\rho_\ve(x, t, T) - \rho'_\ve(x,t,T) = \int^1_0 d\lambda \begin{array}[t]{c}{\displaystyle\sum}\\ {\scriptstyle\{ walks\ X(\cdot):X(0) = x/\ve\}} \end{array}
\begin{array}[t]{c}{\displaystyle\int}\\ {\scriptstyle s_0< s_1 < s_2 .. <s_\tau} \end{array} \del(s_0 - t)\del(s_\tau - T)ds_0 ds_1...ds_\tau
\ee
\[	\prod^\tau_{n=1} \left[ \la\rho_\ve\Big( \ve X(n-1), \ve X(n), s_{n-1}, s_n \Big) + (1-\la)
\rho'_\ve \Big( \ve X(n-1), \ve X(n), s_{n-1}, s_n \Big) \right] \]
\[	\sum^\tau_{n=1} \left[ \rho_\ve\Big( \ve X(n-1), \ve X(n), s_{n-1}, s_n \Big) - \rho'_\ve \Big( \ve X(n-1), \ve X(n), s_{n-1}, s_n \Big) \right ] \]
\[  \Big/ \ \left[ \la\rho_\ve \Big( \ve X(n-1), \ve X(n), s_{n-1}, s_n \Big) + (1-\la) \rho'_\ve \Big( \ve X(n-1), \ve X(n), s_{n-1}, s_n \Big) \right]   .\]
Now using the perturbation method of Lemma 5.1 we see that
\be \label{Z5}
| \rho_\ve(x,x',s,s') - \rho'_\ve(x,x',s,s')| \le \frac{C_1x^{1/3}}{s'-s} \ \exp\left[ -C_2x^{5/3} \big/ \ve^{2/3}(s'-s) \right],
\ee
provided $x = 2^n\ve$ for any integer $n, \ -1 \le n < N$.  Since $2^N\ve \sim 1$ the inequality (\ref{Z5}) also holds for $x = 2^N\ve$.  This follows from (\ref{W5}).  Just as we derived (\ref{X5}) we have from (\ref{Z5}) that
\[   | \rho_\ve(x,t,T) - \rho'_\ve(x,t,T) | \le \frac C{\ve^{2/3}} E \left\{ \sum^\tau_{n=1} X(n)^{1/3} \exp \left[ -\eta \sum^\tau_{n=1} X(n)^{5/3} \right] \big| X(0) = x/\ve \right\},   \]
whence it follows that
\be \label{AA5}
 | \rho_\ve(x,t,T) - \rho'_\ve(x,t,T) | \le \frac {C_1}{\ve^{2/3}}  \exp \left[-C_2(x/\ve)^{5/3} \right] 
\ee
provided $T-t = c_1\ve$ and $x < O(1)$.  For $x > O(1)$ the inequality (\ref{AA5}) follows from Lemma 5.2.

We shall now use (\ref{AA5}) and (\ref{R5}) to compare the integrals of $\rho_\ve(\cd,t,T)$ and $\rho'_\ve(\cd,t,T)$ 
when $T-t = \del$.  In fact we have
\begin{eqnarray} \label{AB5}
&\int^\infty_0 & \rho_\ve(x,t,T) - \rho'_\ve(x,t,T) dx = \\
&\int^\infty_0 & \left[ \rho_\ve(x,T-c_1\ve,T) - \rho'_\ve(x,T-c_1\ve,T)\right] c_\ve(x, T-c_1\ve) dx \nn \\
+ &\int^\infty_0 & \rho'_\ve(x,T-c_1\ve,T)\left[  c_\ve(x,T-c_1\ve) - c'_\ve(x, T-c_1\ve) \right] dx \nn,
\end{eqnarray}
where $c'_\ve$ is the analogous expectation for the diffusion $Z_\ve$ which corresponds to $c_\ve$ for $X_\ve$.  Since $c_\ve(x, T-c_1\ve) < 1$ it follows from (\ref{AA5}) that the first integral on the RHS of (\ref{AB5}) is small.  We shall show the second integral is also small by obtaining pointwise estimates on the differences $c'_\ve(x,s) - c_\ve(x,s)$ for $s \sim \del + t$ and $x > 0$.  In fact, in view of Lemma 5.2, we may restrict ourselves to obtaining an estimate on the difference when $x = O(\ve)$.

To carry this out we consider the representation for $c'_\ve$ analogous to the representation (\ref{S5}) for $c_\ve$.  Let $Y_\ve(s')$ be the diffusion process satisfying
\be \label{AC5}
dY_\ve(s') = -\left\{ 1 + \frac 2 3 \left( 1 + \frac{Y_\ve(s')}{\ve} \right)^{-2/3} \right\} ds' + \sqrt{2\ve} \ \left( 1 +Y_\ve(s')/\ve \right)^{1/6} \; dW(s'),  \ s' < s,
\ee
with $Y_\ve(s) = x$.  If we allow the diffusions $X_\ve(s')$ of (\ref{AF5}) and $Y_\ve(s') $ of (\ref{AC5}) to be driven backwards in time by the same Brownian motion $dW(s')$ then it is clear that if $X_\ve(s) = Y_\ve(s) = x$ then $X_\ve(s') \le Y_\ve(s') , s' < s$.   Thus if $\tau_{x,s}$ is the first hitting time at 0 for $X_\ve(s') $ with $X_\ve(s) = x$ and $\tau'_{x,s}$ is similarly defined for $Y_\ve(s')$, then $\tau_{x,s} \ge \tau'_{x,s}$.  Now similarly to (\ref{S5}) we have the representation,
\be \label{AD5}
c'_\ve(x,s) = E \left[ \exp \left\{ - \int^s_t \frac 2{9\ve} \left( 1 + Y_\ve(s')/\ve\right)^{-5/3} ds' \right\};  \ \tau'_{x,s} < t \right].
\ee
If we compare (\ref{S5}), (\ref{AD5}) and use the fact that $\tau_{x,s} \ge \tau '_{x,s}$, $X_\ve(s') \le Y_\ve(s'), s' \le s$, we see that $c_\ve(x,s) \le c'_\ve(x,s)$.  We also have that 
\begin{multline*}
c'_\ve(x,s) - c_\ve(x,s) 
= E\left[ \exp \left\{ - \int^s_t \frac 2{9\ve} \left( 1 + Y_\ve(s')/\ve \right)^{-5/3} ds' \right\} \; ; \; \tau'_{x,s} < t, \tau_{x,s} > t\right] \\
+  E\bigg[ \exp \left\{ - \int^s_t \frac 2{9\ve} \left( 1 + Y_\ve(s')/\ve \right)^{-5/3} ds' \right\} \\
-  \exp \left\{ - \int^s_t \frac 2{9\ve} \left( 1 + X_\ve(s')/\ve \right)^{-5/3}  + \frac 1{3[X_\ve(s')^2 L_\ve(s')]^{1/3}}ds' \right\} \; ; \; \tau_{x,s} < t \bigg].
\end{multline*}
We conclude from the previous inequality that
\begin{multline} \label{AE5}
0 \le c'_\ve(x,s) - c_\ve(x,s) \le P\left( \tau'_{x,s} < t, \ \tau_{x,s} > t \right) 
+ E \left[ \int^s_t  \frac {ds'}{3[X_\ve(s')^2 L_\ve(s')]^{1/3}} \; ; \; \tau_{x,s} < t \right]  \\
+ E \left[  \frac 2{9\ve} \; \int^s_t \left[ 1 + X_\ve(s')/\ve \right]^{-5/3}  - \left[ 1 + Y_\ve(s')/\ve \right]^{-5/3} ds'  \; ; \; \tau_{x,s} < t \right]. 
\end{multline}
We shall estimate each term on the RHS of (\ref{AE5}).
Note that for the purposes of doing this estimate we may replace $L_\ve(s'), \ t \le s' \le t + \del$, on the RHS of (\ref{AF5}) by $L_0=\inf_{t \le s' \le t + \del}  L_\ve(s')$.

We consider the first term which is the same as $P\left( \tau'_{x,s} < t\right) - P\left( \tau_{x,s} < t \right)$.  We shall estimate this by obtaining an upper bound on $P\Big( \tau'_{x,s} < t\Big)$ and a lower bound on $P\left( \tau_{x,s} < t \right)$.  We have now that for $0 < x < \del^2$, 
\begin{eqnarray} \label{AG5}
P\left( \tau'_{x,s} < t\right) &\le& P\left( Y_\ve(s') \ {\rm exits} \ [0,\del^2] \ {\rm through} \ \del^2 \  | Y_\ve(s) = x \right) \\
&+& P\left( {\rm Time \ to \ exit} \ [0,\del^2] \ {\rm larger \ than} \  \ \del/2 \  | Y_\ve(s) = x \right). \nn
\end{eqnarray}
Letting $u_\ve(x)$ be the first probability on the RHS of (\ref{AG5}) we have that
\begin{multline}  \label{AH5}
\ve(1+x/\ve)^{1/3} u''_\ve(x) + \left[ 1 + \frac 2 3 (1+x/\ve)^{-2/3} \right] u'_\ve(x) = 0, \\
 0 < x < \del^2,  \ u_\ve(0) = 0, \  u_\ve (\del^2) = 1.
\end{multline}
The solution to (\ref{AH5}) is given by the formula,
\begin{eqnarray} \label{EA5}
u_\ve(x) &=& \int^x_0 \ h_\ve(x')dx' \big/ \int^{\del^2}_0 \ h_\ve(x')dx', \\
h_\ve(x) &=& \big(1+ x/\ve\big)^{-2/3} \ \exp \left[ - \frac 3 2 \; \big(1+ x/\ve\big)^{2/3} \right]. \nn
\end{eqnarray}
Observe that $u_\ve(x) \simeq 1$ for $x > O(\ve)$.  We estimate the second probability on the RHS of (\ref{AG5}) by calculating the expected time to exit $[0, \del^2]$.  Letting $\tau_x$ be the time to exit and $v_\ve(x) = E[\tau_x]$ then $v_\ve(x)$ satisfies 
\begin{multline}  \label{AI5}
-\ve(1+x/\ve)^{1/3} v''_\ve(x) - \left[ 1 + \frac 2 3 (1+x/\ve)^{-2/3} \right] v'_\ve(x) = 1, \\
 0 < x < \del^2, \  v_\ve(0)=v_\ve (\del^2) = 0.
\end{multline}
We can write $v_\ve(x)$ as an integral w.r. to the Dirichlet Green's function $G_\ve(x,y)$ for the equation (\ref{AI5}).  Thus we have 
\[  v_\ve(x) = \int^{\del^2}_0 \ G_\ve(x,y)dy, \]
where $G_\ve(x,y)$ is given by the formula,
\be \label{AJ5}
G_\ve(x,y) = \left[ \int^x_0 h_\ve(x')dx'\right] \left[ \int^{\del^2}_y h_\ve(x')dx'\right] \bigg/ \ve(1+ y/\ve)^{1/3}h_\ve(y) \int^{\del^2}_0 h_\ve(x')dx',  
\ee
if $x<y$, and
\be \label{CA5}
G_\ve(x,y) = \left[ \int^{\del^2}_x h_\ve(x')dx'\right] \left[ \int^y_0 h_\ve(x')dx'\right] \bigg/ \ve(1+ y/\ve)^{1/3}h_\ve(y) \int^{\del^2}_0 h_\ve(x')dx',
\ee
if $x>y$. 
In (\ref{AJ5}) and (\ref{CA5}) the function $h_\ve(x)$ is given by (\ref{AI5}).  Observe that upon integration by parts,
\be \label{AK5}
\int^{\del^2}_y h_\ve(x')dx' = \ve(1+ y/\ve)^{1/3}h_\ve(y)  - \ve(1+ \del^2/\ve)^{1/3}h_\ve(\del^2) - \frac 1 3 \int^{\del^2}_y (1+ x'/\ve)^{-2/3}h_\ve(x') dx'.
\ee
Hence we see from (\ref{AJ5})  that $G_\ve(x,y) < 1$ if $x < y$.  Using the fact that $(1 + x/\ve)^{1/3} h_\ve(x)$ is a decreasing function, we similarly see from (\ref{CA5}) that $G_\ve(x,y) < 1$ for $x>y$.  Thus $v_\ve(x) < \del^2$, whence by Chebyshev the second term on the RHS of (\ref{AG5}) is smaller than 
$2\del$.

To obtain a lower bound on $P(\tau_{x,s} < t)$ observe first that 
\begin{multline} \label{AL5}
P\Big( \tau_{x,s} < t\Big) \ge P\left( X_\ve(s') \ {\rm exits} \ [0,\sqrt{\del}] \ {\rm through} \ \sqrt{\del} \  | 
X_\ve(s) = x \right) 
\\
\times P \left( {\rm Time \ to \ exit} \ [0,2\sqrt{\del}] \ {\rm is \ larger \ than} \   \del \ | X_\ve(s)= \sqrt{\del}\right). 
 \end{multline}
Letting $u_\ve(x)$ be the first probability on the RHS of (\ref{AL5}) we have that
\be \label{AM5}
u_\ve(x) \ge \int^x_0 h_\ve(x')dx' \Big/ \int^{\sqrt{\del}}_0 h_\ve(x') \exp\left[ \frac 3 2\left(1+ \frac{x'}{\ve}\right)^{2/3} \left(\frac{x'}{L_0}\right)^{1/3} \right]dx'.
\ee
It is evident that the difference between the RHS of (\ref{EA5}) and the RHS of (\ref{AM5}) converges to zero as $\ve \ra 0$ provided $0 < x < \del^2/2$.  To estimate the second probability on the RHS of (\ref{AL5}) we let $\tau_x$ be the time to exit $[0, 2\sqrt{\del}]$ and put $v_\ve(x) = E[\tau_x]$.  Arguing as we did previously we have that
\be \label{AN5} 
\lim_{\ve \ra 0} v_\ve(\sqrt{\del}) = \int^{2\sqrt{\del}}_{\sqrt{\del}} \ \frac{dx}{1-(x/L_0)^{1/3}} .
\ee
Similarly if we set $w_\ve(x) = E[\tau^2_x]$ then 
\be \label{AO5}
\lim_{\ve \ra 0} w_\ve (\sqrt{\del}) = \left[ \int^{2\sqrt{\del}}_{\sqrt{\del}} \ \frac{dx}{1-(x/L_0)^{1/3}}\right]^2 .
\ee
It follows from (\ref{AN5}), (\ref{AO5}) that the second probability on the RHS of (\ref{AL5}) converges to 1 as $\ve \ra 0$.  We conclude that the first term on the RHS of (\ref{AE5}) is bounded by 4$\del$ as $\ve \ra 0$, uniformly in $x$ for $0 < x < \del^2/2$.

We turn to the second term on the RHS of (\ref{AE5}).  Letting $\tau_x$ be the exit time of $X_\ve(s')$ with $X_\ve(s) = x$ from the interval $[0, 2\sqrt{\del}]$, then if we put
\[	v_\ve(x) = E\left[ \int^{s}_{s-\tau_x} \ \frac{ds'}{X_\ve(s')^{2/3}} \Big| X_\ve(s) = x \right],  \]
one has that $v_\ve(x)$ satisfies
\begin{multline}  \label{AP5}
-\ve(1+x/\ve)^{1/3} v''_\ve(x) - \left[ 1 - \left( \frac x{L_0} \right)^{1/3} + \frac 2 3 (1 + x/\ve)^{-2/3} \right] v'_\ve(x) = 1/x^{2/3},
\\
   0 < x < 2\sqrt{\del}, \ v_\ve(0) = v_\ve(2\sqrt{\del}) = 0.	
\end{multline}
Hence $v_\ve(x)$ is given in terms of the Dirichlet Green's function for (\ref{AP5}) by 
\[	v_\ve(x) = \int^{2\sqrt{\del}}_0 \ G_\ve(x,y) \Big/ y^{2/3} \ dy.	\]
Arguing as we did for the Green's function (\ref{AJ5}) we see that for $\del$ sufficiently small $G_\ve(x,y) \le 2$, whence $v_\ve(x) \le 12\del^{1/6}$.  To bound the second term on the RHS of (\ref{AE5}) we need to add to it paths for which $\tau_x < \del$ and $X_\ve(\tau_x) = 2\sqrt{\del}$.  Such a path can eventually wander back to $x = \sqrt{\del}$ and then we can estimate by $v_\ve(\sqrt{\del})$ the contribution to the integral of $X_\ve(s')^{-2/3}$ until the path again hits $2\sqrt{\del}$.  Thus the second term on the RHS of (\ref{AE5}) is bounded by
\begin{multline} \label{AQ5}
 \frac 1{3L^{1/3}_0}\Big[v_\ve(x) +  v_\ve(\sqrt{\del}) + P(\tau_{\sqrt{\del}} < \del) v_\ve(\sqrt{\del}) \\
+ \left\{ \sum^\infty_{k=2} \ P\big(\tau_{\sqrt{\del}} < \del\big)^k \right\} v_\ve(\sqrt{\del})+ \del \Big/ (\sqrt{\del})^{2/3} \Big] \le K\del^{1/6}, 
\end{multline}
for some constant $K$.  Note the last term on the RHS of (\ref{AP5}) is to take account of the integral of $X_\ve(s')^{-2/3}$ when $X_\ve(s') > \sqrt{\del}$.  We are also using (\ref{AN5}), (\ref{AO5}) to see that $P(\tau_{\sqrt{\del}} < \del) < 1/2$ for small $\ve$.

We consider the final term on the RHS of (\ref{AE5}).  Let $\tau_x$ be the time taken for the diffusion $X_\ve(s'), s'<s$, with $X_\ve(s) = x$ to exit the interval $[0, 2 \sqrt{\del}]$, and let $\tau'_x$ be the corresponding time for $Y_\ve(s'), s'<s$.  Since $X_\ve(s') \le Y_\ve(s')$, $s' \le s$, it follows that for a path $X_\ve(s'), s' \le s$, with $X_\ve(s-\tau_x) = 2\sqrt{\del}$ then $\tau'_x < \tau_x$.  Hence the final term on the RHS of (\ref{AE5}) is bounded by
\begin{eqnarray} \label{AR5}
&\ & E \left[ \frac 2{9\ve} \int^s_{s-\tau_x} \left[ 1+X_\ve(s')/\ve\right]^{-5/3} \; ds' \; ; \; \tau_x > \del, X_\ve(s-\tau_x)=0\right]\\
&+ & E \left[ \frac 2{9\ve} \int^s_{s-\tau_x} \left[ 1+X_\ve(s')/\ve\right]^{-5/3} \; ds' \; ; \;  X_\ve(s-\tau_x)=2\sqrt{\del} \right] \nn\\
&- & E \left[ \frac 2{9\ve} \int^s_{s-\tau'_x} \left[ 1+Y_\ve(s')/\ve\right]^{-5/3} \; ds' \; ; \;  Y_\ve(s-\tau '_x)=2\sqrt{\del} \right] \nn \\
&+ & E \left[ \frac 2{9\ve} \int^s_{t} \left[ 1+X_\ve(s')/\ve\right]^{-5/3} \; ds' \; ; \; \tau_{2\sqrt{\del},s} < t |  X_\ve(s)=2\sqrt{\del} \right] \ . \nn
\end{eqnarray}
Letting $u_\ve(x)$ be the first term in (\ref{AR5}) we see that
\begin{multline} \label{AS5}
-\ve(1+x/\ve)^{1/3} u''_\ve(x) - \left[ 1 - \left( \frac x{L_0} \right)^{1/3} + \frac 2 3 (1+x/\ve)^{-2/3} \right] u'_\ve(x) \\
= \frac 2{9\ve} (1+x/\ve)^{-5/3} P\left( \tau_x > \del, \ X_\ve(s - \tau_x) = 0\right),  \\
0 < x < 2\sqrt{\del}, \ u_\ve(0) = u_\ve(2\sqrt{\del}) = 0. 
\end{multline}
Next we put $v_\ve(x) = E[ \tau_x \; ; \; X_\ve(s-\tau_x) = 0]$.  Then $v_\ve(x)$ satisfies
\begin{multline} \label{AT5}
- \ve(1+x/\ve)^{1/3} v''_\ve (x)- \left[ 1 - \left( \frac x{L_0} \right)^{1/3} + \frac 2 3 (1+x/\ve)^{-2/3} \right] v'_\ve(x) = P\Big( X_\ve(s-\tau_x) = 0\Big),  \\
0 < x < 2\sqrt{\del}, \ v_\ve(0) = v_\ve(2\sqrt{\del}) = 0. 
\end{multline}
We can see from (\ref{EA5}) that $P\big( X_\ve(s-\tau_x) = 0\big) << 1$ if $x > O(\ve)$.  Hence from our estimates on the Green's function (\ref{AJ5}) we conclude from (\ref{AT5}) that $v_\ve(x) \le C\ve$ for $0 < x < 2\sqrt{\del}$, where $C$ is a constant.  Now using the Chebyshev inequality to bound the RHS of (\ref{AS5}) we see that $u_\ve(x) \le C\ve/\del$, $0 < x < 2\sqrt{\del}$.

Next we consider the second and third terms in (\ref{AR5}).  The sum of these two terms is given by $u_\ve(x) - v_\ve(x)$ where $u_\ve(x)$ satisfies
\begin{multline} \label{AU5}
- \ve(1+x/\ve)^{1/3} u''_\ve(x) - \left[ 1 - \left( \frac x{L_0} \right)^{1/3} + \frac 2 3 (1+x/\ve)^{-2/3} \right] u'_\ve(x) \\
 = \frac 2{9\ve} (1 + x/\ve)^{-5/3} P\big( X_\ve(s-\tau_x)  = 2\sqrt{\del}\big), \\
 0 < x < 2\sqrt{\del}, \ u_\ve(0) = u_\ve(2\sqrt{\del}) = 0, 
\end{multline}
and $v_\ve(x)$ satisfies
\begin{multline} \label{AV5}
-\ve(1+x/\ve)^{1/3} v''_\ve(x) - \left[ 1 + \frac 2 3 (1+x/\ve)^{-2/3} \right] v'_\ve(x) \\
= \frac 2{9\ve} (1+x/\ve)^{-5/3} P\left( Y_\ve(s - \tau '_x) = 2\sqrt{\del}\right), \\
 0 < x < 2\sqrt{\del}, \ v_\ve(0) = v_\ve(2\sqrt{\del}) = 0. 
\end{multline}
It is easy to see from (\ref{EA5}) that
\[  | P\big( X_\ve(s-\tau_x)  = 2\sqrt{\del}\big) - P\big( Y_\ve(s-\tau_x) = 2\sqrt{\del}\big) \ | \le C\ve^{1/3}.  \]
We may also easily estimate the difference in the Green's functions (\ref{AJ5}) for (\ref{AU5}) and (\ref{AV5}) to conclude that $|u_\ve(x) - v_\ve(x)| \le C\ve^{1/3}$ for $0 < x < 2\sqrt{\del}$.

We estimate the last term in (\ref{AR5}) in a similar way to how we estimated the second term on the RHS of (\ref{AE5}).  We see that it is bounded by $\ve^{2/3} / \del^{1/3}$.  We have therefore appropriately bounded the third term on the RHS of (\ref{AE5}), whence the result follows.  
\end{proof}
\begin{theorem}  Let $c_0(x), c_0(x,t), c_\ve(x,t)$ be as in Theorem 4.1 and $F(x,t)$ be the mapping defined after (\ref{A4}).  If the function $c_0(\cdot)$ is continuous at $x = F(0, T)$ then (\ref{A5}) holds.
\end{theorem}
\begin{proof} We use the fact that
\[	\frac d{dT} \int^\infty_0 \ c_\ve(x,T)dx = \int^\infty_0 \ \rho_\ve(x,0,T) c_0(x)dx,	\]
where $\rho_\ve$ is as in Lemma 5.3.  We first show that 
\be \label{AW5}
\lim_{\ve \ra 0} \; \int_0^\infty \rho_\ve(x,0,T) c_0(x)dx = c_0(F(0,T)) \lim_{\ve \ra 0} \; \int_0^\infty \rho_\ve(x,0,T)dx.  
\ee
To do this we use the representation,
\be \label{AX5}
\rho_\ve(x,0,T) = E\left[ \rho(X_\ve(t), t, T) \ ; \ \tau_{x,0} > t \  | X_\ve(0) = x \right],
\ee
where $X_\ve(s)$ is a solution to (\ref{E4}) with
 $X_\ve(0) = x$ and $\tau_{x,0}$ is the first hitting time at $0$.  We choose now $t = T-c_1\ve$ so we can bound $\rho_\ve(\cd, t, T)$ from Lemma 5.2.  It follows from Lemma 5.2 that for $x > F(0,T)$ then
\be \label{AY5}
\rho_\ve(x,0,T) \le \frac C \ve \left\{ \exp \left[ -(x/8\ve)^{5/3} \right] + P\left( X_\ve(t) < x/2 \; ; \; \tau_{x,0} > t\right) \right\}.
\ee
It is evident that the first term on the RHS of (\ref{AY5}) is small.  We shall show that the second term is also small provided $x >> 1 +O(T)$.  To do this let $\tau_x$ be the exit time for $X_\ve(s)$ with $X_\ve(0) = x$ from the interval $[x/2, 2x]$ and define $Y_\ve(s), s \ge 0$, by
\[	Y_\ve(s) = \int^{s\wedge \tau_x}_0 \ \sqrt{2\ve} \big(1 + X_\ve(s')/\ve\big)^{1/6} dW(s').	\]
Observe now that for any $\la, M_{\la,\ve}(s)$ defined by 
\[  M_{\la,\ve}(s) = \exp \left[ -\ve\la^2 \int^{s\wedge \tau_x}_0  \big(1 + X_\ve(s')/\ve\big)^{1/3} ds' + \la Y_\ve(s) \right],                 \ \  s\ge 0,   \]
is a Martingale.  Hence for any $a > 0$ there is the inequality,
\begin{eqnarray*}
P \Big( 
\begin{array}[t]{c}
{\displaystyle\sup}\\
 {\scriptstyle 0 \le s \le t}
\end{array} 
Y_\ve(s) > a \Big) &\le& 
P \Big( 
\begin{array}[t]{c}
{\displaystyle\sup}\\
 {\scriptstyle 0 \le s \le t}
\end{array} 
M_{\la,\ve}(s) \ge \exp \left[ -\ve \la^2 (1+2x/\ve)^{1/3}t + \la a \right] \Big) \\
&\le& \exp \left[ \ve \la^2(1+2x/\ve)^{1/3}t - \la a \right].
\end{eqnarray*}
Minimizing the RHS of the last inequality w.r. to $\la$ we conclude that
\[	P \Big( 
\begin{array}[t]{c}
{\displaystyle\sup}\\
 {\scriptstyle 0 \le s \le t}
\end{array} 
|Y_\ve(s)| > a \Big) \le 2\exp \left[ -a^2 / 4\ve(1 + 2x/\ve)^{1/3}) \right].
\]
Let $X^c_\ve(s)$ be the classical solution to (\ref{E4}) with $X^c_\ve(s) = x$, i.e. the solution to the deterministic equation when we eliminate the stochastic term.  Then from Gronwall's inequality we have that
\[  \begin{array}[t]{c}
{\displaystyle\sup}\\
 {\scriptstyle 0 \le s \le t}
\end{array} 
| X_\ve(s \wedge \tau_x) - X^c_\ve(s \wedge \tau_x)| \le
\begin{array}[t]{c}
{\displaystyle\sup}\\
 {\scriptstyle 0 \le s \le t}
\end{array} 
|Y_\ve(s) | \exp[ Ct/x^{2/3}],
\]
for some constant $C$.  If we use the fact that $X^c_\ve(s) \ge x - s$, we may conclude from the last two inequalities that
\[	P\left( X_\ve(t) < x/2 \ ; \; \tau_{x,0} > t \right) \le P(\tau_x < t) \le 2\exp\left[ -x^2 \exp(-Ct^{1/3}) \Big/ 64\ve(1+2x/\ve)^{1/3} \right],	\]
provided $x > 4t$.  Thus we obtain an estimate on the RHS of (\ref{AY5}) which falls off exponentially in $x/\ve$, and we can easily extend this to all $x > F(0,T) + \del$ for any $\del > 0$.

Next we consider the case $0 < x < F(0,T)$, in which case we use the bound $\rho_\ve(x,0,T) \le C P\left(\tau_{x,0} > t \   |X_\ve(0) = x\right) /\ve$ for some constant $C$, which follows from (\ref{AX5}) and Lemma 5.2.  We use the comparison (\ref{H4}) and let $\tau_x$ be the exit time from the interval $[0,\del]$ for the process $Y_\ve(t)$ with $Y_\ve(0) = x$.  Then for a function $f(x)$, if we define $v_\ve(x)$ by
\[	v_\ve(x) = E\left[ \exp \left\{ \int_0^{\tau_x} \ f(Y_\ve(t))dt \right\} | Y_\ve(0) = x \right] \ ,	\]
then $v_\ve(x)$ satisfies
\begin{multline} \label{AZ5}
\ve(1 + x/\ve)^{1/3} v''_\ve(x) - \left[ 1 - (x/L_0)^{1/3} \right] v'_\ve(x) + f(x)v_\ve(x) = 0,\\
 0 < x < \del,  \ v_\ve(0) = v_\ve(\del)=1.
\end{multline}
Observe now that the boundary value problem
\[	\eta u''_\eta(x) - \al u'_\eta(x) + \la u_\eta(x) = 0,  \quad  x > 0,  \ u_\eta(0) = 1,	\]
has solution $u_\eta(x)$ given by the formula
\[	u_\eta(x) = \exp \left[ \frac{\al - \sqrt{\al^2-4\eta \la}}{2\eta} \ x \right], \ \ x > 0,	\]
where we assume $\eta, \al, \la > 0$ and $4\eta \la < \al^2$.  If $\eta \ra 0$ then $u_\eta(x)$ converges to $\exp[\la x/\al]$, corresponding to the fact that the deterministic path starting at $x$ takes time $x/\al$ to hit 0.  It is clear that if we choose $\eta = \ve(1 + \del/\ve)^{1/3}$ and $\al = 1 - (\del/L_0)^{1/3}$ then $v_\ve(x) = u_\eta(x)$ satisfies (\ref{AZ5}) with $f(x) \ge \la$ and $v_\ve(0) = 1$, $v_\ve(\del) > 1$.  We conclude therefore that provided $4\eta \la < \al^2$ there is the inequality,
\[	E\Big[ e^{\la\tau_x}\Big] \le u_\eta(x), \ \ \ 0 < x < \del,		\]
whence 
$$ P(\tau_x > t) \le e^{-\la t} \; u_\eta(x), \ t>0.$$
If we optimize now w.r. to $\la$ we conclude that $P(\tau_x > t)$ satisfies the inequality,
\be \label{EJ5}
P(\tau_x > t) \le \exp \big[ -(t\al - x)^2 / 4\eta t\big], \ \ 0 < x < \al t.
\ee
It follows that if $0 < x < \del$ and $\del$ is sufficiently small, depending on $T$, that $\displaystyle{\lim_{\ve\ra 0}} \; \rho_\ve(x,0,T) = 0$, uniformly in $x$.

We can apply the arguments of the previous two paragraphs to show that for any $\del > 0$ then $\displaystyle{\lim_{\ve\ra 0}} \; \rho_\ve(x,0,T) = 0$ uniformly for $x$ in the interval $0 < x < F(0,T) - \del$.  In fact we have the inequality,
\begin{multline} \label{BA5}
P\left( \tau_{x,0} > t  \ | X_\ve(0) = x \right) \le P \Big(
\begin{array}[t]{c}
{\displaystyle\sup}\\
 {\scriptstyle 0 \le s \le t - \eta} 
\end{array}
| X_\ve(s \wedge \tau_x) 
   - X^c_\ve(s \wedge \tau_x) |> \rho \del/2 \Big) \\
+ 
\begin{array}[t]{c}
{\displaystyle\sup}\\
 {\scriptstyle 0 \le y <(1+\rho)\del/2}
\end{array} 
P(\tau_y > \eta),   
\end{multline}
where $\rho,\eta$ are small positive constants.  The exit time $\tau_y$ in the second term on the RHS of (\ref{BA5}) refers to the exit time from $[0,\del]$ for $Y_\ve(s)$ with $Y_\ve(0) = y$.  The exit time $\tau_x$ in the first term is the time taken for $X_\ve(s)$ with $X_\ve(0) = x$ to exit the interval $[\del/2, a]$, where $a = \sup\{x(s) : 0\le s \le T\}$ and $x(s)$ is the solution to (\ref{A4}) with $x(T) = 0$.  For the solution to (\ref{A4}) with $x(0) = x < F(0,T)-\del$, let $\eta$ be defined by $x(T - \eta) = (1-\rho)\del/2$ for some small $\rho > 0$ to be determined.  Then $X^c_\ve(t-\eta) \le \del/2$ for small $\ve$, whence (\ref{BA5}) holds for sufficiently small $\ve$.  We can choose now $\rho > 0$ small enough so that $\eta \ge (1 + 2\rho)\del/2$.  In that case both terms on the RHS of (\ref{BA5}) go to zero exponentially in $1 /\ve^{2/3}$ as $\ve \ra 0$.  We have therefore proved (\ref{AW5}).

We complete the proof of the theorem by showing that
\be \label{BB5}
\lim_{\ve \ra 0} \int^\infty_0 \rho_\ve(x,0,T)dx = \exp \left[ - \; \frac 1 3 \int^{T}_0 \ \frac{ds}{x(s)^{2/3}L(s)^{1/3}}\right],
\ee
where $x(s)$ is the solution to (\ref{A4}) with $x(T) = 0$.  To do this we write as in (\ref{R5})
\[
\int^\infty_0 \rho_\ve(x,0,T)dx = \int^\infty_0 \rho_\ve(x,T-\del, T) c_\ve(x, T-\del)dx .
\]
Since by our previous argument $\rho_\ve(x, T-\del, T)$ is concentrated at $x \sim \del$, (\ref{BB5}) will follow from Lemma 5.3 if we can show that
\[
\lim_{\ve \ra 0} c_\ve(x,T-\del) = \exp \left[ - \; \frac 1 3 \int^{T-\del}_0 \ \frac{ds}{x(s)^{1/3}L(s)^{1/3}}\right],
\]
where $x(s)$ is the solution to (\ref{A4}) with $x(T-\del) = x \sim \del$.  This follows as before using the representations (\ref{S5}), (\ref{AF5}) for $c_\ve(x,s)$ and comparing solutions of the stochastic equation to solutions of the corresponding deterministic equation. 
\end{proof}

\section{Discussion}
This paper is primarily concerned with the study of a diffusive version of the LSW model and the convergence of its solutions on any finite time interval as the coefficient of diffusion  $\ve$ goes to $0$ to solutions of the classical LSW model. This is a first step towards showing that solutions of the diffusive LSW equation behave similarly to solutions of the LSW equation for arbitrarily large time. In particular, one expects that solutions of the diffusive LSW equation coarsen in a fashion similar  to solutions of LSW. Thus the average cluster volume  is expected to increase linearly in time at large times, and this should be uniform in $\ve$ for small $\ve$. In Theorem 3.2 we obtained a uniform time averaged upper bound on the rate of coarsening. 

Although the effect of the diffusion decreases with time, discussion in the physics literature suggests that it plays an important role in the asymptotic behavior of the solution by acting as a {\it selection principle}. Thus at moderate times the diffusion produces a Gaussian tail to initial data (with compact support for example). It is the tail which determines asymptotic behavior, singling out the  infinitely differentiable self-similar solution of the LSW model.

The paper begins with a derivation of the diffusive LSW model from the Becker-D\"{o}ring model. The least justified step in the derivation is the imposition of a Dirichlet boundary condition, which cuts off interaction between clusters of large volume and clusters of $O(1)$ volume that are almost in equilibrium at large time. The boundary condition is imposed at a particular cluster volume. One can see that this is a significant simplification since it implies that after the condition is imposed,  the monomer density is for all subsequent times strictly larger than its equilibrium value. Despite this issue, the derivation of the diffusive LSW model from the BD model does offer a strategy for attempting to understand the mechanism of coarsening in the BD model. In order to carry it through one will have to understand in a precise way how a boundary condition is imposed on the diffusive LSW model by the BD dynamics.

There has been some previous literature showing in an almost mathematically rigorous way a connection between solutions of the LSW and BD models. The key assumption required in this work is a type of upper bound on the coarsening rate for the BD model. It is not clear how this is related to the boundary condition assumption discussed in the previous paragraph. One should note however that a positive lower bound on the monomer density minus its equilibrium value, which is a consequence of the Dirichlet boundary condition, is a type of upper bound on the coarsening rate. 

\bigskip

\thanks{ {\bf Acknowledgement:} The author would like to thank Peter Smereka and Barbara Niethammer for many helpful conversations. This research was partially supported by NSF
under grants DMS-0500608 and DMS-0553487.

\end{document}